\documentclass[openany, 12pt]{memoir}
\usepackage{lingmacros}
\usepackage{tree-dvips}
\usepackage{amssymb,amsfonts,amsthm,amsmath}

\usepackage{mathrsfs}
\usepackage{subcaption}
\usepackage[english]{babel}
\usepackage[all,cmtip]{xy}
\usepackage{tikz}
\usepackage{tikz-cd}
\usepackage{comment}
\usepackage{mathtools}

\usepackage{tensor}
\usepackage[breaklinks]{hyperref}
\usepackage{float}
\usetikzlibrary{arrows,chains,matrix,positioning,scopes}
\usepackage[margin = 1in]{geometry}
\usepackage{microtype}
\usepackage [autostyle, english = american]{csquotes}
\usepackage{amsmath,amsthm,amssymb}
\usepackage{enumitem}
\DeclareFontFamily{OT1}{pzc}{}
\DeclareFontShape{OT1}{pzc}{m}{it}{<-> s * [1.10] pzcmi7t}{}
\DeclareMathAlphabet{\mathpzc}{OT1}{pzc}{m}{it}
\newcommand\restr[2]{{% we make the whole thing an ordinary symbol
  \left.\kern-\nulldelimiterspace % automatically resize the bar with \right
  #1 % the function
  \vphantom{\big|} % pretend it's a little taller at normal size
  \right|_{#2} % this is the delimiter
  }}

\renewcommand{\d}[1]{\ensuremath{\operatorname{d}\!{#1}}}

\newtheorem{theorem}{Theorem}
[section]

\theoremstyle{definition}
\newtheorem{defn}{Definition}[section]

\theoremstyle{plain}
\newtheorem{corollary}[theorem]{Corollary}
\newtheorem{lemma}[theorem]{Lemma}

\newtheorem{proposition}[theorem]{Proposition}

\theoremstyle{remark}
\newtheorem{remark}[theorem]{Remark}

\theoremstyle{definition}

\usepackage[
backend=bibtex,
style=alphabetic,
sorting=anyvt,
natbib=true
]{biblatex}
\makeatletter
\tikzset{join/.code=\tikzset{after node path={%
\ifx\tikzchainprevious\pgfutil@empty\else(\tikzchainprevious)%
edge[every join]#1(\tikzchaincurrent)\fi}}}
\makeatother
\tikzset{>=stealth',every on chain/.append style={join},
         every join/.style={->}}

\addto\captionsenglish{% Replace "english" with the language you use
  \renewcommand{\contentsname}%
    {Table of Contents}%
}

\numberwithin{equation}{section}

\begin{document}
\setsecnumdepth{subsection}
\setsecnumdepth{subsubsection}
%%%%%%%% Title page
\pagenumbering{gobble}
\title{\textbf{Training Neural Networks Using Reproducing Kernel Space Interpolation and Model Reduction}}
\vspace*{3\baselineskip} 
\centerline{\textbf{Training Neural Networks Using Reproducing Kernel Space Interpolation and Model Reduction}}
\vspace*{1\baselineskip}
\centerline{A Dissertation presented} 
\vspace*{1\baselineskip}
\centerline{by}
\vspace*{1\baselineskip} 
\centerline{\textbf{Eric Arthur Werneburg}} 
\vspace*{1\baselineskip}
\centerline{to}
\vspace*{1\baselineskip}
\centerline{The Graduate School} 
\vspace*{1\baselineskip}
\centerline{in Partial Fulfillment of the} 
\vspace*{1\baselineskip} 
\centerline{Requirements} 
\vspace*{1\baselineskip}
\centerline{for the Degree of} 
\vspace*{1\baselineskip} 
\centerline{\textbf{Doctor of Philosophy}} 
\vspace*{1\baselineskip}
\centerline{in}
\vspace*{1\baselineskip} 
\centerline{\textbf{Applied Mathematics and Statistics}}
\vspace*{2\baselineskip}
\centerline{Stony Brook University} 
\vspace*{2\baselineskip} 
\centerline{\textbf{August, 2022}}

%%%%%%%
\newpage
\pagenumbering{roman}
\setcounter{page}{2}
\centerline{\textbf{Stony Brook University}} \vspace*{1\baselineskip}
\centerline{The Graduate School} 
\vspace*{1\baselineskip}
\centerline{\textbf{Eric Arthur Werneburg}} 
\vspace*{1.5\baselineskip}
\centerline{We, the dissertation committee for the above candidate for
the}
\vspace*{1\baselineskip}
\centerline{Doctor of Philosophy degree, hereby recommend}
\vspace*{1\baselineskip}
\centerline{acceptance of this dissertation} \vspace*{2\baselineskip}
\centerline{\textbf{Andrew Peter Mullhaupt}} 
\centerline{\textbf{Research Professor, Applied Mathematics and Statistics}} 
\vspace*{1.\baselineskip}
\centerline{\textbf{Pawel Polak}} 
\centerline{\textbf{Assistant Professor, Applied Mathematics and Statistics}}
\vspace*{1.\baselineskip}
\centerline{\textbf{Stanislav Uryasev}}
\centerline{\textbf{Professor and  Frey Family Endowed Chair, Applied Mathematics and Statistics}}
\vspace*{1.\baselineskip}
\centerline{\textbf{Ming Gu}}
\centerline{\textbf{Professor of Mathematics, University of California, Berkeley}}
\vspace*{1.\baselineskip}
\centerline{\textbf{Raphael Douady}}
\centerline{\textbf{Research Professor, University of Paris l-Sorbonne}}
\vspace*{1.\baselineskip}
\centerline{This dissertation is accepted by the Graduate School}
\vspace*{1\baselineskip}
% Change this according to Dean
\centerline{Eric Wertheimer}
\vspace*{1\baselineskip}
\centerline{Dean of the Graduate School}

%%%%%%%%%%%% Abstract
\newpage
\centerline{Abstract of the Dissertation} \vspace*{1\baselineskip} \centerline{\textbf{Training Neural Networks Using Reproducing Kernel Space Interpolation and Model Reduction}} \vspace*{1\baselineskip}
\centerline{by} 
\vspace*{1\baselineskip} 
\centerline{\textbf{Eric Arthur Werneburg}} 
\vspace*{1\baselineskip} 
\centerline{\textbf{Doctor of Philosophy}} 
\vspace*{1\baselineskip} 
\centerline{in} 
\vspace*{1\baselineskip} 
\centerline{\textbf{Applied Mathematics and Statistics}} 
\vspace*{1\baselineskip}
\centerline{Stony Brook University} 
\vspace*{1\baselineskip} 
\centerline{\textbf{2022}} 
\vspace*{2\baselineskip}
We introduce and study the theory of training neural networks using interpolation techniques from reproducing kernel Hilbert space theory. We generalize the method to Krein spaces, and show that widely-used neural network architectures are subsets of reproducing kernel Krein spaces (RKKS). We study the concept of \enquote{associated Hilbert spaces} of RKKS and develop techniques to improve upon the expressivity of various activation functions. Next, using concepts from the theory of functions of several complex variables, we prove a computationally applicable, multidimensional generalization of the celebrated Adamjan-Arov-Krein (AAK) theorem. The theorem yields a novel class of neural networks, called Prolongation Neural Networks (PNN). We demonstrate that, by applying the multidimensional AAK theorem to gain a PNN, one can gain performance superior to both our interpolatory methods and current state-of-the-art methods in noisy environments. We provide useful illustrations of our methods in practice.

%%%%%%%%%%%%% Dedication
%\newpage
%\centerline{\textbf{Dedication Page}}
%\vspace*{4\baselineskip}
%This page is optional.

%%%%%% Table of Contents
\newpage
\begin{KeepFromToc}
  \tableofcontents
\end{KeepFromToc}

\newpage
\centerline{\textbf{Acknowledgements}}
\addcontentsline{toc}{chapter}{\numberline{}Acknowledgements}
\vspace*{2\baselineskip}

First, I extend my utmost thanks to Professor Andrew Mullhaupt for being a supreme role model and mentor. He has shared with me an unfathomable wealth of knowledge, suggested a captivating line of research, and taught me the necessary tools to \enquote{learn quickly}. I also extend my thanks to Zari Rachev and James Glimm for teaching me the foundations of classical quantitative finance. I thank Sister Jeannine Toppin, whose calculating ability was a mathematical inspiration early in my career. I thank Pawel Polak, Raphael Douady, Stanislav Uryasev, Ming Gu, Yifan Sun, John Pinezich, Jiazhou Wang, Chi Kong, and Glenn Thomas Werneburg for helping me strengthen my work. I thank Laurie Dalessio, Christine Rota, and Cathy Arrighetta for their logistical help throughout my time at Stony Brook.

I thank AlphaCrest for granting me an intern position which has played a pivotal role in my education as an applied mathematician. In particular, I thank Xiao Yu for teaching me how to write strong code. 

I wish to extend my special thanks to my mother, father, and brother for their unrelenting love and support.

I thank Dana Michele Castro, as well as my friends Stephen Erickson, Michael Erickson, James Rula, Tara Burns, Kevin Talbot, Danielle Aliotta, Andy LaBella, Alyssa Korpi, and Abraham Rabinowitz for all their support.

Finally, I wish to thank God for the opportunity and talent to pursue mathematics, and for everything else.

\newpage
\pagenumbering{arabic}

\begingroup\DoubleSpacing

\chapter{Overview}

Training neural networks has overwhelmingly been performed via descent methods. These methods include, but are not limited to stochastic gradient descent (SGD), root mean square propagation (RMSProp) \cite{rmsprop}, adaptive gradient method (Adagrad) \cite{adagrad}, natural gradient descent (NGD) \cite{amari1998natural}, and adaptive moment estimation method (Adam) \cite{adam}. Such methods have found some success, particularly in natural language processing and imaging (see \cite{gpt3} and \cite{vgg16} for particularly impressive results). On the other hand, there are notable drawbacks to such training approaches, for which there is no obvious remedy. 

Neural networks often come with a forward-backward error discrepancy, which descent methods cannot overcome. Particularly, the authors in \cite{firenets} prove the existence of neural networks which cannot be trained to the correct parametrization by any algorithm, randomized or deterministic. In other words, error in the \textit{output} often does not correspond with error in the \textit{parameters} of the network. These adversarial examples were constructed using a sophisticated diagonalization argument, attacking the halting property of neural network training methods. Backward-unstable problems such as these are particularly difficult to solve using gradient descent (see \cite{trefethen1997numerical} for an exposure to the subject of forward-backward error). The one-pixel attacks of \cite{onepixelattacks} exemplify the chaotic nature of training neural networks via descent methods. \cite{ljung} gives a rigorous treatment of this behavior (albeit not in the context of imaging) and generalizes it to the broad class of recursive stochastic algorithms.

The asymptotic behavior of a descent-trained neural network is generally unknown; in fact, such neural networks may not even reach a global optimum \cite{zhang2018bpgrad}. Moreover, in the special cases where asymptotics are known, one is often restricted to a small class of activation functions (e.g. ReLU \cite{hajek2019ece}, Section $8.5$) or Lipschitz continuous optimization \cite{zhang2018bpgrad}. This problem is not unique to neural network theory; the field of signal processing once shared a similar problem. In particular, the asymptotic behavior of least mean square (LMS) algorithms, invented in 1960 in \cite{widrow1960adaptive}, was unknown until the 1990s when Hassibi proved that LMS algorithms are $H^{\infty}$-optimal \cite{hassibi1996h}. While it is possible that similar results hold for neural networks, we are unaware of any rigorous treatment presented for the general case.

Finally, descent methods are notoriously inefficient in both time and memory. As a result, there is a growing environmental concern, in that the energy requirement for training neural networks via descent is unsustainable. Concretely, a study done in 2019 by \cite{strubell2019energy} shows that training a relatively large neural network can burn the same amount of carbon in one week as five American cars in their lifetime. The neural network used in this calculation was a \textit{transformer} called GPT-2, which was used for natural language processing (NLP) \cite{gpt2}. A common argument for using such an energy intensive training method is \enquote{You only have to train it once, and never again [leaving room for fine-tuning some parameters].} Unfortunately, this argument fails because there, to date, is no catch-all neural network for NLP, as can be shown empirically by the creation of GPT-3 in 2020 \cite{gpt3}.

The training method developed here is inherently different from its predecessors, because it is performed via interpolation rather than descent. In particular, the interpolation technique merely requires solving a highly structured system of linear equations. This makes the resulting neural network cheaper and faster to train than conventional networks. Moreover, the resulting neural network is well-behaved in the sense that it is of minimal norm when the activation function is induced by the kernel of a reproducing kernel Hilbert space (RKHS). When the activation function is induced by a Krein kernel, the neural network is viewed as the minimizer of a regularized minimization, or loss, problem.

% \begin{enumerate}
%     \item Training is performed via interpolation, rather than descent.
%     \item There is no need to see training data more than once.
%     \item The neural network is "well-behaved" in the sense that it is of minimal norm when the activation function is induced by the kernel of a RKHS. When the activation function is induced by a Krein kernel, the neural network is the orthogonal projection of an interpolating function onto the span of reproducing Krein kernels at the given interpolation nodes.
%     %\item Using the theory of RKHS, one can develop a two-layer neural network of minimal width, given a threshold for in-sample error.
%     \item The activation function studied in this work is optimal in terms of its expressivity.
%     %\item Our training method achieves out-of-sample performance superior to and faster than conventional methods.
% \end{enumerate}

% Point $1$ above establishes that our training method controls the backward error on in-sample performance. Point $2$ implies that our training method is cheaper and faster than descent methods. Point $3$ shows rigorously how the neural network behaves asymptotically. Point $4$ implies that the neural network architecture introduced here is better than any other ones used in practice. 

Our training method transplants the well-established theory of RKHS into the framework of neural network theory. We develop novel tools which allow us to study the expressivity of various, attractive neural network architectures. Our theory is general enough to include a rigorous treatment of popular activations, including softplus, hyperbolic tangent, ReLU, Leaky ReLU, GELU, and many others. We leverage and generalize the concept of contractive containment to Krein spaces in order to rigorously justify improvements on these.

The main drawback of interpolation is that, in noisy environments, noise can be interpreted as signal and significantly affect performance. In such situations, one often prefers to perform model reduction, thereby reducing the intrinsic complexity of the model and reducing the amount of noise captured. Classical model reduction techniques such as Adamjan-Arov-Krein (AAK) theory are not suited to the framework of neural network theory, because they are applicable only in the single complex variable case \cite{adamyan1971analytic}. While there has been much progress in generalizing the AAK theory to several complex variables (SCV), namely by Cotlar and Sadosky in \cite{cotlar1994nehari} and \cite{cotlar1996two}, these methods are, as they stand, computationally intractable.

To address this, we introduce and prove a practical, SCV analogue of the AAK theorem, called the \textit{marginal AAK theorem} (mAAK theorem). The mAAK theorem naturally extends to the concept of a \textit{Prolongation Neural Network} (PNN), a fast, interpolation-type neural network which does not suffer from noise-modelling.

Chapters \ref{ch: nn} and \ref{ch: rkhs} respectively introduce the necessary background in neural network theory and reproducing kernel Hilbert spaces. Chapter \ref{ch: rkks} generalizes the theory of RKHS to Krein spaces. Chapter \ref{ch: marriage} forms a bridge between neural network theory and RKKS to form a novel theory for efficiently training neural networks. Of particular importance is the use of the Agler-McCarthy Theorem, which allows one to extend our theory to deep learning. We also point out that many activation functions in practice today generate neural networks which are subsets of our constructed reproducing kernel Krein spaces. 

In Chapter \ref{ch: cs}, we provide background to the SCV AAK theory of Cotlar and Sadosky (\cite{cotlar1994nehari}, \cite{cotlar1996two}). We then introduce and prove the marginal AAK theorem which gives numerically tractable approximation bounds on model reduction. Drawing from this, we introduce the concept of a Prolongation Neural Network, a model-reduced neural network which achieves the bounds given by the mAAK theorem. 

In Chapter \ref{ch: apps}, we apply our theory to several nontrivial problems, comparing our methods against current, state-of-the-art training methods. In particular, we train our neural networks to calculate matrix permanents, complete elliptic integrals of the second kind, and shrinkage techniques from the theory of statistics. We also note that this work provides a complete and rigorous treatment of the theory of neural networks found in \cite{werneburgmachine}.

\chapter{Neural Network Preliminaries}\label{ch: nn}

We begin with some definitions and notation.

\begin{defn}\label{def:neuralnetwork}
A \textbf{feedforward neural network}, $F: \mathbb{C}^n \rightarrow \mathbb{C}^m$ is a finite, alternating composition of linear transformations and nonlinear functions:

\begin{equation}\label{eq:neuralnetwork}
    F(x) = f_{N}(A_{N}f_{N - 1}(A_{N - 1} \dots f_{1}(A_1 x))),
\end{equation}
where $f_N$ may be the identity. $\{f_j\}$ are called \textbf{activation functions} and $\{A_j\}$ are called \textbf{weight matrices}. Given \eqref{eq:neuralnetwork}, we shall call

\begin{equation}
    F_k(x) := f_{k}(A_{k}f_{k - 1}(A_{k - 1} \dots f_{1}(A_1 x)))
\end{equation}
the \textbf{$k^{th}$ layer of $F$ given $x$}, or simply the \textbf{$k^{th}$ layer of $F$}. Given $A_j \in \mathbb{C}^{m_j \times n_j}$, we say the $j^{th}$ layer has \textbf{width} $m_j$. We say $F$ has \textbf{width} $M$ if $M$ is the maximum width over all the layers of $F$.

\end{defn}

\begin{defn}
A \textbf{neural network with skip connection from layer $j$ to layer $k$} ($j < k$) is a function of the form

\begin{equation}
    F(x) = f_{N}(A_{N}f_{N - 1}(A_{N - 1} \dots f_{k + 1}(A_{k + 1}f_{k}(BF_j(x), A_{k}f_{k - 1}(A_{k - 1} \dots f_{1}(A_1 x)))))),
\end{equation}
where $B \in  \mathbb{C}^{b \times n_{j + 1}}$ for the $b \in \mathbb{N}$ appropriate with respect to $f_k$ and $A_{k + 1}$. Here, the matrix $B$ is called a \textbf{skip connection}. More generally, a \textbf{neural network with skip connections} is a function of the form

\begin{multline}\label{eq:skipconn}
    F(x) = f_{N}(B_{N - 1}^{(N)}F_{N - 1}(x), B_{N - 2}^{(N)}F_{N - 2}(x), \dots, \\
    B_{1}^{(N)}F_1(x), B_{0}^{(N)}x, A_{N}f_{N - 1}(B_{N - 2}^{(N - 1)}F_{N - 2}(x), B_{N - 3}^{(N - 1)}F_{N - 3}(x), \dots, \\
    B_{1}^{(N - 1)}F_1(x), B_{0}^{(N - 1)}x, A_{N - 1} \dots f_{1}(B_{0}^{(1)}x, A_1 x))),
\end{multline}
where each $B_j^{(k)}$ has appropriate dimension. Again the collection $\{B_j^{(k)}\}$ are called \textbf{skip connections}.
\end{defn}

Moving forward, we shall denote neural networks by capital letters $F, G, H, \dots$ and their corresponding activation functions by lower case letters $f, g, h, \dots$. Different authors may define what we say is an $N-$layer neural network as an $(N - 1)-$layer neural network. An easy way to remember our definition is that the number of layers corresponds precisely to the number of weight matrices.

\begin{defn}
Given a neural network $F$ of the form
\begin{equation}
    F(x) = f_{N}(A_{N}f_{N - 1}(A_{N - 1} \dots f_{1}(A_1 x))),
\end{equation}
where $A_j \in \mathbb{C}^{m_j \times n_j}$, we define the \textbf{architecture of $F$} as the tuple 

\begin{equation}
    (f_1, f_2, \dots f_N, m_1, n_1, m_2, n_2, \dots, m_N, n_N)
\end{equation}
\end{defn}

\begin{defn}
Given a neural network $F$ with width $N$ and activation functions $\{f_j\}_{j=1}^{N}$, the space of functions defined by the completion of the span of all neural networks with width $N$ and activation functions $\{f_j\}_{j=1}^{N}$ is called the \textbf{expressivity} of $F$.
\end{defn}

While much work has been put into finding \enquote{good} architectures for neural networks, there have been few advances in regard to choosing activation functions. The \textit{universal approximation theorems} (UAT) in the field take a large part in this. In particular, Cybenko's work on sigmoidal functions \cite{cybenko1989approximation} and Perekrestenko's work on the rectified linear unit (ReLU) \cite{perekrestenko2018universal} have been pivotal in the popularity of these as activation functions in neural networks.%The following theorem is a classic example and is one of the most well-known UAT.

%\begin{theorem}[\cite{cybenko1989approximation}]\label{thm:cybenko}
%Let $f: \mathbb{R} \rightarrow \mathbb{R}$ be such that

%\begin{equation}
%    f(t) = \begin{cases}
%    0 \qquad t \rightarrow -\infty \\
%    1 \qquad t \rightarrow \infty
%    \end{cases}
%\end{equation}
%Then functions of the form

%\begin{equation}
%    G(x) = \sum\limits_{j=1}^N \alpha_j f(y_j^Tx + \beta_j)
%\end{equation}
%are dense in the space of continuous functions on the unit hypercube $[0, 1]^n$ under the supremum norm $||\cdot||_{\infty}$.
%\end{theorem}

%\begin{proof}
%See \cite{cybenko1989approximation}.
%\end{proof}

%Theorem \ref{thm:cybenko} tells us that, given enough width, a $2-$layer neural network with a sigmoidal activation function can approximate a continuous function on the unit hypercube arbitrarily well. Unfortunately, as with all other UAT of its kind, Theorem \ref{thm:cybenko} does not tell us how to find such a network. 

%There are similar UAT for specific activation functions. Most notably, \cite{perekrestenko2018universal}, proved that ReLU networks have a universal approximation property. This, along with the (supposed) ease with which calculation can be done with ReLU, has led to an explosion of neural networks possessing ReLU as their sole activation function.

In Chapters \ref{ch: rkhs} and \ref{ch: rkks} we will construct many functions which not only have a universal approximation property, but whose expressivity is strictly superior to ReLU and various sigmoidal functions currently being used in practice. 

\chapter{Reproducing Kernel Hilbert Space Preliminaries}\label{ch: rkhs}
We review some basic notions from the theory of reproducing kernel Hilbert spaces (RKHS). We show a classical interpolation result from RKHS and use it to construct a learning algorithm for neural networks with activation functions derived from kernel functions.

\section{Introduction}
While reproducing kernel Hilbert spaces over more general fields can be studied, we will confine ourselves to looking at the complex numbers $\mathbb{C}$.

Recall a Hilbert space is a complete, inner product space.

\begin{defn}
Given a set $\mathcal{X}$, a reproducing kernel Hilbert space (RKHS) $\mathcal{H}$ is a Hilbert space with inner product $\langle \cdot, \cdot \rangle$, which is a subspace of the set of all functions $\mathcal{F}(\mathcal{X}, \mathbb{C})$ from $\mathcal{X}$ to $\mathbb{C}$, such that for every $y \in \mathcal{X}$, the evaluation functional $E_y: \mathcal{H} \rightarrow \mathbb{C}$, $E_y(f) = f(y)$ is bounded. 
\end{defn}

\begin{theorem}[Riesz Representation Theorem]
Let $\mathcal{H}$ be a complex Hilbert space over $\mathcal{X}$ with inner product $\langle \cdot, \cdot \rangle$. Then for every continuous linear functional $L: \mathcal{H} \rightarrow \mathbb{C}$ on $\mathcal{H}$, there exists a unique $\nu \in \mathcal{X}$ such that $L_{\nu} = L$, where $L_{\nu}(\omega) = \langle \omega, \nu \rangle$.
\end{theorem}

% \begin{proof}
% We follow a proof given in \cite{tao2010epsilon}. If $L$ is identically equal to zero, the proof is trivial. Therefore, suppose $L$ is not identically zero. Then, the kernel of $L$ is a closed, proper subspace of $\mathcal{H}$ and is therefore not dense in $\mathcal{H}$. (As a side note, if $L$ were not continuous, then the kernel would not be closed and the theorem may fail). Let $V = \mathrm{ker}L$. Then $V^{\perp}$ is nontrivial. Let $\nu \in V^{\perp}$ so that $L(\nu) \neq 0$. Moreover, normalize $\nu$ so that $\langle \nu, \nu \rangle = 1$. Then 

% \begin{equation}
%     x - \frac{L(x)}{L(\nu)}\nu \in V
% \end{equation}

% Taking the inner product

% \begin{equation}
%     \langle x, \nu \rangle - \frac{L(x)}{L(\nu)}\langle \nu, \nu \rangle = 0,
% \end{equation}

% we see that $L(x) = \langle x, \overline{L(\nu)} \nu \rangle$.

% To prove uniqueness, simply note that if $L(x) = \langle x, \nu \rangle = \langle x, \mu \rangle$, then $\langle x, \nu - \mu \rangle = 0$ for all $x \in \mathcal{H}$. Specifically, $\langle \nu - \mu, \nu - \mu \rangle = 0$, so $\nu = \mu$.
% \end{proof}

By the Riesz Representation Theorem, we have for every $y \in \mathcal{X}$ the existence of a unique $k_y \in \mathcal{H}$ such that $\langle f, k_y \rangle = f(y) \ \forall f \in \mathcal{H}$. The function $k_y$ is called the \textbf{reproducing kernel for the point $\mathbf{y}$}.

It will be helpful to introduce the following notation:

\begin{equation}
    K(x, y) = k_y(x)
\end{equation}

$K$ is called the \textbf{reproducing kernel for} $\mathcal{H}$.

The following lemma is of fundamental importance, because it shows that when we do have a RKHS, we can approximate all its elements arbitrarily well using reproducing kernels at points $y$.

\begin{lemma}
Given a RKHS $\mathcal{H}$ over $\mathcal{X}$ with reproducing kernel $K(\cdot, y) = k_y(\cdot)$, the space of functions $\mathrm{span}\{k_y \ | \ y \in \mathcal{X}\}$ are dense in $\mathcal{H}$.
\end{lemma}

\begin{proof}
$\mathrm{span}\{k_y\}$ are \textit{not dense} in $\mathcal{H}$ only if there exists a function $f \in \mathcal{H}$ which is not everywhere zero, such that $f$ is orthogonal to all functions $k_y$. This means 

\begin{equation}
    0 = \langle f, k_y \rangle = f(y) \qquad \forall y \in \mathcal{X}
\end{equation}

This implies $f \equiv 0$, thus yielding a contradiction.
\end{proof}

\begin{defn}
Given a set $\mathcal{X}$, a function $K: \mathcal{X} \times \mathcal{X} \rightarrow \mathbb{C}$ is a \textbf{kernel function} if for any set of points $x_1, \dots, x_N \in \mathcal{X}$, the matrix $P \in \mathbb{C}^{N \times N}$, whose $i, j$ element is defined as $K(x_i, x_j)$ is nonnegative semidefinite, denoted $P \succeq 0$.
\end{defn}

The next two theorems show there is a bijection between kernel functions on a set and RKHS on the same set.

\begin{theorem}
Given a RKHS $\mathcal{H}$ on $\mathcal{X}$ with reproducing kernel $K$, $K$ is a kernel function.
\end{theorem}

\begin{proof}
Consider $x_1, \dots, x_N \in \mathcal{X}$, for some arbitrary $N \in \mathbb{N}$. Let $P \succeq 0$ denote the matrix whose $i, j$ element is $K(x_i, x_j)$. Because $P \succeq 0$ if and only if $\sum\limits_{i, j=1}^N y_iy_jP_{ij} \geq 0$ for all $y_1, \dots y_N \in \mathbb{C}$, we have

\begin{equation}
    \sum\limits_{i, j=1}^N \bar{y_i}y_jP_{ij} = \sum\limits_{i, j=1}^N \bar{y_i}y_jK(x_i, x_j) = \sum\limits_{i, j = 1}^N \bar{y_i} y_j\langle k_{x_j}, k_{x_i} \rangle = ||\sum\limits_{i=1}^Ny_ik_{x_i}||^2 \geq 0,
\end{equation}

where $||\cdot||$ is the norm induced by the inner product $\langle, \rangle$ of the Hilbert space $\mathcal{H}$.
\end{proof}

Note that in general the reproducing kernel of a RKHS is strictly positive definite; indeed, if $P$ were not strictly positive definite, then there would exist a finite set of points $x_1, \dots x_M$ such that \textit{every} function $f \in \mathcal{H}$ would have a linear dependence on them. In other words, there would exist $\alpha_1, \alpha_2, \dots, \alpha_M$, such that

\begin{equation}
    \langle f, \sum\limits_{i=1}^M\alpha_i k_{x_i} \rangle = 0,
\end{equation}
for all $f \in \mathcal{H}$.

We will see later on that these spaces will have drawbacks in practice. The next theorem forms the backbone of our RKHS theory which will be applied to training neural networks.

\begin{theorem}[E. H. Moore]
Let $\mathcal{X}$ be a set, and let $K: \mathcal{X} \times \mathcal{X} \rightarrow \mathbb{C}$ be a kernel function. Then there exists a RKHS on $\mathcal{X}$, denoted $\mathcal{H}$ such that $K$ is the reproducing kernel for $\mathcal{H}$.
\end{theorem}

\begin{proof}
Here, we follow \cite{paulsen2016introduction}. As usual, let $k_y(x) = K(x, y)$, for each $y \in \mathcal{X}$ and let $\mathcal{W} \subseteq \mathcal{F}(\mathcal{X}, \mathbb{C})$ be the span of the functions $\{k_y\}$.

We proceed by first constructing an inner product $B(\cdot, \cdot)$ on $\mathcal{W}$. Then we shall take the completion of $\mathcal{W}$ and show that it includes only functions on $\mathcal{X}$. 

Consider the function $B: \mathcal{W} \times \mathcal{W} \rightarrow \mathbb{C}$, 

\begin{equation}
    B(\sum\limits_{i}\alpha_i k_{x_i}, \sum\limits_{j}\beta_j k_{x_j}) = \sum\limits_{i,j}\bar{\beta_j}\alpha_i K(x_j, x_i)
\end{equation}
    
$B$ is trivially sesquilinear, so we need only show it is well-defined. To that end, suppose $f = \sum\limits_i \alpha_i k_{x_i} \equiv 0$. Then, 

\begin{equation}
    B(f, k_{x_j}) = \sum\limits_{k} \alpha_k K(x_j, x_k) = f(x_j) = 0
\end{equation}

Similarly, $B(k_{x_j}, f) = 0$. On the other hand, if $B(f, \omega) = 0$ for all $\omega \in \mathcal{W}$, then in particular, $B(f, k_y) = 0$ for all $y \in \mathcal{X}$ which means $f \equiv 0$.

In addition, $B(f, k_y) = f(y)$. Finally, because

\begin{equation}
    B(f, f) = \sum\limits_{i,j}\bar{\alpha_i}\alpha_j K(x_i, x_j) \geq 0,
\end{equation}
and, by a Cauchy-Schwarz-type argument, $B(f, f) = 0$ if and only if $B(f, \omega) = 0$ for all $\omega \in \mathcal{W}$, $B$ is an inner product on $\mathcal{W}$. Thus, we have a well-defined inner product $B(\cdot, \cdot)$, and a reproducing kernel at each point $y \in \mathcal{X}$.

To complete the proof, we must take the completion of $\mathcal{W}$, denoted $\mathcal{H}$ and show that only functions which are in the span of reproducing kernels $\{k_y\}$ are in $\mathcal{H}$. Let $h \in \mathcal{H}$ and consider a Cauchy sequence $\{f_n\}$ converging to $h$. Then, 

\begin{equation}
    |f_n(x) - f_m(x)| = \langle f_n - f_m, k_x \rangle \leq ||f_n - f_m||\sqrt{K(x, x)}.
\end{equation}

In other words, the Cauchy sequence $\{f_n\}$ is actually pointwise Cauchy, so we may take $h = \lim\limits_{n \rightarrow \infty}f_n$. To show that $h$ is not only in the completion of $\mathcal{W}$, but also spanned by $\{k_y\}$, we simply compute 

\begin{equation}
    B(h, k_y) = B(\lim\limits_{n \rightarrow \infty}f_n, k_y) = \lim\limits_{n \rightarrow \infty}B(f_n, k_y) = \lim\limits_{n \rightarrow \infty}f_n(y) = h(y).
\end{equation}

We have thus constructed the reproducing kernel Hilbert space $\mathcal{H}$ on $\mathcal{X}$ with reproducing kernel $K$.

\end{proof}

\begin{theorem}\label{thm: sumkern}
A countable sum of kernel functions, if it converges, is a kernel function.
\end{theorem}

\begin{proof}
The sum of a finite number of nonnegative semidefinite matrices is nonnegative semidefinite. Given nonnegative semidefinite matrices $P_i$, we have 

\begin{equation}\label{eq: posdef}
    x^*(\sum\limits_{i=1}^N P_i)x \geq 0
\end{equation}

Because $[0, \infty]$ is closed and the product $x^*Px$ is continuous in the elements of $P$, \eqref{eq: posdef} holds in the limit as $N \rightarrow \infty$.
\end{proof}

The following function will be of particular interest. Let $\mathbb{B}^n$ denote the complex unit $n-$ball.

\begin{defn}
The $\mathbf{\omega-}$\textbf{kernel} $K: \mathbb{B}^n \times \mathbb{B}^n \rightarrow \mathbb{C}$ is defined as 

\begin{equation}
    K(x, y) = \frac{1}{1 - y^*x}
\end{equation}
\end{defn}

\begin{theorem}\label{thm: schurkern}
Let $x, y \in \mathcal{X}$ for some set $\mathcal{X} \in \mathbb{C}^d$ and let $K(x, y) = \sum\limits_{n=0}^{\infty}a_n(y^*x)^n$, where $\sum\limits_{n=0}^{\infty}a_n < \infty$ and $a_j \geq 0$ for all $j$. Then $K$, if it is bounded on $\mathcal{X}$, is a kernel function on $\mathcal{X}$. 
\end{theorem}
\begin{proof}
Note $f(x, y) = y^*x$ is a kernel function. Therefore, the matrix 

\begin{equation}
    \begin{bmatrix}
    (y_1^*x_1)^m & (y_2^*x_1)^m & \dots & (y_N^*x_1)^m \\
    (y_1^*x_2)^m & (y_2^*x_2)^m & \dots & (y_N^*x_2)^m \\
    \vdots & \vdots & \dots & \vdots \\
    (y_1^*x_N)^m & (y_2^*x_N)^m & \dots & (y_N^*x_N)^m
    \end{bmatrix}
\end{equation}
is nonnegative semidefinite, because it is a Schur product of $m$ nonnegative semidefinite matrices. So, $(y^*x)^m$ is a kernel function for all $m$. The result follows from Theorem \ref{thm: sumkern}.
\end{proof}

\begin{corollary}\label{cor: wkern}
The $\omega-$kernel is a kernel function.
\end{corollary}

% Another way to prove this is to use the fact that the Hadamard product of two nonnegative semidefinite matrices is nonnegative semidefinite, but we did not need such machinery for our result.

\section{Interpolation}

\begin{theorem}[\cite{paulsen2016introduction}]\label{thm: interpolation}
Let $\mathcal{H}$ be a RKHS over $\mathcal{X}$ with reproducing kernel $K(\cdot, \cdot)$ and consider $x_1, \dots, x_n \in \mathcal{X}$, $\lambda_1, \dots, \lambda_n \in \mathbb{C}$. Let $\boldsymbol{x} = (x_1, \dots, x_n)$ and $\boldsymbol{\lambda} = (\lambda_1, \dots, \lambda_n)^T$. Let $\mathcal{P}$ denote the matrix whose $i,j$ element is $K(x_i, x_j)$. Then, if $\mathcal{P}$ is invertible, the function $f = \sum\limits_{i=1}^n (\mathcal{P}^{-1} \boldsymbol{\lambda})_{i}k_{x_i}$ interpolates $\boldsymbol{x}$ and $\boldsymbol{\lambda}$, i.e. $f(x_j) = \lambda_j, \ j=1, \dots, n$. Moreover, $f$ is the function of minimum norm induced by the inner product of $\mathcal{H}$.
\end{theorem}

\begin{proof}[\cite{paulsen2016introduction}]
Plugging in, we get

\begin{equation}
    f(x_j) = \sum\limits_{i=1}^n (\mathcal{P}^{-1} \boldsymbol{\lambda})_{i}k_{x_i}(x_j) = \sum\limits_{i=1}^n (\mathcal{P}^{-1} \boldsymbol{\lambda})_{i}K(x_j, x_i) = (\mathcal{P}\mathcal{P}^{-1})_{j, :} \boldsymbol{\lambda} = \lambda_j,
\end{equation}

where $A_{j, :}$ denotes the $j^{th}$ row of matrix $A$.

Denote the subspace $\text{span}\{k_{x_j}\}_{j=1}^{n}$ as $\mathcal{H}_F$. The function $f$ is of minimum norm because, given another interpolating function $g$, then $(f - g) \in \mathcal{H}_F^{\perp}$. Now, the interpolating function of minimum norm must be the projection of $f$ onto $\mathcal{H}_F$; otherwise, the function would have terms which increase the norm, that could be zeroed out. In particular, if an interpolating function $g$ was of the form $g(z) = \sum\limits_{j=1}^{n}\beta_jk_{x_j}(z)$, then denoting $\boldsymbol{\beta} = (\beta_1, \beta_2, \dots, \beta_n)$, $(\boldsymbol{\beta} - \mathcal{P}^{-1} \boldsymbol{\lambda})$ is in the kernel of $\mathcal{P}$. Therefore, $f$ is the projection of an interpolating function onto $\text{span}\{k_{x_j}\}_{j=1}^{n}$ and we are done.

%The function $f$ is of minimum norm among the set of all interpolating functions, because $\boldsymbol{\alpha} = \mathcal{P}^{-1}\boldsymbol{\lambda}$ is the unique vector such that $\mathcal{P}\boldsymbol{\alpha} = \boldsymbol{\lambda}$, and the $k_{x_i}$'s are a basis for $\mathcal{H}$. Thus, $f$ is the projection of \textit{any} interpolating function in $\mathcal{H}$ onto the subspace spanned by $\{k_{x_i} \ | \ i = 1, \dots, n\}$.
\end{proof}

% \begin{figure}[H]
%     \centering
%     \includegraphics[scale=.3]{RKHS_Minimization.png}
%     \caption{A depiction of the stationarity condition of RKHS interpolation. The red line corresponds to the affine space of all interpolants in a RKHS $\mathcal{H}$ and the blue line corresponds to the span of reproducing kernels at points $\{x_j\}_{j=1}^{N}$. Here the level set corresponding to the stationarity condition given by interpolation is depicted as a black circle. Note the uniqueness of the stationarity condition in this picture. This uniqueness holds across all interpolation problems which satisfy the conditions given in Theorem \ref{thm: interpolation}}.
%     \label{fig: rkhs_min}
% \end{figure}

\section{Functions in our RKHS}

We begin with a lemma.

\begin{lemma}[\cite{paulsen2016introduction}]\label{lemma: prange}
Let $P \succeq 0$ be an $n \times n$ matrix and $x \in \mathbb{C}^n$. If there exists $c > 0$ such that $cP \succeq xx^*$, then $x$ is in the range of $P$. Moreover, if $Py = x$, then $0 \leq \langle x, y \rangle \leq c$.
\end{lemma}

\begin{proof}
Decompose $x = v + w$, where $v \in \mathrm{range}P$ and $w \in \mathrm{ker}P$. Then $\langle x, w \rangle = \langle w, w \rangle$, so

\begin{equation}
    ||w||^4 = w^*xx^*w \leq cw^*Pw = 0.
\end{equation}

So $x \in \mathrm{range}P$. 

Now if $x = Py$, then $\langle x, y \rangle = \langle Py, y \rangle \geq 0$. Finally,

\begin{equation}
    \langle x, y \rangle ^2 = y^*xx^*y \leq cy^*Py = c\langle x, y \rangle, 
\end{equation}

so $0 \leq \langle x, y \rangle \leq c$.
\end{proof}

The following theorem describes precisely which functions are included in our RKHS.

\begin{theorem}[\cite{paulsen2016introduction}]\label{thm: includedf}
Let $\mathcal{H}$ be a RKHS over $\mathcal{X}$ with reproducing kernel $K(\cdot, \cdot)$ and let $f: \mathcal{X} \rightarrow \mathbb{C}$ be a function. Then the following are equivalent.

\begin{enumerate}

    \item $f \in \mathcal{H}$.
    
    \item \textrm{There exists } $c > 0$ \textrm{ such that for any } $x_1, \dots, x_n \in \mathcal{X}$, \textrm{ there exists } $h \in \mathcal{H}$ \textrm{ such that } $||h|| \leq c, \quad f(x_i) = h(x_i), \ i = 1,\dots, n$.
    
    \item \textrm{There exists } $c > 0$ \textrm{ such that} 
    
    \begin{equation}
        c^2K(x, y) - f(x) \overline{f(y)}
    \end{equation}
    
    \textrm{is a kernel function.}
\end{enumerate}
\end{theorem}

\begin{proof}
We follow \cite{paulsen2016introduction}. 

\textbf{(1) $\implies$ (3)}
Letting $g = \sum_k k_{x_k}$, we get 

\begin{equation}
    \sum\limits_{i,j=1}^n f(x_j)\overline{f(x_i)} = |\langle f, g \rangle|^2 \leq ||f||^2||g||^2 = ||f||^2 \sum\limits_{i,j=1}^n K(x_i, x_j)
\end{equation}

\textbf{(3)} holds if we let $c = ||f||$.

\textbf{(3) $\implies$ (2)}
Let $\{x_1, \dots x_n\} \subseteq \mathcal{X}$ and let $h = \sum_j \alpha_j k_{x_j}$ such that $h(x_i) = f(x_i)$, for $i = 1, \dots, n$. Note we can do this by Theorem \ref{thm: interpolation}. Denoting $\lambda_i = f(x_i)$, by \ref{lemma: prange}, we get

\begin{equation}
    ||h||^2 = \langle \boldsymbol{\alpha}, \boldsymbol{\lambda}\rangle \leq c ^ 2.
\end{equation}

\textbf{(2) $\implies$ (1)}
We assume that for any finite set $X = \{x_1, \dots x_n\} \subseteq \mathcal{X}$, there exists a function $h_X \in \mathcal{H}$, such that $||h_X|| \leq c$ and $h_X(x_i) = f(x_i)$ for all $x_i \in X$. Let $g_X = P_X(h_X)$ denote the projection of $h_X$ onto the subspace spanned by $\{k_{x_i} \ | \ x_i \in X\}$. Then 

\begin{equation}
    g_X(x_i) = h_X(x_i) = f(x_i) \qquad x_i \in X
\end{equation}

and 

\begin{equation}
    ||g_X|| \leq h_X
\end{equation}

Let $\{g_X\}_{X \in \mathcal{F}_{\mathcal{X}}}$ denote the net of elements $g_X$, where $\mathcal{F}_{\mathcal{X}}$ denotes the set of all subsets of $\mathcal{X}$, let $M = \sup\limits_X ||g_X|| \leq c$, and let $\epsilon > 0$. Now let $X_0$ be an arbitrary finite subset of $\mathcal{X}$ such that $||g_{X_0}|| > M - \frac{\epsilon^2}{8M}$. Then we can choose an arbitrary $X \supseteq X_0$, and let $g_{X_0} = P_{X_0}(g_X)$. Then $\langle g_X - g_{X_0}, g_{X_0} \rangle = 0$. This implies

\begin{equation}
    ||g_X||^2 = ||g_X - g_{X_0}||^2 + ||g_{X_0}||^2,
\end{equation}

i.e.

\begin{multline}
    ||g_X - g_{X_0}||^2 = ||g_X||^2 - ||g_{X_0}||^2 = \\
    = (||g_X|| + ||g_{X_0}||)(||g_X|| - ||g_{X_0}||) = 2M * \frac{\epsilon^2}{2M} = \frac{\epsilon^2}{4}
\end{multline}

Thus, given two sets $X_1, \ X_2$ such that $X_0 \subseteq X_1$ and $X_0 \subseteq X_2$, we get

\begin{equation}
    ||g_{X_2} - g_{X_1}|| < \epsilon
\end{equation}

Therefore, the net (here meaning in the sense of a generalized sequence) is Cauchy and, because any net which converges in norm converges pointwise, we get that the limit 

\begin{equation}
    \lim\limits_{X} \{g_X\}_{X \in \mathcal{F}_{\mathcal{X}}} \rightarrow g = f,
\end{equation}

at all points $x \in \mathcal{X}$, and we are done.

\end{proof}

Thus, in the case where the kernel function is defined by 

\begin{equation}
    K(x, y) = \frac{1}{1 - y^*x}
\end{equation}

with induced RKHS $\mathcal{H}$, we find 

\begin{equation}
    f \in \mathcal{H} \iff \exists \ c > 0, \textrm{ such that } \frac{c^2}{1 - y^*x} - \overline{f(x)}f(y) \geq 0.
\end{equation}

% \section{Training: Expressivity and Obstacles}

% Despite our new machinery, there is in general no guarantee that a trained neural network will perform well on out-of-sample data. There are a few reasons for this unfortunate result, all of which aid in one's understanding of the interplay between training and accuracy. 

\section{Which Functions Are Included in the Induced RKHS?}\label{sec: w_vs_sb}

Theorem \ref{thm: includedf} tells us precisely which functions can be captured by our RKHS training method. On the other hand, given an activation function $f: \mathbb{C} \rightarrow \mathbb{C}$ and a function $g: \Omega \subseteq \mathbb{C}^{n} \rightarrow \mathbb{C}^{m}$ to be approximated, it is in general complicated to determine whether $g \in \mathcal{H}$, where $\mathcal{H}$ is the RKHS induced by $f$. There are some examples where the problem simplifies considerably. For example, the RKHS induced by the $\omega-$kernel consists of analytic functions on the unit ball, and is strictly contained in the Hardy space on the unit ball $H^2(\mathbb{B}^n)$ \cite{alpay2002schur}. Another noteworthy example is the \textit{Segal-Bargmann} space, which is the RKHS induced by $K(x, y) = e^{y^*x}$ in one dimension. This RKHS consists of entire functions \cite{paulsen2016introduction}. It is easy to generalize this to higher dimensions by simply taking $y^*x$ to be a dot product instead of an inner product. In particular, the power series expansion $\sum\limits_{n=0}^{\infty}\frac{(y^*x)^n}{n!}$, has nonnegative coefficients and is therefore, by Theorem \ref{thm: schurkern}, a kernel function.

Both these spaces have advantages and disadvantages against the other. Entire functions in general tend to be smoother than analytic functions in the unit ball. Moreover, they are defined throughout the complex plane. On the other hand, entire functions are far more restricted in their behavior, as is exemplified by the following, classical theorem.

\begin{theorem}[Jensen-type Result \cite{taojensen}]\label{thm: jensen}
Let $f$ be an entire function, not identically zero, of the form $|f(z)| \leq Ce^{C|z|^{\alpha}}$ for some $C, \alpha > 0$. Then there exists $C' > 0$ such that the number of zeros (counting multiplicity) of $f$ in a disk around zero of radius $R \geq 1$ is at most $C'R^{\alpha}$.
\end{theorem}

\begin{proof}
If, $f(0) \neq 0$, Jensen's formula from complex analysis tells us, for any $r > 0$, 

\begin{equation}
    \frac{1}{2\pi}\int_{0}^{2\pi}\log|f(re^{it})|\mathrm{d}t = \log|f(0)| + \sum\limits_{|\rho| \leq r \ ; \ f(\rho) = 0}\log\frac{r}{|\rho|}
\end{equation}

Plugging in $r = 2R$, we get

\begin{equation}
    \int_{0}^{2\pi} \log|Ce^{C(2R)^{\alpha}}| \geq \log|f(0)| + N_R \log2,
\end{equation}
where $N_R$ denotes the number of zeros in a disk around the origin of radius $R$. 

If $f(0) = 0$, then we can perturb $f$ by a small amount to make it nonzero at $0$, and adjust our constant $C$ accordingly.
\end{proof}

The theorem above states that, if $f$ is forced to take on a large number of zeros in some disk about the origin, then $f$ must proportionally increase exponentially (up to a constant $\alpha$). For our purposes, this means that using too many interpolation nodes in the Segal-Bargmann space will cause our function to blow up outside a disk containing the nodes. Thus, any hope of generalizing outside such a disk will be lost. Figures \ref{fig: homo13} and \ref{fig: homo25} exemplify this.

\begin{figure}[H]
\begin{subfigure}{.5\textwidth}
  \centering
  \includegraphics[width=.8\linewidth]{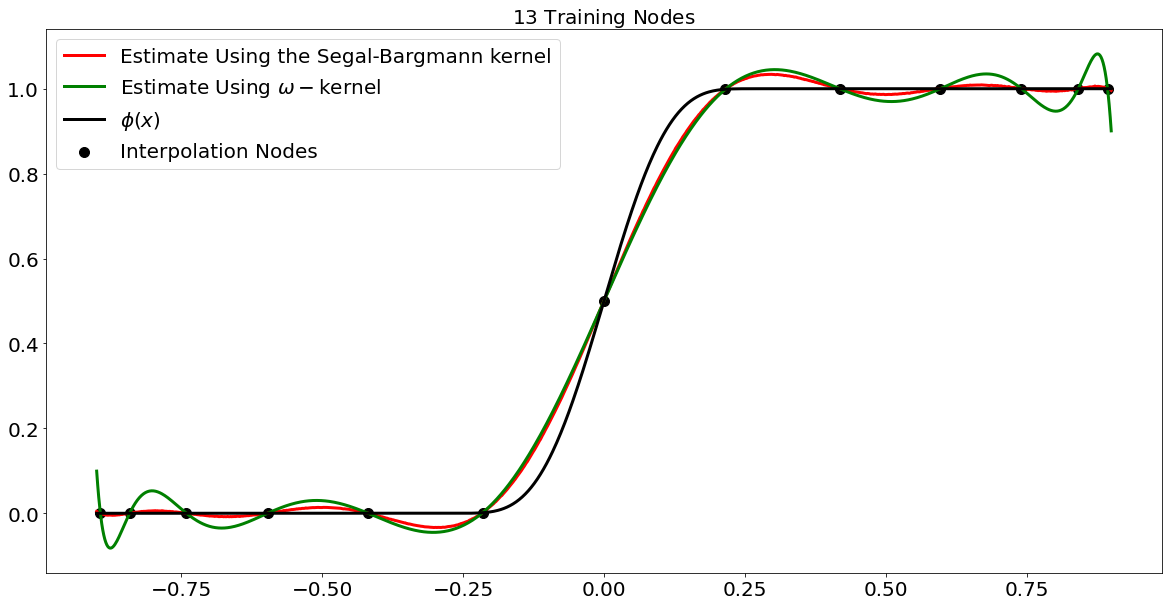}
  \caption{$10$ Training Nodes}
  \label{fig: homo13}
\end{subfigure}%
\begin{subfigure}{.5\textwidth}
  \centering
  \includegraphics[width=.8\linewidth]{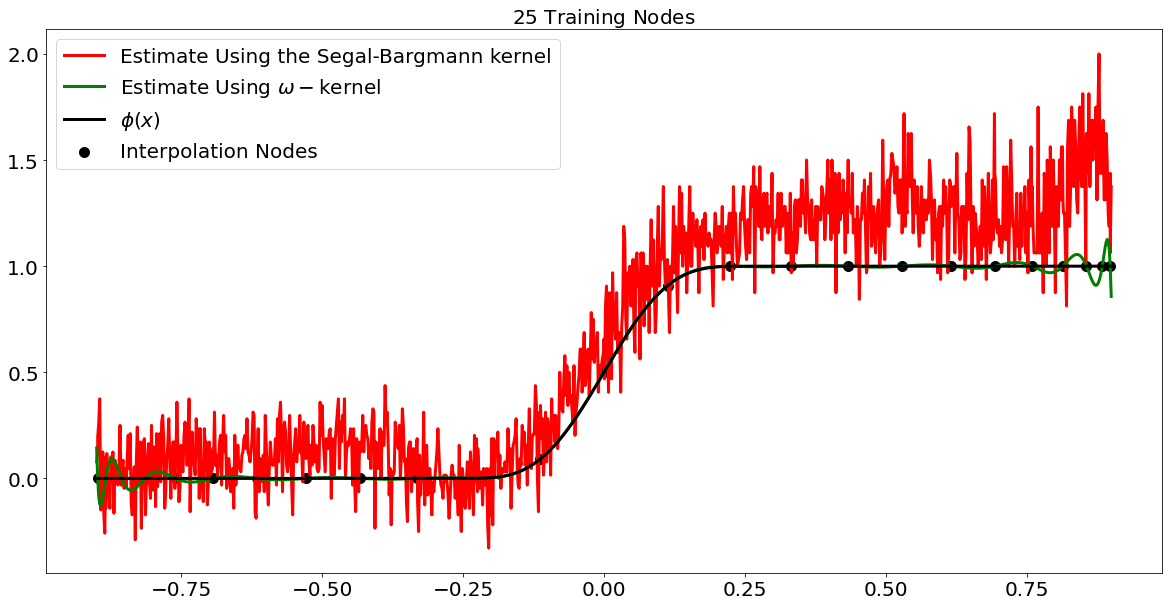}
  \caption{$25$ Training Nodes}
  \label{fig: homo25}
\end{subfigure}
\caption{Interpolating a $C^{\infty}$ function $\phi$ which is not analytic. Note how the increase in training examples forces the Segal-Bargmann interpolant to become unstable, while the $\omega-$kernel interpolant remains tame, and increases in accuracy.}
\label{fig:fig}
\end{figure}

In these figures, we are trying to learn the function

\begin{equation}\label{eq: homotopy}
    \phi(x) = \frac{\lambda (x - \frac{1}{3})}{\lambda (x - \frac{1}{3}) + \lambda (\frac{2}{3} - x)},
\end{equation}

where

\begin{equation}
    \lambda(t) = \begin{cases}
        0, \qquad t \leq 0 \\
        e^{-t^{-1}}, \qquad t > 0
    \end{cases}
\end{equation}

In particular, we perform interpolation using the $\omega-$kernel and the Segal-Bargmann kernel. Our interpolation nodes are Chebyshev, contracted to the interval $[-.9, .9]$ to make sure the $\omega-$kernel interpolant does not blow up. The function $\phi$ has uncountably many zeros in the unit disk and so it should be difficult for the Segal-Bargmann interpolant to remain stable as the number training examples increases. In Figure \ref{fig: homo13}, we see the Segal-Bargmann network is closer to the true curve than the $\omega$-kernel network; it is also smoother and does not thrash as it reaches the boundary of $[-.9, .9]$. However, as the number of nodes increases from $13$ to $25$, the Segal-Bargmann network breaks (Figure \ref{fig: homo25}). This is because it is an entire function, so it must grow extremely fast due to Theorem \ref{thm: jensen}. In addition, the network resulting from interpolation has ill-conditioned weight matrices, relying on catastrophic cancellation in order to interpolate the data. Indeed, even with $13$ interpolation nodes, the Segal-Bargmann interpolant is struggling. Explicitly, the minimum magnitude of the weights $\mathbf{\alpha}$ (in the context of Theorem \ref{thm: interpolation}) of the Segal-Bargmann interpolant is $~10^{11}$. In comparison, the maximum element in $\mathbf{\alpha}$ corresponding to the $\omega-$kernel interpolant is about $16,000$, which is still stable. Note, this analysis is based on the conditioning of the interpolation problem (in the obvious coordinate system), not on the conditioning of the algorithm.

Even without the numerical instability involved in the above example, the Segal-Bargmann interpolant must blow up rather quickly. Indeed, if we expand the $x-$axis in Figure \ref{fig: homo13}, as shown in Figure \ref{fig: homo13extend}, one sees clearly that the interpolant grows (or decreases) exponentially outside a disk containing its interpolation nodes. 

\begin{figure}[H]
    \centering
    \includegraphics[scale=.3]{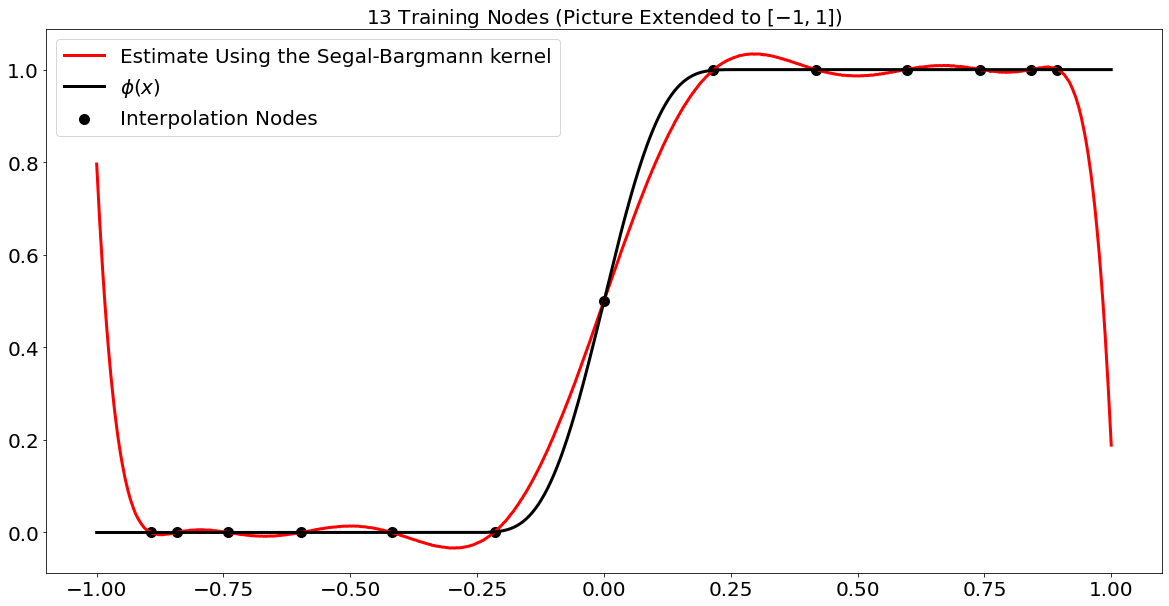}
    \caption{A Segal-Bargmann interpolant, expanded outside a disk containing its interpolation nodes.}
    \label{fig: homo13extend}
\end{figure}

\section{Agler-McCarthy Theorem}\label{sec: agler}

Until now, we have been using the $\omega-$kernel without much motivation for the multivariate case. In this section, we study a universality property for the $\omega-$kernel given by the Agler-McCarthy theorem. This property will be useful when extending our theory to deep learning in Section \ref{sec: deeplearning} and introducing model reduction in Chapter \ref{ch: cs}.

Let $\mathcal{A} = \{(x_1, \lambda_1), (x_2, \lambda_2), \dots, (x_N, \lambda_N)\}$ be interpolation data, and let $\mathcal{H}$ be a RKHS over $\mathcal{X}$ with kernel $K$. 

\begin{defn}
Let $\mathcal{H}$ be a RKHS over $\mathcal{X}$. A \textbf{multiplier of $\mathcal{H}$} is a function $\phi$ on $\mathcal{X}$, such that if $f \in \mathcal{H}$, then $\phi f \in \mathcal{H}$.
\end{defn}

Let $R_{x, \lambda}$ be the operator which sends $k_{x_j} \mapsto \bar{\lambda_j}k_{x_j}$ for all $j$. In order to solve the interpolation problem given by set $\mathcal{A}$, it is necessary that $R_{x, \lambda}$ is a contraction on the span of $\{k_{x_j}\}$. $K$ is a \textit{Nevanlinna-Pick kernel} if this condition on $R_{x, \lambda}$ is also sufficient to solve the interpolation problem.

\begin{remark}
\cite{agler2000complete} calls Nevanlinna-Pick kernels \textit{complete Nevanlinna-Pick kernels}, to distinguish them from another, broader class of kernels, which we will not study here.
\end{remark}

\begin{defn}
Let $\mathcal{H}$ be a RKHS over $\mathcal{X}$ with kernel $K$. $K$ is \textbf{irreducible} if, for all $x, y \in \mathcal{X}$, $K(x, y) \neq 0$.
\end{defn}

If $\mathcal{H}$ is a RKHS over $\mathcal{X}$ whose kernel $K$ is a Nevanlinna-Pick kernel, then $\mathcal{X}$ can be partitioned into disjoint subsets $\mathcal{X}_j$ such that $x, y \in \mathcal{X}_j \implies K(x, y) \neq 0$ and $x \in \mathcal{X}_j, y \in \mathcal{X}_k \implies K(x, y) = 0$ for $j \neq k$ \cite{agler2000complete}.

\begin{theorem}\label{thm: agler}
Let $\mathcal{H}$ be a RKHS over $\mathcal{X}$ with irreducible kernel $K$. $K$ is a Nevanlinna-Pick kernel, if and only if there exists a nowhere vanishing function $\delta(\cdot)$ on $\mathcal{X}$, and an injection $f: \mathcal{X} \rightarrow \mathbb{B}^m$, such that $K$ is of the form

\begin{equation}
    K(x, y) = \frac{\delta(x)\overline{\delta(y)}}{1 - \langle f(x), f(y)\rangle}
\end{equation}

Moreover, denoting $k_y(\cdot) = K(\cdot, y)$, the map $k_x(\cdot) \mapsto \frac{\overline{\delta(y)}}{1 - \overline{f(y)}\cdot}$ extends to an isometric embedding from $\mathcal{H}$ to the RKHS induced by the kernel $K(x, y) = \frac{1}{1 - y^*x}$, multiplied by $\delta$. Finally, if there is a topology on $\mathcal{X}$ such that $K$ is continuous, then $f$ is a continuous embedding of $\mathcal{X}$ into $\mathbb{B}^m$.

\end{theorem}

\begin{proof}
See \cite{agler2000complete}.
\end{proof}

This theorem states that, up to scaling (by $\delta$), the collection of functions in the RKHS induced by the $\omega-$kernel is the most expressive set one can use, if one requires a complete Nevanlinna-Pick kernel. Note that the function $f$ included above is simply a preprocessing of our data. In Section \ref{sec: cs}, we will see that complete Nevanlinna-Pick kernels are particularly attractive, so the Agler-McCarthy theorem tells us that the $\omega-$kernel is optimal in terms of expressivity.

\begin{remark}
In practice, the universal property of the $\omega-$kernel can lead to some issues, in that the space of functions included in the RKHS induced by the $\omega-$kernel can be \enquote{too big}. In particular, the functions one obtains via interpolation can thrash, or simply perform poorly under norms which are regularly used in practice (e.g. $L^2, L^{\infty}$). In other words, there is a trade-off between expressivity and generalization. Therefore, it will sometimes be helpful to restrict ourselves to smaller classes of functions.
\end{remark}

\chapter{Reproducing Kernel Krein Spaces}\label{ch: rkks}
We introduce a generalization of RKHS, called \textit{reproducing kernel Krein spaces} (RKKS). It will be helpful to carry some simple examples along during this chapter, and we shall confine ourselves to using \textit{softplus} ($f(z) = \log(1 + e^z)$) and the hyperbolic tangent, $\tanh(z)$ as such.

We begin by writing out the formal power series expansions for softplus and tanh:

\begin{equation}
    \log(1 + e^z) = \log2 + \frac{1}{2}z + \frac{1}{8}z^2 - \frac{1}{192}z^4 + O(z^6)
\end{equation}

\begin{equation}
    \tanh(z) = \sum\limits_{n=1}^{\infty}\frac{2^{2n}(2^{2n} - 1)B_{2n}}{2n!}z^{2n - 1} = z - \frac{1}{3}z^3 + \frac{2}{15}z^5 + O(z^7),
\end{equation}
where $B_k$ are the Bernoulli numbers. Now from Theorem \ref{thm: schurkern}, we know if the power series coefficients of an activation function are nonnegative (and decay sufficiently quickly in the domain of the activation function), then that activation induces a reproducing kernel Hilbert space, as outlined in the previous sections. Here, however, the power series can include negative coefficients, which are accomodated by the generalization to RKKS.

\begin{defn}
    A \textbf{Krein space} is a (complex) vector space $\mathcal{K}$ with Hermitian form $[\cdot, \cdot]$, which admits a decomposition 

    \begin{equation}
        \mathcal{K} = \mathcal{K}_+ \oplus \mathcal{K}_-,
    \end{equation}
    where 
    
    \begin{itemize}
        \item $\mathcal{K}_+$ (resp., $\mathcal{K}_-$), when endowed with the form, $[\cdot, \cdot]$ (resp., $-[\cdot, \cdot]$), is a Hilbert space. 
        \item $\mathcal{K}_+ \cap \mathcal{K}_- = 0$ and $[k_+, k_-] = 0$, for $k_+ \in \mathcal{K}_+$, $k_- \in \mathcal{K}_-$.
    \end{itemize}
If $[f, g] = 0$, we say $f$ is \textbf{orthogonal in the Krein sense} to $g$.
\end{defn}

\begin{defn}
Given $\mathcal{X} \subseteq \mathbb{C}^n$, a \textbf{reproducing kernel Krein space} $\mathcal{K}$ is a Krein space of functions on $\mathcal{X}$ equipped with Hermitian form $[\cdot, \cdot]$ and a function $k: \mathcal{X} \times \mathcal{X} \rightarrow \mathbb{C}$, such that, for any $z \in \mathcal{X}$, 

\begin{itemize}
    \item $k_z \in \mathcal{K}$, where $k_z: y \mapsto k_z(y)$
    \item $\forall f \in \mathcal{K}, f(z) = [f, k_z]$
\end{itemize}

We call $K(x, y) = k_y(x)$ the \textit{reproducing kernel} for $\mathcal{K}$.
\end{defn}

Below, we state, without proof, a theorem of Schwartz \cite{schwartz1964sous}. See also \cite{alpay1991some}:

\begin{theorem}[Schwartz]\label{thm: schwartz}
$K$ is a reproducing kernel for some reproducing kernel Krein space if $K = K_+ + K_-$, where $K_+$ and $-K_-$ are positive definite in their arguments. Similarly, if $\mathcal{K}$ is a reproducing kernel Krein space with reproducing kernel $K$, then $K$ must admit a decomposition $K = K_+ + K_-$, where again, $K_+$ and $-K_-$ are positive definite. Moreover, given a reproducing kernel Krein space $\mathcal{K}$ with kernel $K = K_+ + K_-$, then without loss of generality, we can assume $\mathcal{H}(K_+) \cap \mathcal{H}(K_-) = \{0\}$. Here, $\mathcal{H}(\cdot)$ represents the reproducing kernel Hilbert space induced by kernel $\cdot$. With this representation, $\mathcal{K}$ consists of functions of the form $f = f_+ + f_-$, where $f_+ \in \mathcal{H}(K_+)$ and $f_- \in \mathcal{H}(K_-)$.
\end{theorem}

For more information about reproducing kernel Krein spaces, \cite{alpay1991some} is an excellent resource. 

We now show that the hyperbolic tangent function has power series coefficients which decay sufficiently quickly, so as to apply Theorem \ref{thm: schwartz}. In particular, the resulting power series coefficients decay fast enough for one to take the terms with positive coefficients and sum them to get a positive definite function. Similarly, one may take the sum of the terms with negative coefficients to obtain a negative definite function. Considering 

\begin{equation}
    \log(1 + e^z) + C = \int \frac{1}{2} + \frac{1}{2}\tanh(\frac{z}{2}) dz,
\end{equation}
the result given will also hold for softplus.

\begin{lemma}
$\tanh(z)$ has an absolutely convergent Taylor series with radius of convergence $\frac{\pi}{2}$.
\end{lemma}

\begin{proof}
The even Bernoulli numbers satisfy

\begin{equation}
    B_{2n} = (-1)^{n - 1}\frac{2 (2n)!}{(2\pi)^{2n}}\zeta(2n)
\end{equation}

Therefore,

\begin{equation}
    \tanh(z) = 2\sum\limits_{n=0}^{\infty}\frac{2^{2n} - 1}{\pi^{2n}}\zeta(2n)
\end{equation}

Because $\zeta(\cdot)$ is bounded on the nonnegative reals, result follows.
\end{proof}

\begin{defn}
Given a RKKS $\mathcal{K}$ with bilinear form $[, ]$, the \textbf{Krein quadratic form} of a function $f \in \mathcal{K}$ is 

\begin{equation}
    [f, f]
\end{equation}

\end{defn}

\begin{remark}
\textbf{The square root of a Krein quadratic form is generally not a norm.}
\end{remark}

\section{The Associated Hilbert Space Theorem}
\begin{defn}
Let 

\begin{equation}
    K(x, y) = \sum\limits_{j=0}^{\infty}a_j (y^*x)^j
\end{equation}
be a kernel function for a Krein space $\mathcal{K}$. Then the \textbf{associated Hilbert kernel of $K$}, denoted $\overset{\mathcal{H}}{\mathcal{K}}$ is given by

\begin{equation}
    \overset{\mathcal{H}}{K}(x, y) = \sum\limits_{j=0}^{\infty}|a_j| (y^*x)^j
\end{equation}
The RKHS induced by $\overset{\mathcal{H}}{\mathcal{K}}$ is called the \textbf{associated Hilbert space of $\mathcal{K}$}.
\end{defn}

Note the associated Hilbert space of a RKKS and its kernel are unique.

The concept of an associated Hilbert space is of fundamental value because of the following theorem:

\begin{theorem}[Associated Hilbert Space]\label{thm: assochilb}
Let $f$ have an absolutely summable power series expansion

\begin{equation}
    f(z) = \sum\limits_{j=0}^{\infty}a_jz^j
\end{equation}
and consider a reproducing kernel $K: \mathcal{X} \times \mathcal{X} \rightarrow \mathbb{C}$ such that

\begin{equation}
    K(x, y) = f(y^*x)
\end{equation}
Denote $\mathcal{K}$ as the Krein space corresponding to reproducing kernel $K$. Let $\{c_k\}_{k=0}^{\infty}$ be a sequences of numbers in $\{-1, 1\}$. Then the Krein space induced by 

\begin{equation}
    K(x, y) = \sum\limits_{j=0}^{\infty}c_ja_jz^j
\end{equation}
contains the same functions as the original Krein space $\mathcal{K}$.

\end{theorem}

\begin{proof}
Because $f$ has an absolutely summable set of power series coefficients, we can re-arrange them freely. Specifically, we can group the terms with positive coefficients together and the terms with negative coefficients together to form two kernel functions $K_+, K_-$, respectively. Now, by Theorem \ref{thm: schwartz}, each $f \in \mathcal{K}$ is of the form $f_+ + f_-$, where $f_+ \in \mathcal{H}(K_+)$ and $f_- \in \mathcal{H}(K_-)$. Consider the function 

\begin{equation}
    g(z) = \sum\limits_{j=0}^{\infty}|a_j|z^j
\end{equation}
This function yields a non-negative semidefinite kernel

\begin{equation}
    K(x, y) = \sum\limits_{j=0}^{\infty}|a_j|(y^*x)^j
\end{equation}
which induces a RKHS all of whose functions are of the form $f_+ + f_-$, where $f_+ \in \mathcal{H}(K_+)$ and $f_- \in \mathcal{H}(K_-)$. Therefore the factors $c_j$ in a kernel function of the form $K(x, y) = \sum\limits_{j=0}^{\infty}c_ja_j(y^*x)^j$ are irrelevant to the functions included in the corresponding Krein space.
\end{proof}

Theorem \ref{thm: assochilb} is important because it will allow a seamless way to examine the expressivity of Krein kernels via the theory of RKHS.

\begin{proposition}\label{prop: kreincounter}
Given a Krein space $\mathcal{K}$ over $\mathcal{X}$ with kernel $K$, the interpolation of Theorem \ref{thm: interpolation} does not in general result in an interpolant of minimal Krein quadratic form. Moreover, the resulting interpolant is not in general of minimum Hilbert space norm in the associated Hilbert space  of $\mathcal{K}$.
\end{proposition}

\begin{proof}
For the first claim, consider the counterexample of using a strictly negative kernel function $-K_+(x, y)$ as one's kernel for a RKKS. Then the interpolant resulting from Theorem \ref{thm: interpolation} results in a maximum quadratic form interpolant.

For the second claim, consider using softplus $K(x, y) = \log(1 + e^{y^*x})$ as the kernel for a RKKS. The associated Hilbert kernel is given by 

\begin{equation}
    \overset{\mathcal{H}}{K}(x, y) = -\log(1 + e^{iz}) + \log4 + \frac{1}{2}(1 + i)z,
\end{equation}
which has power series coefficients equal to the absolute values of the coefficients of softplus.

Consider the interpolation point $\{(1, 1)\}$, i.e. we want a function $f$ such that $f(1) = 1$. Then, using the softplus kernel in the algorithm of Theorem \ref{thm: interpolation}, we get

\begin{equation}
    f(z) = \frac{\log(1 + e^z)}{\log(1 + e)}
\end{equation}
Moreover, $[f, f] = (\frac{1}{\log(1 + e)})^*(1) = \frac{1}{\log(1 + e)} \approx 1.75$, where$[\cdot, \cdot]$ denotes the Krein quadratic form.

On the other hand, using the associated Hilbert kernel, we obtain the minimum norm interpolant

\begin{equation}
    g(z) = \frac{-\log(1 + e^{iz}) + \log4 + \frac{1}{2}(1 + i)z}{-\log(1 + e^i) + \log4 + \frac{1}{2}(1 + i)} \neq \frac{\log(1 + e^z)}{\log2},
\end{equation}
and $||g||^2 = (-\log(1 + e^i) + \log4 + \frac{1}{2}(1 + i))^*(1) = (\frac{1}{2} + \log4 + - \frac{1}{2}\log(2(1 + \cos(1))))^*(1) \approx 1.32$, where $||\cdot||$ denotes the Hilbert space norm.

\end{proof}

It is easy to see that the interpolation algorithm of Theorem \ref{thm: interpolation} results in an interpolant, even in a Krein space. However, by Proposition \ref{prop: kreincounter}, the result is not of minimal norm, as orthogonal projections are not Krein form-minimizing in a Krein space.

Note that while sign changes in a function's power series do not affect the functions in one's induced reproducing kernel space, they do, however, change the bilinear form of one's Krein space, and therefore change interpolation results.

\begin{remark}
It is important to note that Theorem \ref{thm: schwartz} does \textit{not} imply a bijection between differences of kernel functions and RKKS. It does show, however, that once one has a difference of kernel functions, there must exist a RKKS. For our purposes, however, there will be a bijective relationship, specifically because we are dealing with an RKKS analogue of \textit{weighted Hardy spaces}, i.e. RKKS with kernels of the form $K(x, y) = \sum\limits_{j=0}^{\infty}a_j (y^*x)^j$. These spaces consist of elements with formal power series expansions, which allows us to draw the bijection.

Concretely, consider a kernel $K(x, y) = K_1(x, y) - K_2(x, y) = \sum\limits_{j=0}^{\infty}a_j (y^*x)^j - \sum\limits_{k=0}^{\infty}b_k (y^*x)^k$, where $a_j \geq 0, \ b_k \geq 0$ for all $j, k$ and each series is absolutely summable within some radius of convergence. Moreover, suppose $a_j > 0 \implies b_j = 0$ and $b_j > 0 \implies a_j = 0$. Then the RKHS induced by $K_1$ is disjoint from the RKHS induced by $K_2$. This follows immediately from Theorem \ref{thm: includedf}. Indeed, denote $\mathcal{K}_j$ as the RKHS induced by $K_j$, $j=1, 2$. Then, given some function $f \in \mathcal{K}_1$, we can write $f(z) = \sum\limits_{l=0}^{\infty}c_lz^l$. By Theorem \ref{thm: includedf}, there exists some constant $c$ such that $cK_1(x, y) - f(x)\overline{f(y)}$ is a kernel function. $f(x)\overline{f(y)}$ itself has a (SCV) power series expansion, and each of its terms with nonzero coefficients must line up with the terms of the power series of $K_1$. Considering the terms with nonzero coefficients in $K_1$ are disjoint from the terms with nonzero coefficients in $K_2$, it follows that $f \not\in \mathcal{K}_2$. Similarly, if $f \in \mathcal{K}_2$, then $f \not\in \mathcal{K}_1$.
\end{remark}

\subsection{Construction of Associated Hilbert Spaces}
In practice there are some concrete ways of forcing the power series coefficients of a function to be positive. For example, in the softplus case, we have

\begin{equation}
    -\log(1 + e^{iz}) + \log4 + \frac{1}{2}(1 + i)z = \log2 + \frac{1}{2}z + \frac{1}{8}z^2 + \frac{1}{192}z^4 + O(z^6)
\end{equation}

In the hyperbolic tangent case, we have

\begin{equation}
    -i\tanh{iz} = z + \frac{1}{3}z^3 + \frac{2}{15}z^5 + O(z^7) = \tan(z)
\end{equation}

For computational purposes, one may be worried that the matrix to be inverted in Theorem \ref{thm: interpolation} might not be positive definite. The following result shows that, if one's RKHS contains the polynomials, the corresponding kernel function must be strictly positive definite.  

\begin{theorem}[\cite{paulsen2016introduction}]\label{thm: polynomials}
Let $\mathcal{H}$ be a RKHS on $\mathcal{X}$ with kernel function $K$. If $K$ is not strictly positive definite, then $\mathcal{H}$ cannot contain the polynomials.
\end{theorem}

\begin{proof}
Because $K$ is not positive definite, there exist $\{x_j\} \subseteq \mathcal{X}$ such that $\sum\limits_{j, k=0}^{\infty}\alpha_k \overline{\alpha_j}K(x_j, x_k) = ||\sum\limits_{j=0}^{\infty} \alpha_j k_{x_j}|| = 0$. In particular, this means that, for all $f \in \mathcal{H}$, $\langle f, \sum\limits_{j=0}^{\infty} \alpha_j k_{x_j} \rangle = 0$. On the other hand, there is no $\alpha$ such that, for all polynomials $p$, $\sum\limits_{j=0}^{\infty} \alpha_j p(x_j) = 0$ holds.
\end{proof}

This motivates us, for example, to make sure the power series expansions of our kernel functions have a constant term.

%This motivates us, for example, to use what are called \textit{bias terms} in neural networks. Specifically, neural network practitioners often use the structure 

%\begin{equation}
%    F(x) = A_2 \begin{bmatrix}
%    1 \\
%    [f(A_1x)]_1 \\
%    [f(A_1x)]_2 \\
%    \vdots \\
%    [f(A_1x)]_N
%    \end{bmatrix}
%\end{equation}

%in their neural networks (note $A_2$ must be augmented for this to work). The above theorem yields a rigorous reason for doing so.

To illustrate how important this constant term is, let us consider a simple example of approximating a polynomial $p(x) = 5x^3 - x^2 + 3x$.

\begin{figure}[H]
\begin{subfigure}{.5\textwidth}
  \centering
  \includegraphics[width=.8\linewidth]{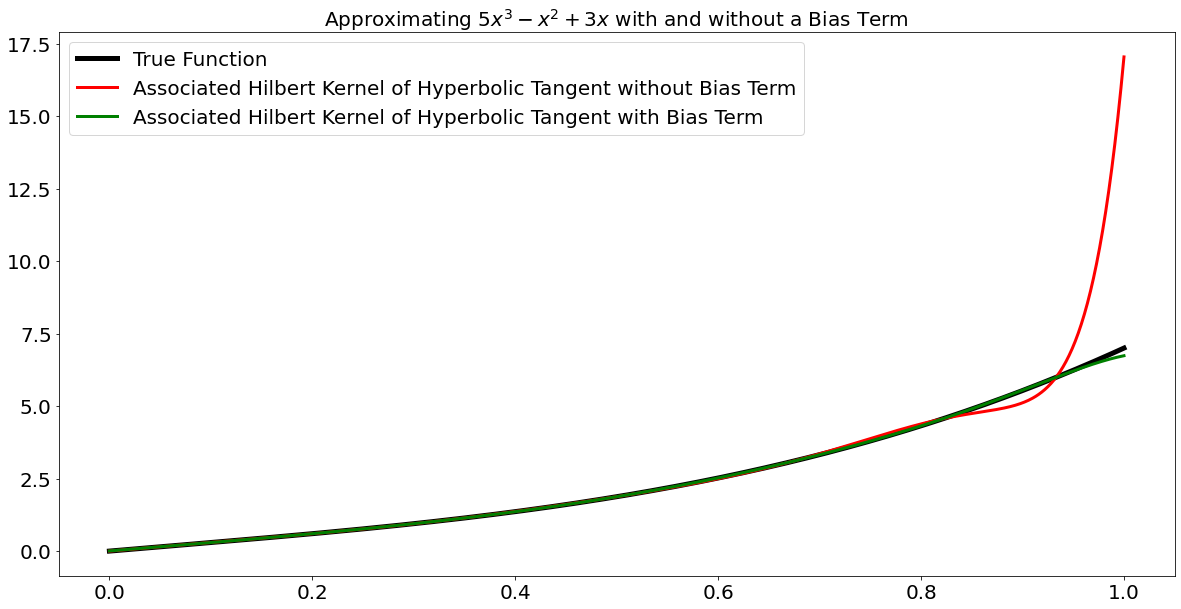}
  \caption{True function and interpolating estimates}
  \label{fig: tanh1}
\end{subfigure}%
\begin{subfigure}{.5\textwidth}
  \centering
  \includegraphics[width=.8\linewidth]{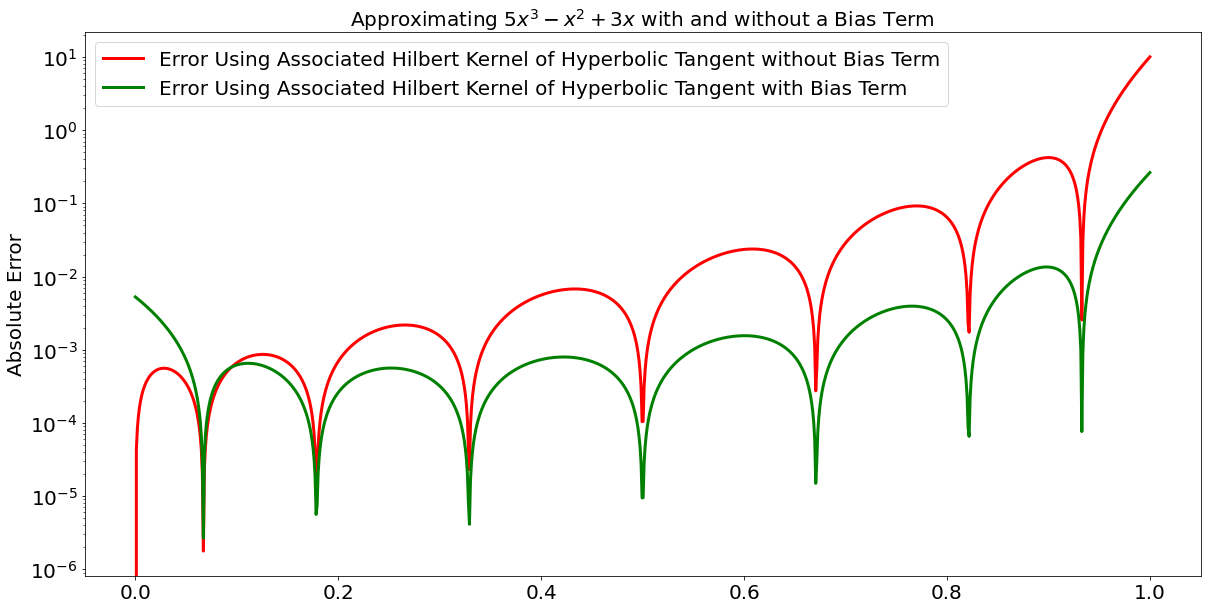}
  \caption{Error}
  \label{fig: tanh2}
\end{subfigure}
\caption{Interpolating a polynomial at $7$ nodes with and without a bias term. Note that the constant term grants the green line the ability to undershoot at the beginning, thereby decreasing the overall uniform error.}
\label{fig:fig}
\end{figure}

Even though the function to be interpolated lies in a subspace which does not contain the constant functions, this does not help the associated Hilbert kernel of hyperbolic tangent to approximate it any better. Adding in the constant term $+1$, however, significantly helps approximation. In fact, looking on the right-hand side of the plot, the red curve's error is more than double the uniform error of the green curve, which shows that merely shifting the red curve down will not allow it to overtake the green curve's uniform accuracy.

\begin{remark}
A power series with complex-valued coefficients is not the kernel for a RKKS. A \enquote{kernel} of this form would apply imaginary magnitudes to some elements in the reproducing kernel space, which perhaps could be an alternative line of research. However, for the work presented here, we are more interested in finding extrema (e.g. minimum norm interpolants), and there is no clear way of interpreting an element of imaginary magnitude, because $\mathbb{C}$ is not an ordered field. 
\end{remark}

\section{Krein Space Interpolation Viewed as Regularization}
Krein spaces which have the same associated Hilbert space are equipped with different geometries and therefore result in different regularizations through interpolation. It should be noted that the use of RKKS in function approximation has been studied; in particular, Oglic and G\"artner studied its use in the context of developing generalization bounds in terms of Rademacher complexity \cite{oglic2018learning}. Our approach and results differ in the sense that we view our optimization problem from the context of an associated Hilbert space. Our regularization stems from here, rather than by imposing constraints and Tikhonov regularization-type terms on an optimization problem.

\subsection{Example}\label{subsec: tantanh}
Consider Figures \ref{fig: rkks_regularize} and \ref{fig: rkks_regularize_error}. Again, we are interpolating a simple polynomial, but rather than working only in the associated Hilbert space of the RKKS induced by $K(x, y) = \tanh(y^*x)$, we use the interpolant given to us by performing RKKS interpolation as well. Note, by Theorem \ref{thm: assochilb} both neural networks will have the same expressivity.

\begin{figure}[H]
\begin{subfigure}{.5\textwidth}
  \centering
  \includegraphics[width=.8\linewidth]{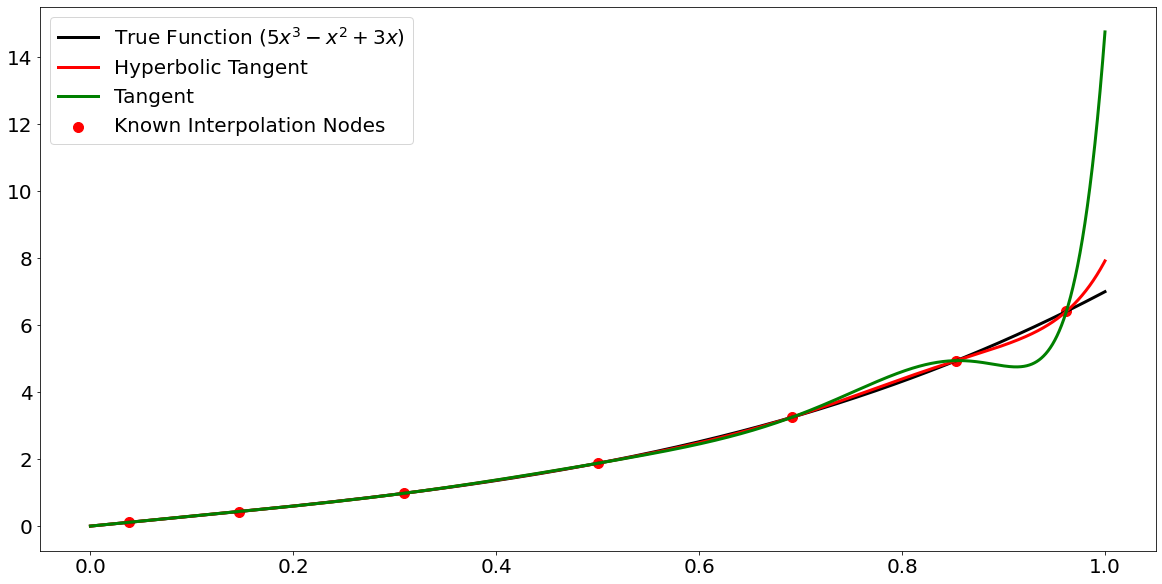}
  \caption{True function and interpolating estimates}
  \label{fig: rkks_regularize}
\end{subfigure}%
\begin{subfigure}{.5\textwidth}
  \centering
  \includegraphics[width=.8\linewidth]{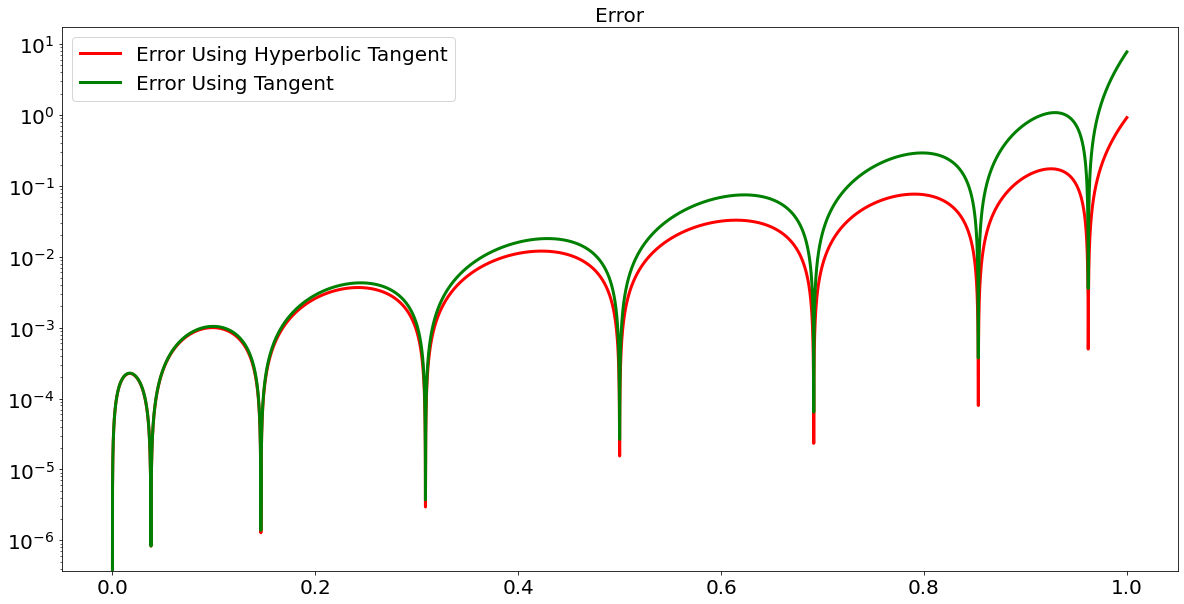}
  \caption{Error}
  \label{fig: rkks_regularize_error}
\end{subfigure}
\caption{Finding an interpolant of minimal norm in the RKHS induced by using $K(x, y) = \tan(y^*x)$ does not yield an approximant with good generalization properties. In this case, using hyperbolic tangent is more sensible.}
\label{fig:fig}
\end{figure}

In this case, using the associated Hilbert space (which again, corresponds to using $K(x, y) = \tan(y^*x)$ as our kernel function) causes thrashing toward the right portion of the graph. On the other hand, the RKKS interpolant, while it is not a minimum norm interpolant, generalizes quite better.

This example illustrates that working in a RKKS may have attractive properties. Below, we will show that these properties correspond to performing regularization against the optimization problem given in an associated Hilbert space.

\subsection{The Geometry of RKKS Interpolation}

% While the above argument works, there is an alternative, geometric picture which clearly describes RKKS interpolation as a regularized version of minimum norm interpolation. 

The following will present rigorous justification for why the neural network with $\tanh(\cdot)$ activation outperforms the neural network with $\tan(\cdot)$ activation in Subsection \ref{subsec: tantanh}. In particular, we show that RKHS/RKKS interpolation results in the unique realization of a stationarity condition, given by the norm/Krein quadratic form.

We begin with a motivational lemma, which will later be generalized.

\begin{lemma}\label{lemma: rkksmin}
A line $l$ in $\mathbb{R}^2$ given by $y = mx + b$ is always tangent to a unique circle about the origin. In addition, if $m \neq \pm 1$ and $l$ does not pass through the origin, then $l$ is always tangent to a unique hyperbola of either the form $x^2 - y^2 = r^2$ or $y^2 - x^2 = r^2$.
\end{lemma}

\begin{proof}
The first claim is classical and immediately follows from the fact that the (unique) perpendicular drawn from the origin onto $l$ is orthogonal to a circle about the origin with radius equal to the perpendicular.

To prove the second claim, note that

\begin{equation}
    y^2 - x^2 = r^2 \implies \frac{\d y}{\d x}y - x = 0 \implies \frac{\d y}{\d x} = \frac{x}{y},
\end{equation}
so we get

\begin{equation}
    y = m^2 y + b \implies y = \frac{b}{1 - m^2}.
\end{equation}
Plugging into the equation $y^2 - x^2 = r^2$, we obtain $r = \frac{b}{+\sqrt{1 - m^2}}$. Note that this only works if $|m| < 1$. If $|m| > 1$, then we use the above argument, fitting a hyperbola of the form $x^2 - y^2 = r^2$ to obtain $r = \frac{b}{+\sqrt{m^2 - 1}}$. These results are unique.
\end{proof}

See Figure \ref{fig: rkks_min}. Lemma \ref{lemma: rkksmin} and Figure \ref{fig: rkks_min} provide a visual interpretation for Theorem \ref{thm: rkksmin} below.

\begin{theorem}\label{thm: rkksmin}
Given a RKKS $\mathcal{K}$ with nonsingular kernel $K = K_+ - K_-$, and interpolation data $\mathcal{A} = \{(x_1, \lambda_1), (x_2, \lambda_2), \dots, (x_N, \lambda_N)\}$, let $\mathcal{F}$ denote the affine space of functions in $\mathcal{K}$ which interpolate $\mathcal{A}$. If $\mathcal{F}$ does not contain the zero function, and $K_+ \neq c \pm K_-$ for any $c \in \mathbb{R}$, then it is orthogonal in the Krein sense to the space $\mathcal{H}_{\mathcal{A}} := \text{span}\{k_{x_j} \ | \ j = 1, 2, \dots, N\}$. Moreover, there is a unique function $h \in \mathcal{H}_{\mathcal{A}} \cap \mathcal{F}$.
\end{theorem}

\begin{proof}
If $f, g \in \mathcal{F}$, then $(f - g) \in \mathcal{F}^{\perp}$, because $(f - g)$ must interpolate zero at each point $x_j$. Also note that $[(f - g), k_{x_j}] = (f - g)(x_j) = 0$ for all $j = 1, 2, \dots, N$, so $(f - g) \in \mathcal{H}_{\mathcal{A}}^{\perp}$. Let $P$ denote the (nonsingular) matrix whose $i,j$ element is defined by $K(x_i, x_j)$, and consider two functions $h_1, h_2 \in \mathcal{H}_{\mathcal{A}} \cap \mathcal{F}$. In particular, let $h_1(z) = \sum\limits_{j=1}^{N}\alpha_j k_{x_j}(z)$ and $h_2(z) = \sum\limits_{j=1}^{N}\beta_j k_{x_j}(z)$, and denote $\boldsymbol{\alpha} = (\alpha_1, \alpha_2, \dots, \alpha_N)^T$, $\boldsymbol{\beta} = (\beta_1, \beta_2, \dots, \beta_N)^T$. Then, $(\boldsymbol{\alpha} - \boldsymbol{\beta})$ is in the kernel of $P$, and so $\boldsymbol{\alpha} = \boldsymbol{\beta}$ and $h_1 = h_2$. Therefore, there is a unique $h \in \mathcal{H}_{\mathcal{A}} \cap \mathcal{F}$, and because $\mathcal{H}_{\mathcal{A}}$ and $\mathcal{F}$ share the same orthogonal complement, they must be tangent at $h$.
\end{proof}

\begin{figure}[H]
    \centering
    \includegraphics[scale=.3]{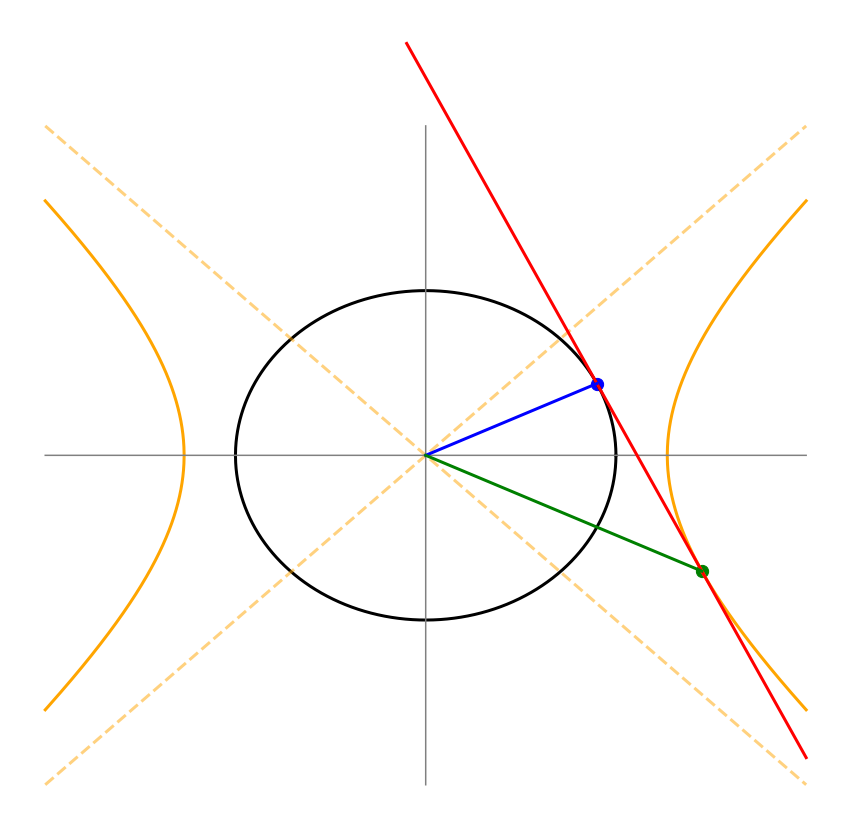}
    \caption{A depiction of the stationarity conditions of RKHS interpolation and RKKS interpolation. Using the terminology in Theorem \ref{thm: rkksmin}, the red line corresponds to $\mathcal{F}$, the green line corresponds to $\mathcal{H}_{\mathcal{A}}$, and the blue line corresponds to the span of associated Hilbert space kernels at points $\{x_j\}_{j=1}^{N}$. Here the level set corresponding to the stationarity condition given by interpolation in $\mathcal{H}_{\mathcal{A}}$ is depicted as a orange hyperbola; the level set corresponding to the stationarity condition given by interpolation in the associated Hilbert space is depicted as a black circle. Note the uniqueness of the stationarity conditions in this picture. This uniqueness holds across all interpolation problems which satisfy the conditions given in Theorem \ref{thm: rkksmin}.}
    \label{fig: rkks_min}
\end{figure}

\begin{remark}\label{remark: rkksreg}
Figure \ref{fig: rkks_min} makes it clear that performing interpolation in a RKKS with kernel $K = K_+ - K_-$ results in the unique interpolating function which is either minimal with respect to the norm induced by $K_+$ or maximal with respect to the Krein quadratic form induced by $-K_-$. This extremal property provides a method of regularizing the optimization problem outlined in Theorem \ref{thm: interpolation} (using kernel $\overset{\mathcal{H}}{K} = K_+ + K_-$).
\end{remark}

Remark \ref{remark: rkksreg}, together with Figure \ref{fig: rkks_min} and Theorem \ref{thm: rkksmin} are to what we refer when we relate RKKS interpolation and regularization.

\section{The Theory of Contractive Containment and Filling up the Gaps in a RKKS}

We now showcase the power of Theorem \ref{thm: assochilb} by showing that activation functions such as hyperbolic tangent and softplus generate function spaces which are strict subsets of the space induced by the $\omega-$kernel. 

\begin{theorem}[Aronszajn]\label{thm: aronszajn}
Let $\mathcal{H}_1, \mathcal{H}_2$ be RKHSs on $\mathcal{X}$ with respective reproducing kernels and norms $K_1, K_2$ and $||\cdot||_1, ||\cdot||_2$. If $K = K_1 + K_2$ and $\mathcal{H}$ is the induced RKHS with norm $||\cdot||$, then

\begin{equation}
    \mathcal{H} = \{f_1 + f_2 \ | \ f_1 \in \mathcal{H}_1, \ f_2 \in \mathcal{H}_2\}.
\end{equation}
Moreover, for $f \in \mathcal{H}$, 

\begin{equation}
    ||f||^2 = \min\{||f_1||_1^2 + ||f_2||_2^2 \ | \ f = f_1 + f_2, f_1 \in \mathcal{H}_1, \ f_2 \in \mathcal{H}_2\}
\end{equation}

\end{theorem}

\begin{proof}
See \cite{paulsen2016introduction}.
\end{proof}

\begin{defn}
Given two Hilbert spaces $\mathcal{H}_1, \mathcal{H}_2$ with norms $||\cdot||_1, ||\cdot||_2$, respectively, $\mathcal{H}_1$ is \textbf{contractively contained} in $\mathcal{H}_2$ if 

\begin{enumerate}
    \item $\mathcal{H}_1$ is a subspace of $\mathcal{H}_2$
    \item For every $f \in \mathcal{H}_1$, $||f||_2 \leq ||f||_1$
\end{enumerate}
\end{defn}

\begin{theorem}
Let $\mathcal{H}_1, \mathcal{H}_2$ be RKHSs over $\mathcal{X}$ with reproducing kernels $K_1, K_2$, respectively. Then $\mathcal{H}_1$ is contractively contained in $\mathcal{H}_2$ if and only if $K_2 - K_1$ is a kernel function.
\end{theorem}

\begin{proof}
See \cite{paulsen2016introduction}.
\end{proof}

In our case, the important point above is that $\mathcal{H}_1$ is a subspace of $\mathcal{H}_2$.

\begin{theorem}
Let $K(x, y) = \sum\limits_{j=0}^{\infty}a_j(y^*x)^j$ be a Krein kernel with $|a_k| < 1$ for all $k$ and at least one $a_k = 0$. The RKKS induced by $K$ is a proper subspace of the RKHS induced by the $\omega-$kernel.
\end{theorem}

\begin{proof}
Taking the associated Hilbert kernel of $K$, we can use RKHS machinery to determine the expressivity given by functions in the RKKS induced by $K$. Let

\begin{equation}
    \overset{\mathcal{H}}{K}(x, y) = \sum\limits_{n=0}^{\infty}|a_j|(y^*x)^j
\end{equation}
Assume without loss of generality that $a_l = 0$. Now, consider the kernel function

\begin{equation}
    \kappa(x, y) = \overset{\mathcal{H}}{K}(x, y) + (y^*x)^l
\end{equation}

Let $\mathcal{H}(\overset{\mathcal{H}}{K})$ denote the RKHS induced by $\overset{\mathcal{H}}{K}$, and let $\mathcal{H}(\kappa)$ denote the RKHS induced by $\kappa$. $\kappa$ is the sum of $\overset{\mathcal{H}}{K}(x, y)$ and another kernel function, and so every function in $\mathcal{H}(\overset{\mathcal{H}}{K})$ is also in $\mathcal{H}(\kappa)$ by Aronszajn's Theorem. Moreover, if $f \in \mathcal{H}(\overset{\mathcal{H}}{K})$, then again by Aronszajn's Theorem, we have $||f||_{\mathcal{H}(\kappa)} \leq ||f||_{\mathcal{H}(\overset{\mathcal{H}}{K})}$. So $\mathcal{H}(\overset{\mathcal{H}}{K})$ is contractively contained in $\mathcal{H}(\kappa)$. 

On the other hand, by Theorem \ref{thm: includedf}, we know that $f$ is in an arbitrary RKHS $\mathcal{H}$ with kernel $K$ if and only if there exists a constant $c$ such that

\begin{equation}
    c^2K(x, y) - f(x)\overline{f(y)}
\end{equation}
is a kernel function. So, we know that $f(x) = x^l$ is in the RKHS $\mathcal{H}(\kappa)$, because

\begin{equation}
    \kappa(x, y) - f(x)\overline{f(y)} = \sum\limits_{n=0}^{\infty}|a_j|(y^*x)^j
\end{equation}
has all nonnegative coefficients, and so is a kernel function. However, 

\begin{equation}
    c^2\overset{\mathcal{H}}{K}(x, y) - f(x)\overline{f(y)} = c^2|a_0| + c^2|a_1| (y^*x) + \dots + c^{l - 1}|a_{l - 1}| (y^*x)^{l - 1} + c^2|a_{l + 1}| (y^*x)^{l + 1} + \dots - (y^*x)^l
\end{equation}
is \textit{not} a kernel function for any $c$, and so $f$ is not in $\mathcal{H}(\overset{\mathcal{H}}{K})$.

Finally, $\mathcal{H}(\kappa)$ is contractively contained in the RKHS induced by the $\omega-$kernel, because

\begin{equation}
    \omega(x, y) - \kappa(x, y) = (1 - |a_0|) + (1 - |a_1|)y^*x + (1 - |a_2|)(y^*x)^2 + (1 - |a_3|)(y^*x)^3 + ...,
\end{equation}
which has nonnegative coefficients, and is therefore a kernel function. 
\end{proof}

\begin{corollary}\label{cor: inferior}
The RKKS induced by $K(x, y) = \tanh(y^*x)$ and by $K(x, y) = \log(1 + e^{y^*x})$ are proper subspaces of the RKHS induced by the $\omega-$kernel.
\end{corollary}

To show the power that \enquote{filling up the gaps} has (via Theorem \ref{thm: assochilb}), consider the plots in Figure \ref{fig: filled1} and \ref{fig: filled2}. The green and red lines included are identical to those given in Figures \ref{fig: tanh1} and \ref{fig: tanh2}. The purple line added is the result of including the power series coefficients which went unused in the Hilbert kernel associated to hyperbolic tangent. This simple manipulation improved performance by about one order of magnitude.

\begin{figure}[H]
\begin{subfigure}{.5\textwidth}
  \centering
  \includegraphics[width=.8\linewidth]{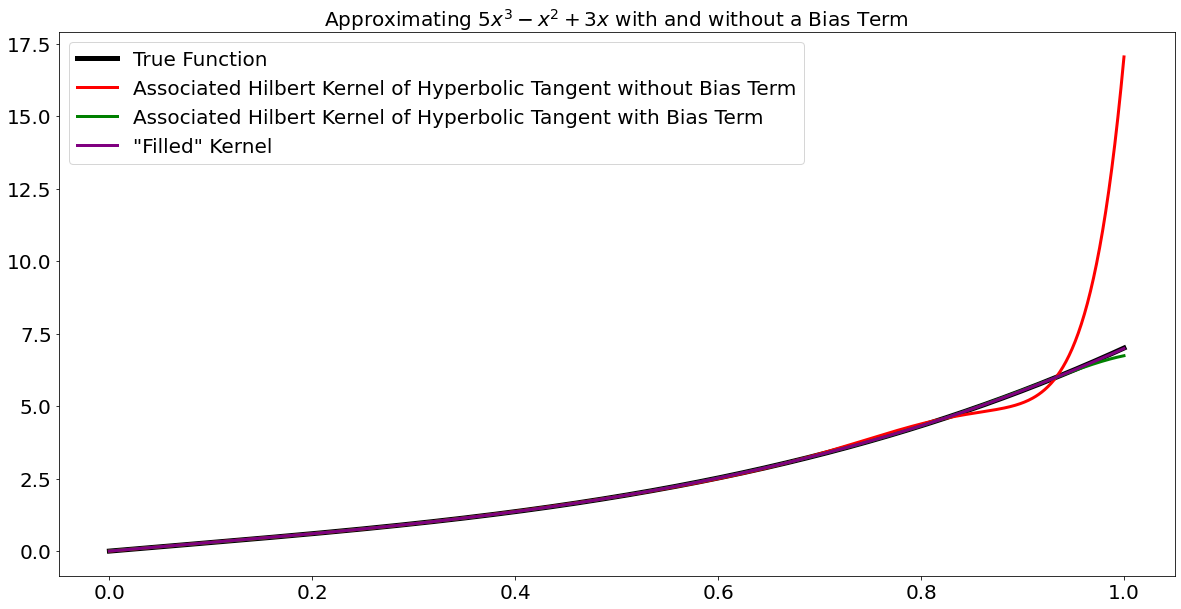}
  \caption{True function and interpolating estimates}
  \label{fig: filled1}
\end{subfigure}%
\begin{subfigure}{.5\textwidth}
  \centering
  \includegraphics[width=.8\linewidth]{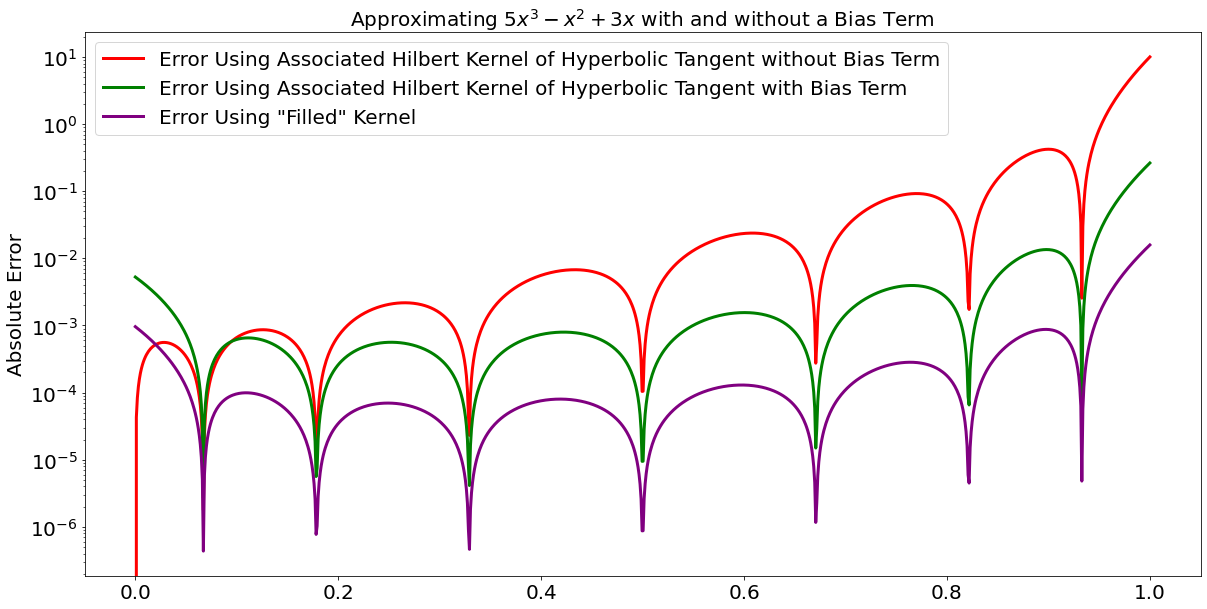}
  \caption{Error}
  \label{fig: filled2}
\end{subfigure}
\caption{Interpolating a polynomial at $7$ nodes using three methods. Note the simple act of \enquote{filling up the gaps} in the power series coefficients significantly improved performance.}
\label{fig:fig}
\end{figure}

\chapter{Applying RKHS to Training Neural Networks}\label{ch: marriage}

Consider the two-layer neural network 

\begin{equation}
    F(x) = A_2f(A_1x)
\end{equation}
and a training set $\mathcal{A} = \{(x_1, \lambda_1), (x_2, \lambda_2), \dots (x_n, \lambda_n)\}$. 

In general we want to find $A_1, A_2$, such that $F(x_j) = \lambda_j$ for all $j$. Moreover, we want our neural network to \textit{generalize}, i.e. perform well on some set $\{(x_{n + 1}, \lambda_{n + 1}), (x_{n + 2}, \lambda_{n + 2}), \dots (x_{n + m}, \lambda_{n + m})\}$ that we have not yet seen. There is in general no way of universally guaranteeing that one's neural network can generalize, unless one already knows the true, underlying function which creates such training and testing sets. Often, the best one can hope for out of a neural network is that it is \enquote{well-behaved} on some set $\mathcal{X}$ of inputs. Here, \enquote{well-behaved} can mean Lipschitz, bi-Lipschitz, $C^1$, $C^{\infty}$, bounded, etc. In this paper, we train train neural networks to be of minimal norm in an RKHS.

First, note that the resulting RKHS interpolant of Theorem \ref{thm: interpolation} is 

\begin{equation}
    F(z) = \sum\limits_{j=1}^{n}(\mathcal{P}^{-1}\mathbf{\lambda})_{j}k_{x_j}(z).
\end{equation}

Now, if $k_{x_j}(z)$ is of the form $k_{x_j}(z) = f(x_j^*z)$, then we get

\begin{equation}
    F(z) = \sum\limits_{j=1}^{n}(\mathcal{P}^{-1}\mathbf{\lambda})_{j}f(x_j^*z).
\end{equation}

If $f$ is elementwise, then we can simplify even further. Let $A_1 \in \mathbb{C}^{d \times n}$ denote the matrix whose columns are the training points $x_1, \dots, x_n$. Also, let $A_2 := (\mathcal{P}^{-1}\mathbf{\lambda})$. Then

\begin{equation}
     F(z) = A_2f(A_1^*z).
\end{equation}

Keep in mind that interpolation also works in a Krein space, so long as $\mathcal{P}$ is invertible.

\subsection{Bias Term}

Given the interpolation algorithm above, a constant term in an activation function's formal power series expansion corresponds to what neural network practitioners call a \textit{bias term}. If an activation function $f$ induces a kernel function of the form

\begin{equation}
    K(x, y) = f(y^*x),
\end{equation}
and $f$ has a power series expansion with a constant term, i.e.

\begin{equation}
    f(z) = c_0 + c_1 z + ... \qquad c_0 \neq 0,
\end{equation}
then $c_0$ above has the effect of inserting a $1$ into the first dimension of every input vector to our neural network.

Specifically, neural network practitioners often use the structure 

\begin{equation}
    F(x) = A_2 \begin{bmatrix}
    1 \\
    [f(A_1x)]_1 \\
    [f(A_1x)]_2 \\
    \vdots \\
    [f(A_1x)]_N
    \end{bmatrix}
\end{equation}
in their neural networks (note $A_2$ must be augmented for this to work). The machinery of RKHS developed above (in particular, Theorem \ref{thm: polynomials} and Figures \ref{fig: tanh1} and \ref{fig: tanh2}) yields rigorous reasons for doing so.

\subsection{Inferiority of Hyperbolic Tangent and ReLU Activations}\label{subsec: inferior}

The primary reason for introducing RKKS is many neural network practitioners use neural networks which fall under the heading of RKKS interpolants. Specifically, two classes of functions are widely used in practice, and fall into the RKKS framework: the Rectified Linear Unit (ReLU) $\mathrm{ReLU}(x) = \max\{0, x\}$, and sigmoidal functions, i.e. those functions $f$ such that $\lim\limits_{x \rightarrow -\infty} f(x) = -1$ (or $0$) and $\lim\limits_{x \rightarrow \infty} f(x) = 1$. The UAT of Cybenko \cite{cybenko1989approximation} and Perekrestenko et al. \cite{perekrestenko2018universal} have resulted in an attraction to these activations, respectively. In addition, the ReLU activation function comes with the perhaps advantageous property that its gradient is trivial to compute. %While the RKHS theory above is sufficient to mimic such universal approximation properties, it is relatively easy to extend our theory to ReLU and sigmoidal functions. In addition, we point the neural network community in a more promising direction by way of \textit{associated Hilbert spaces} of such Krein spaces.

In Agler and McCarthy \cite{agler2000complete}, it is shown that the Sobolev kernel on $[0, 1]$ is a complete Nevanlinna-Pick kernel, and so the expressivity of the piecewise linear activation ReLU is inferior to the expressivity of the $\omega-$kernel. Thus, in conjunction with Corollary \ref{cor: inferior},  we find the activation function $f(z) = \frac{1}{1 - z}$ is strictly more expressive than the hyperbolic tangent, softplus, and ReLU. 

It should be emphasized here that the activation $f(z) = \frac{1}{1 - z}$ is more expressive than $\tanh$ and softplus by contractive containment. However, in the context of ReLU, the superior expressivity does not follow from contractive containment (note piecewise linear functions are not in $H^2$). Here, the superior expressivity is due to the Agler-McCarthy theorem \cite{agler2000complete}. In this case, there exists a function $g: [0, 1] \rightarrow \mathbb{B}^{\aleph_0}$ and some nonvanishing function $\delta$ on $[0, 1]$ such that, for any function $f$ in the Sobolev RKHS, $(\delta \circ g^{-1})(f \circ g^{-1})$ is analytic on $\mathbb{B}^{\aleph_0}$. In particular, the Sobolev space is, after normalization, the restriction of the RKHS induced by the $\omega-$kernel on $\mathbb{B}^{\aleph_0}$ to a subspace spanned by a set of kernel functions \cite{agler2000complete}.

\section{Deep Learning}\label{sec: deeplearning}

For the sake of completeness, we extend our theory to that of deep learning by use of the Agler-McCarthy theorem and the activation function induced by the $\omega-$kernel. It should be noted that Gu has indirectly studied such a theory through the lens of \textit{Generalized Matrix-Valued Continued Fraction} (GMVCF) \cite{gu2001generalized}. In particular, an $N-$layer neural network with the activation $f(z) = \frac{1}{1 - z}$, is of the form

\begin{equation}\label{eq:continuedfraction}
    F(x) = f_{N}(A_{N}f_{N - 1}(A_{N - 1} \dots f_{1}(A_1 x))) = \frac{\textbf{e}}{\textbf{1} - A_{N} \frac{\textbf{e}}{\textbf{e} - A_{N - 1}\frac{\textbf{e}}{\ddots \frac{\textbf{e}}{\textbf{e} - A_{1} x}}}},
\end{equation}
where $\textbf{e}$ denotes the vector of ones $(1, 1, \dots, 1)^T$ and the division here is elementwise. The fraction in \eqref{eq:continuedfraction} is called a GMVCF. See \cite{gu2001generalized} for details. In general, the expressivity proofs of GMVCF follow the same arguments as those in the classical theory of continued fractions, e.g. \cite{wall2018analytic}. This is precisely because we are not using matrix inverses in the classical sense.

In general, however, the multi-dimensional analogue of GMVCF interpolation has not been studied. We create such a theory here. Note that our theory will fall into the class of training methods called \textit{Forward Thinking}, i.e. our method will train neural networks one layer at a time \cite{hettinger2017forward}.

The problem with iterating our interpolation algorithm via Theorem \ref{thm: interpolation} is that the result after each iteration may no longer lie inside the unit ball. We would like to \textit{scale} the data back to the unit ball after each iteration. The problem is that there does not exist a constant factor which reliably does so, as analytic functions need not be bounded on the $d-$sphere. The Agler-McCarthy Theorem, however, states that one need not be restricted to scaling by a constant. Indeed, scaling by non-vanishing functions $\delta(\cdot)$ is not a hindrance toward the expressivity of our neural network. The RKHS induced by such a rescaled kernel function is merely a rescaled copy of the original RKHS.

The problem now becomes finding a function $\delta(\cdot)$ such that

\begin{equation}\label{eq: contractiveke}
    |K(x, y)| = \Big|\frac{\overline{\delta(x)}\delta(y)}{1 - y^*x}\Big| < 1
\end{equation}
$\delta(c) = 1 - ||c||$ will do, by the following calculation:

\begin{multline}
    1 - y^*x = 1 - \langle y, x \rangle \geq 1 - ||x||||y|| \\
    > 1 - ||x|| > 1 - ||x|| - (||y|| - ||x||||y||)
\end{multline}

\eqref{eq: contractiveke} allows one to form deep neural networks without the need to perform some artificial scaling between layers.

\begin{remark}
The deep neural network resulting from iterating the kernel function above is not a feedforward neural network, because it does not rely solely on the term $y^*x$. Rather, it is a neural network with skip connections, as defined in \eqref{eq:skipconn}, i.e. each layer takes the previous layer as input, along with some other layer before that. In our case, the skip connections connect every other layer and the nonzero connections $\{B_j^{(k)}\}$ as defined in \eqref{eq:skipconn} are all equal to the identity $\mathbb{I}$.
\end{remark}

The Agler-McCarthy theorem implies the activation function 

\begin{equation}\label{eq:szego}
    f(z) = \frac{1}{1 - z}
\end{equation}
will be of supreme importance (note the division here is elementwise). It is the activation function associated with the $\omega-$kernel.

In one dimension this is the \textit{Szeg\"o kernel}. In higher dimensions however, the $\omega-$kernel is \textit{not} the multivariate Szeg\"o kernel.

The Agler-McCarthy theorem (Theorem \ref{thm: agler}) states that the RKHS associated with the $\omega-$kernel is \textit{universal} in the sense that every RKHS on the unit ball with a complete Nevanlinna-Pick kernel is a restriction of this space \cite{agler2000complete}.

There is a potential benefit to creating an interpolatory neural network with more layers than we have been using (two, because of the two weight matrices). In particular, an $n-$layer neural network of the form described above still does basic, RKHS interpolation, but now the kernel has been changed. Indeed, the first $n - 2$ layers act as a nonlinear transformation of the data $\phi: \mathbb{B}^m \rightarrow \mathbb{B}^N$ from one unit ball to another, where $m$ is the dimension of the input data and $N$ is the number of interpolation nodes. Note that $K\circ \phi$ is a kernel function, when $K$ is a kernel function (cf. Proposition 5.13 in \cite{paulsen2016introduction}). Depending on the new kernel and the analogous norm induced by this kernel, the resulting minimum norm interpolant may generalize better.  

%Take the simple example of learning a univariate function $f$ from $N$ interpolation nodes in the unit disk. In a two-layer network one is forced to perform interpolation in one-dimension, which can lead to numerical instability, especially when the interpolation nodes are relatively close together (see Section \ref{sec: w_vs_sb} for an example using the Segal-Bargmann kernel). On the other hand, iterating the Agler-McCarthy kernel and \textit{then} performing interpolation leads to solving an interpolation in $N$ dimensions at the last layer. This can give the neural network a more desirable expressivity than would be allowed by using only two layers. 

To make the above reflection concrete, consider the example of learning a sine curve using a two-layer network and a three-layer network (Figures \ref{fig: deepagler10} and \ref{fig: deepagler100}) and uniformly random interpolation nodes.

In Figure \ref{fig: deepagler10}, the most prominent phenomenon is the familiar thrashing nature of the red curve, generated by the two-layer neural network using the $\omega-$kernel. The green curve, however, lacks such behavior, because we have removed the poles near $\pm 1$. Indeed, let $\{y_j\}_{j=1}^{10}$ be our interpolation nodes. After applying the first layer of our neural network to $x=\pm 1$, we are left with

\begin{equation}
    \begin{bmatrix}
    \frac{(1 - ||y_1||)(1 - ||x||)}{1 - y_1^*x} \\
    \frac{(1 - ||y_2||)(1 - ||x||)}{1 - y_2^*x} \\
    \vdots \\
    \frac{(1 - ||y_N||)(1 - ||x||)}{1 - y_N^*x}
    \end{bmatrix} = \begin{bmatrix}
    \frac{(1 - ||y_1||)(1 - 1)}{1 - y_1^*1} \\
    \frac{(1 - ||y_2||)(1 - 1)}{1 - y_2^*1} \\
    \vdots \\
    \frac{(1 - ||y_N||)(1 - 1)}{1 - y_N^*1}
    \end{bmatrix} = \begin{bmatrix}
    0 \\
    0 \\
    \vdots \\
    0 \\
    \end{bmatrix}
\end{equation}

Therefore, using a $3-$layer network to interpolate at $\pm 1$ in this case numerically corresponds to performing RKHS at $0$. This is also why the green curve can withstand the problem of interpolating many more nodes than the red curve (Figure \ref{fig: deepagler100}).

While the interpolation algorithm will have no poles here (so long as the interpolation data $\{y_j\}$ have moduli less than one), this does show that the neural network's value at $x=-1$ will equal its value at $x=1$. Again, this is because the first layer of our $3-$layer neural network maps both $+1$ and $-1$ to $0$. This can lead to problems (e.g. Figure \ref{fig: deepaglerline}). 

The second phenomenon of note in Figure \ref{fig: deepagler10} is the \enquote{cusp} generated by the green curve near the origin. This is because poles are now allowed to be formed inside the unit disk. Again, let $\{y_j\}_{j=1}^{10}$ be our interpolation nodes. After applying the first layer of our neural network to $x=0$, we have

\begin{equation}
    \begin{bmatrix}
    \frac{(1 - ||y_1||)(1 - ||x||)}{1 - y_1^*x} \\
    \frac{(1 - ||y_2||)(1 - ||x||)}{1 - y_2^*x} \\
    \vdots \\
    \frac{(1 - ||y_N||)(1 - ||x||)}{1 - y_N^*x}
    \end{bmatrix} = \begin{bmatrix}
    \frac{(1 - ||y_1||)(1 - 0)}{1 - y_1^*0} \\
    \frac{(1 - ||y_2||)(1 - 0)}{1 - y_2^*0} \\
    \vdots \\
    \frac{(1 - ||y_N||)(1 - 0)}{1 - y_N^*0}
    \end{bmatrix} = \begin{bmatrix}
    (1 - ||y_1||) \\
    (1 - ||y_2||) \\
    \vdots \\
    (1 - ||y_N||) \\
    \end{bmatrix}
\end{equation}

Of course, this arithmetic shows that we need an additional normalizing term in our kernel function, namely $\frac{1}{\sqrt{N}}$. More importantly, it shows us that interpolating data near zero in $3$ layers is tantamount to interpolating data near the boundary of the unit disk in $2$ layers. 

\begin{figure}[H]
\begin{subfigure}{.5\textwidth}
  \centering
  \includegraphics[width=.8\linewidth]{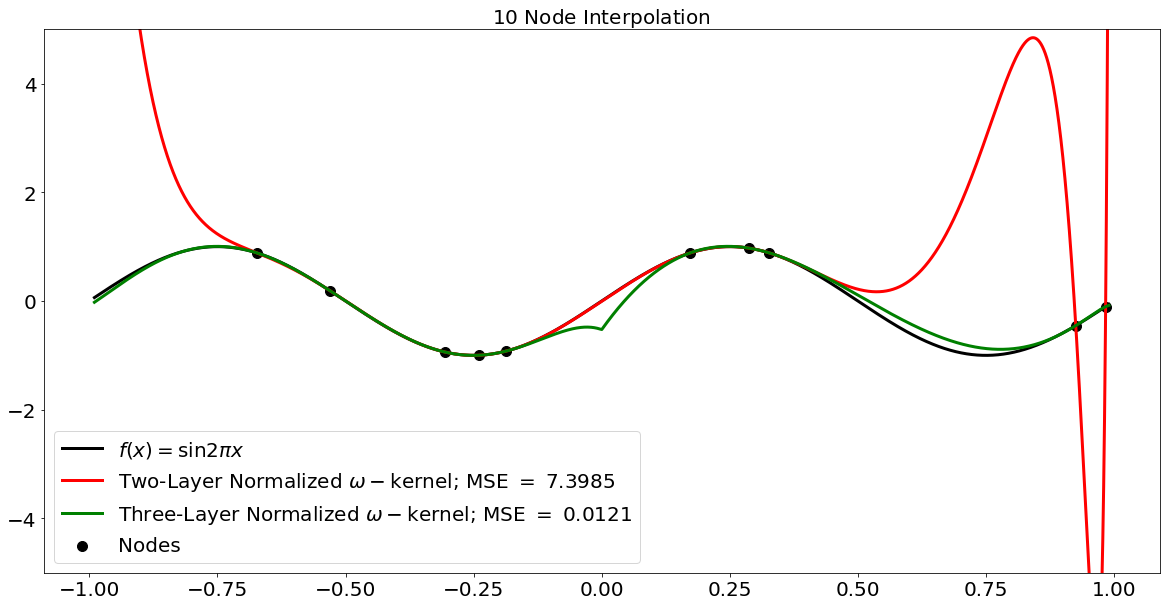}
  \caption{$10$ Training Nodes}
  \label{fig: deepagler10}
\end{subfigure}%
\begin{subfigure}{.5\textwidth}
  \centering
  \includegraphics[width=.8\linewidth]{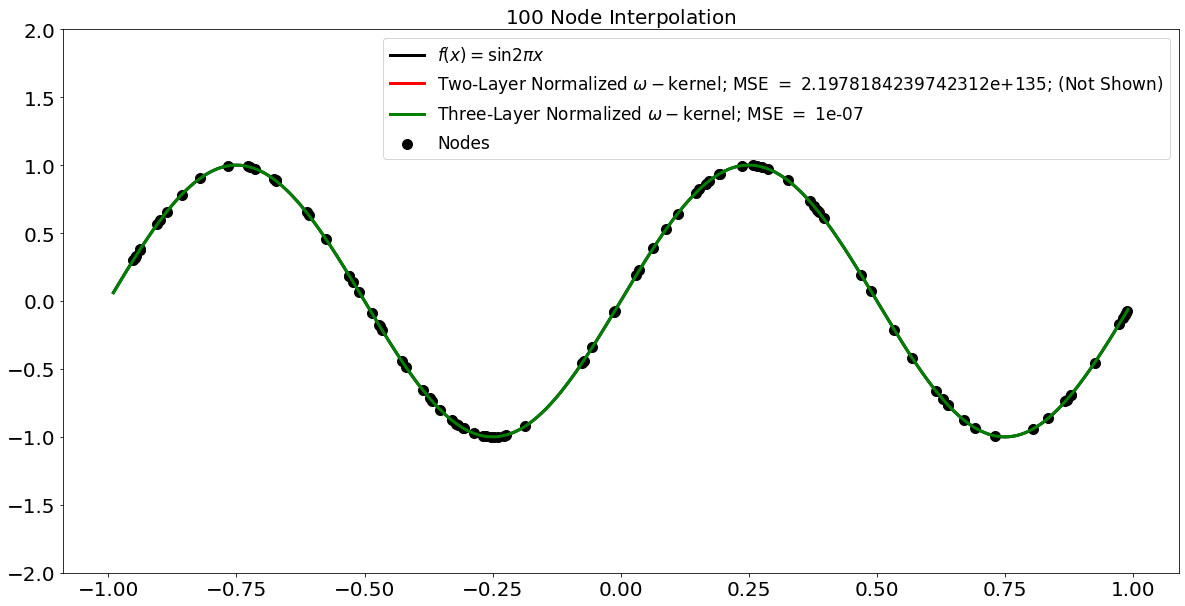}
  \caption{$100$ Training Nodes}
  \label{fig: deepagler100}
\end{subfigure}
\caption{Adding another layer to our neural network can allow our neural network more expressivity, as one can see by the \enquote{cusp} made by the green curve in \ref{fig: deepagler10}. The better conditioning of such a problem means we have the added benefit of introducing more interpolation nodes without numerical blow-up.}
\label{fig:fig}
\end{figure}

\begin{figure}[H]
    \centering
    \includegraphics[scale=.3]{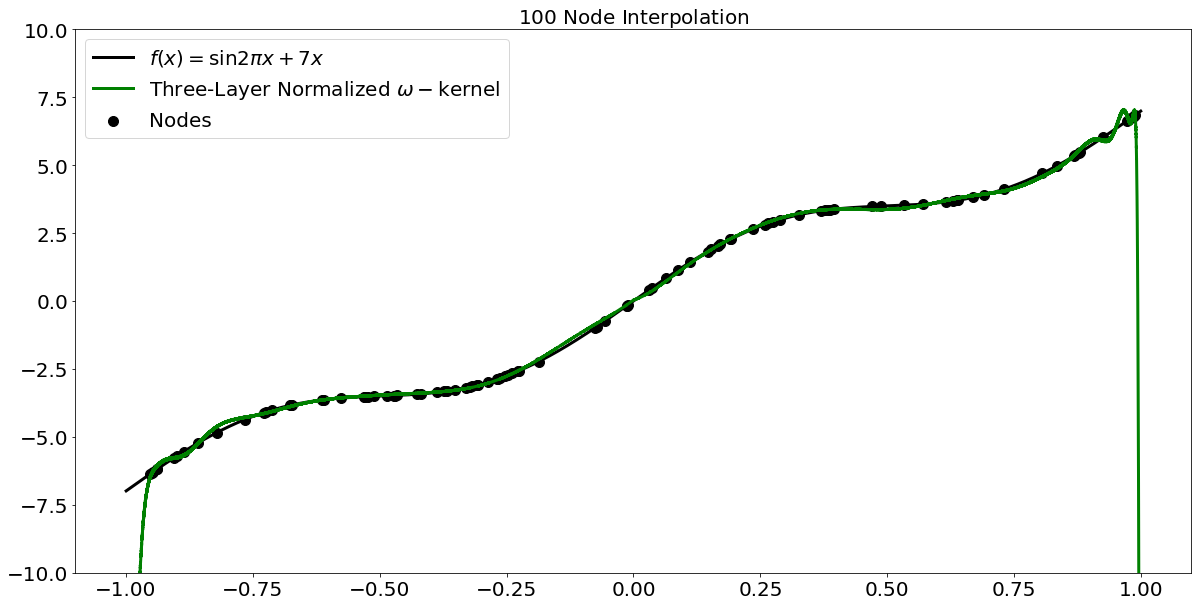}
    \caption{Here, we are attempting to learn the function $f(x) = \sin 2 \pi x + 7x$ via a $3-$layer neural network. Note the green curve shown here exhibits strange behavior near the points $x=\pm 1$, because the neural network corresponding to the green curve needs to take on the same value at both these points. This is because the first layer of our neural network maps $+1$ and $-1$ to the same point ($0$).}
    \label{fig: deepaglerline}
\end{figure}

\section{A Stable, Accurate Method for Obtaining ReLU Interpolants}\label{sec: relu}

While the expressivity of ReLU is inferior to that of the activation function $f(z) = \frac{1}{1 - z}$, as stated in Subsection \ref{subsec: inferior}, it is quite easy to train a ReLU network using interpolation, so we include a method of doing so here. Consider the function $\mathfrak{K}: [0, 1]^d \times [0, 1]^d \rightarrow \mathbb{R}$

\begin{equation}\label{eq: linearkernel}
    \mathfrak{K}(x, y) = \min\{e^Tx, e^Ty\} - y^Tx
\end{equation}
This is a natural, multidimensional analogue of the kernel function on the Sobolev space 
\begin{equation} \quad
    \mathcal{H} = \{f: [0, 1] \rightarrow \mathbb{R} \ \Big| \ f \text{ is absolutely continuous, } f(0)=f(1)=0, \ f' \in L^2[0, 1]\}
\end{equation}

whose kernel function is 

\begin{equation}
    K(x, y) = \begin{cases}
    (1 - y)x, \qquad x \leq y \\
    (1 - x)y, \qquad x \geq y
    \end{cases}
\end{equation}

The multidimensional analogue is indeed a kernel function, because 

\begin{equation}
    e^Tx_j \geq x_i^Tx_j, \qquad e^Tx_i \geq x_i^Tx_j
\end{equation}
for all $x_j, x_j \in [0, 1]^d$, and we merely need the matrix whose elements are $K(x_i, x_j)$ to be nonnegative semidefinite. Note that the neural network associated with such a kernel function will have skip connections, as defined in \eqref{eq:skipconn}, which connect every other layer. The nonzero skip connections $\{B_j^{(k)}\}$ will be equal to the identity $\mathbb{I}$.

The kernel function above has the property that the interpolants it produces in the context of Theorem \ref{thm: interpolation} are piecewise linear, where in higher dimensions \enquote{piecewise linear} is taken to mean the graph is a collection of sections of hyperplanes. Note that this is a result of the stationarity condition of the interpolation problem in the Sobolev space, as the space does \textit{not} consist only of piecewise linear functions. %two properties in the context of Theorem \ref{thm: interpolation} which together make it a function of interest:

%\begin{enumerate}
%    \item $\mathfrak{K}$ fits into the framework of our interpolation algorithm given by Theorem \ref{thm: interpolation}.
%    \item The interpolants $\mathfrak{K}$ produces are piecewise linear, where in higher dimensions "piecewise linear" is taken to mean the graph is a collection of subsets of hyperplanes.
%\end{enumerate}

Considering ReLU neural networks are merely piecewise linear functions, this means we can essentially generate interpolating ReLU networks without using descent methods. Approximation by piecewise linear functions has already been studied in great detail, and a digression into the theory would take us too far afield. In particular, Perekrestenko's paper on the universality of ReLU is quite applicable here \cite{perekrestenko2018universal}. See also \cite{malozemov2018examples} for uniform approximation results in the one-dimensional case.

\begin{remark}
Because $\mathfrak{K}$ does not depend solely on the inner product $y^*x$, the resulting interpolants of Theorem \ref{thm: interpolation} are not feedforward neural networks (they are neural networks with skip connections). This is why we denote the kernel by the fraktur $\mathfrak{K}$, so as to distinguish it from the rest of the kernels we study.
\end{remark}

\begin{remark}
As stated in \cite{agler2000complete}, it was proven in \cite{agler1990nevanlinna} that the Sobolev space in one dimension has a complete Nevanlinna-Pick kernel and so has expressivity no greater than the RKHS induced by the $\omega-$kernel.
\end{remark}

To see the power of $\mathfrak{K}-$interpolants, consider the problem of approximating $\phi$ given by \eqref{eq: homotopy}.

\begin{figure}[H]
\begin{subfigure}{.5\textwidth}
  \centering
  \includegraphics[width=.8\linewidth]{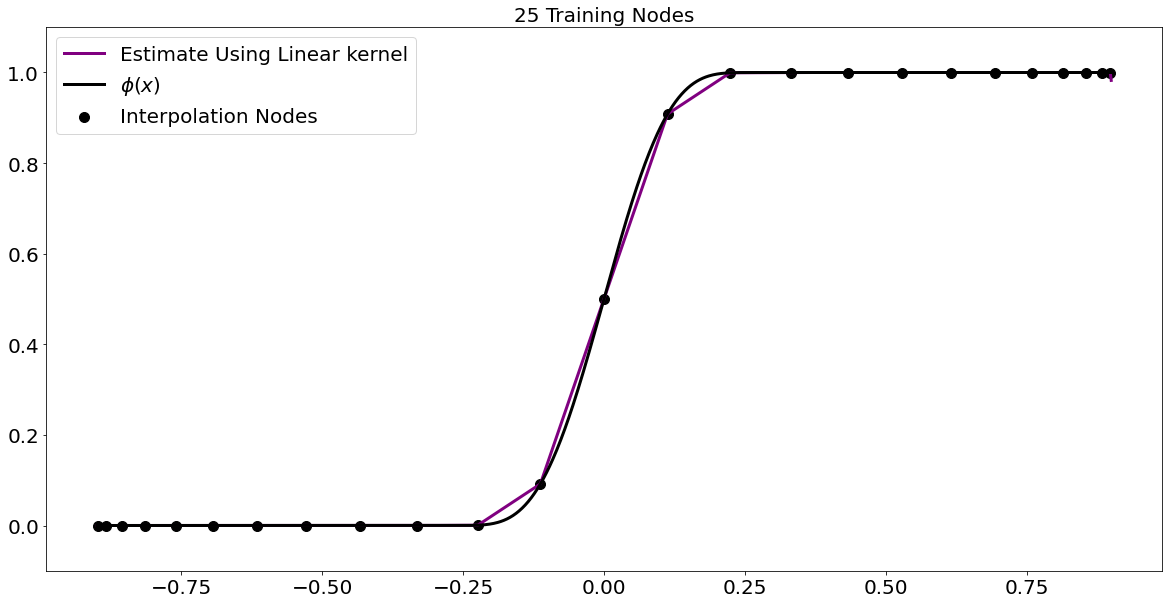}
  \caption{$10$ Training Nodes}
  \label{fig: homotopylinear25}
\end{subfigure}%
\begin{subfigure}{.5\textwidth}
  \centering
  \includegraphics[width=.8\linewidth]{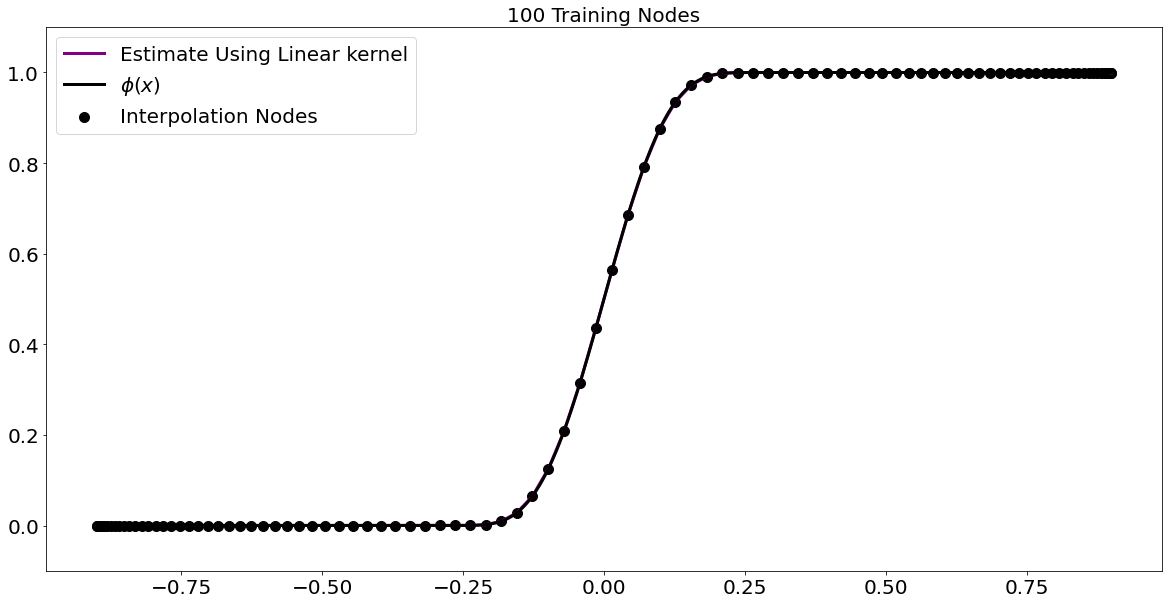}
  \caption{$100$ Training Nodes}
  \label{fig: homotopylinear100}
\end{subfigure}
\caption{The $\mathfrak{K}-$approximation of $\phi$ is far more stable under the addition of more interpolation nodes. Here, we see that, even when given $100$ interpolation nodes, the $\mathfrak{K}-$interpolant is not prone to numerical instability.}
\label{fig:fig}
\end{figure}

Note that the $\mathfrak{K}-$interpolant remains quite stable, even as the number of interpolation nodes significantly increases, as the illustrations in Figures \ref{fig: homotopylinear25} and \ref{fig: homotopylinear100} demonstrate.

Let us consider another example, which is perhaps even more difficult to learn by other methods.

\begin{figure}[H]
    \centering
    \includegraphics[scale=.3]{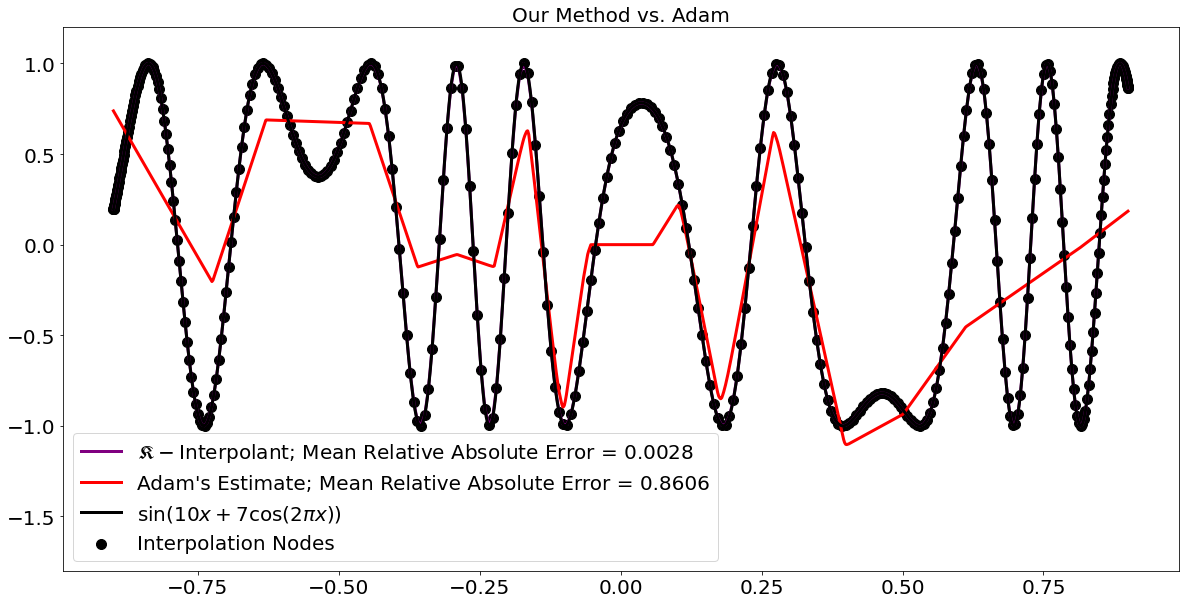}
    \caption{The $\mathfrak{K}-$interpolation algorithm and the Adam algorithm were both trained on the same dataset of $400$ Chebyshev nodes, shifted to be on the interval $[-.9, .9]$. Adam is having a significant amount of difficulty learning our function, because each random batch (of size $64$) per gradient step will have a different behavior.}
    \label{fig: linearkernelvsadam}
\end{figure}

The function 

\begin{equation}
    f(x) = \sin(10x + 7\cos(2\pi x))
\end{equation}
has an ever-changing frequency and so should be difficult to learn by descent methods, even with batch learning, as any two random batches one takes as training examples will have significantly different behaviors. We trained our $\mathfrak{K}-$interpolant and a classical ReLU network on the same $400$, shifted Chebyshev nodes in $[-.9, .9]$. We used the Adam optimizer with a learning rate of $10^{-2}$ (the best we saw) and batches of size $64$ over $10,000$ epochs to train the ReLU network \cite{adam}. We used mean square error as our loss function. In Figure \ref{fig: linearkernelvsadam} we see that Adam was unable to train the classical neural network to the same accuracy as our $\mathfrak{K}-$interpolation algorithm did. Moreover, the training time for Adam was $51.2$ seconds, whereas the training time for the $\mathfrak{K}-$interpolation algorithm was $1.62$ seconds.

\chapter{Cotlar-Sadosky Theory}\label{ch: cs}

\section{Background: The AAK Theorem}

In approximation theory, it is well known that one can over-interpolate data. There are two notable reasons for which interpolation can fail. First, the given interpolation nodes may be undesirable (e.g evenly spaced under the Euclidean norm). Second, the basis of interpolating functions may be chosen poorly. These problems are ubiquitous in the theory of interpolation. In the presence of noise, these problems are amplified, because interpolating noisy data points corresponds to interpreting noise as a signal. 

% \begin{figure}[H]
%     \centering
%     \includegraphics[scale=.3]{Thrashing.png}
%     \caption{Lagrange interpolation is prone to "thrashing" when fit on either too much data, or poorly chosen nodes.}
%     \label{fig:thrashing}
% \end{figure}

We have already introduced methods of regularizing our interpolants to combat these issues. In particular, RKHS interpolation results in functions of minimum norm; one can regularize further by performing interpolation in a reproducing kernel Krein space. These methods will tend to reduce \textit{thrashing}, or adverse behavior between interpolation nodes. On the other hand, there are techniques other than regularization which can further improve performance, particularly in noisy environments. Among such methods are those of \textit{model reduction}. These reduce the intrinsic complexity of one's model in the hope that such complexity was fitting more noise than signal.

To develop the novel theory of model reduction on neural networks, we first introduce the celebrated theorem of Adamjan, Arov, and Krein (AAK theorem) \cite{adamyan1971analytic}, as well as a multidimensional version given by Cotlar and Sadosky \cite{cotlar1996two}. We will then prove our marginal AAK theorem, which will be directly applicable to neural networks. See \cite{chui1997discrete}, Chapter 4, Section 3, for an exposition on the classical AAK theory.

\begin{defn}
If $f(z) = \sum\limits_{n=-\infty}^{\infty}f_{n}z^{-n}$, then the \textbf{Hankel operator}, with symbol $f$, $\Gamma_f$ is 

\begin{equation}
    \Gamma_f = \begin{bmatrix}
    f_{1} & f_{2} & f_{3} & \dots \\
    f_{2} & f_{3} & f_{4} & \dots \\
    f_{3} & f_{4} & f_{5} & \dots \\
    \vdots
    \end{bmatrix}
\end{equation}
and the \textbf{Hankel norm} of $f$ (or $\Gamma_f$) is given by

\begin{equation}
    ||f||_{\Gamma} = \sup_{||x||=1}||\Gamma x||_{\ell^{2}}
\end{equation}
\end{defn}

\begin{theorem}[Nehari]
Let $f \in L^{\infty}(\mathbb{T})$. Then 

\begin{equation}
    ||f||_{\Gamma} = \inf_{g \in H^{\infty}}||f - g||_{L^{\infty}}
\end{equation}
\end{theorem}

Nehari's theorem implies many important facts. First, the Hankel norm is invariant under addition by $H^{\infty}$ functions. Second, Nehari's theorem shows that $f \in L^{\infty}(\mathbb{T})$ implies $f$ has finite Hankel norm. Finally, we see that a function $h_s(z) = \sum\limits_{n=1}^{\infty}h_{n}z^{-n}$ has finite Hankel norm if and only if there exists an analytic function $h_a(z) = \sum\limits_{n=0}^{\infty}h_{-n}z^{n}$ such that $h_z(z) + h_s(z) \in L^{\infty}(\mathbb{T})$.

% Note $f(z) = \sum\limits_{n=-N}^{N}a_nz^n$ has finite Hankel norm because it results in a finite-rank matrix. On the other hand, $f(z) = \sum\limits_{n=-\infty}^{\infty}nz^n$ does \textit{not} have finite Hankel norm, because we can take $x = \frac{6}{\pi^2}(1, 1/2, 1/3, 1/4, 1/5, \dots)^T$, and then 

% \begin{equation}
%     -\frac{6}{\pi^{2}}\begin{bmatrix}
%     1 & 2 & 3 & \dots \\
%     2 & 3 & 4 & \dots \\
%     3 & 4 & 5 & \dots \\
%     \vdots
%     \end{bmatrix} \begin{bmatrix}
%     1 \\
%     1/2 \\ 
%     1/3 \\
%     \vdots
%     \end{bmatrix} = \begin{bmatrix}
%     \infty \\
%     \infty \\
%     \infty \\
%     \vdots
%     \end{bmatrix}
% \end{equation}

The following important corollary shows us not only that convergence in $L^{\infty}(\mathbb{T})$ implies convergence in Hankel norm, but also that convergence in Hankel norm implies convergence in $L^{2}(\mathbb{T})$. For notational convenience, we will begin using $L^{\infty}$ in place of $L^{\infty}(\mathbb{T})$.

\begin{corollary}[Nehari]
Let $f(z) = \sum\limits_{n=1}^{\infty}$ be in $L^{\infty}$. Then 

\begin{equation}
    ||f||_{L^2} \leq ||f||_{\Gamma} \leq ||f||_{L^{\infty}}.
\end{equation}
\end{corollary}

The first inequality is easy to prove:

\begin{equation}
    ||f||_{L^2} = (\sum\limits_{n=1}^{\infty}|f_n|^2)^{1/2} = ||\Gamma e_1||_{2} \leq ||f||_{\Gamma}.
\end{equation}

The second inequality follows from Nehari's theorem:

\begin{equation}
    ||f||_{\Gamma} \leq ||f - 0||_{L^{\infty}}.
\end{equation}

Neharis's theorem is interesting per se, but the AAK theorem is more general and proves to be more applicable in practice.

To begin, we write the singular value decomposition of a Hankel operator by $\Gamma_f = U\Sigma V^*$. Note in this case, $U$ is equal to $V$, up to a sign change: $u^{(i)} = \mathrm{sgn}(\lambda_i)v^{(i)}$, where $\lambda_i$ is the $i^{th}$ eigenvalue (ordered by \textit{magnitude}) of $\Gamma_f$. Let 

\begin{equation}
    v_{+}(z) = \sum\limits_{j=1}^{\infty}v^{(m + 1)}_j z^{j-1}
\end{equation}
and

\begin{equation}
    u_{-}(z) = \sum\limits_{j=1}^{\infty}u^{(m + 1)}_j z^{-j}
\end{equation}
Finally, let 

\begin{equation}
    \mathcal{R}_m^s = \{r(z) = \frac{p_1z^{n-1} + \dots + p_k}{z^m + q_1z^{m-1} + \dots + q_m} | r \text{ has all its zeros in the open unit disk and } n \leq m\}
\end{equation}

The AAK theorem is largely a result about approximating $L^{\infty}$ functions by strictly proper, rational functions of given order. Perhaps the most surprising consequence of the AAK theorem is that we can \textit{analytically construct} the rational functions of best approximation under the Hankel norm. Again, because the Hankel norm is \enquote{sandwiched} between the $L^2$ norm and the $L^{\infty}$ norm, such a minimal result is quite desirable. In fact, it only takes a few lines of code to actually implement.

\begin{theorem}
Let $f \in L^{\infty}$ with a compact corresponding Hankel operator $\Gamma_f = U\Sigma V^*$. Then

\begin{equation}
    ||f - \hat{r}_m||_{\Gamma} = \inf_{r \in \mathcal{R}_m^s}||f - r||_{\Gamma} = s_{m + 1},
\end{equation}
where $\hat{r} = [\hat{h(z)}]_s$ and 

\begin{equation}
    \hat{h}(z) = f(z) - s_{m + 1}\frac{u_-(z)}{v_+(z)}
\end{equation}
\end{theorem}

There are a number of results in the AAK theorem which may/should be surprising. First, $f(z) - s_{m + 1}\frac{u_-(z)}{v_+(z)}$ is a rational function. This is by no means obvious. Second, the minimum norm of the resulting error is $s_{m + 1}$. This might seem along the same lines as the Eckhart-Young theorem, but in this case, we are constraining ourselves to Hankel matrices rather than arbitrary matrices. Nevertheless, the similarity of the two results could help as a mnemonic.

\section{The Need for Several Complex Variables}\label{sec:scv}

The discussion given in the previous section creates a general guideline for the method of model reduction we wish to create. Specifically, we want to take our neural network, and reduce its model complexity (analogous to going from $\mathcal{R}^s_m$ to $\mathcal{R}^s_{m - 1}$), while keeping the distance from the original neural network minimal. Unfortunately, as the reader has perhaps noticed, the AAK theorem only fits into the context of transfer functions which are functions of a single complex variable. Here however, neural networks most often have multidimensional input. From the point of view of signal processing, this corresponds to a transfer function, where each of its inputs corresponds to a different shift, which presents an issue. The reason why scalar inputs seem to be required is because \textit{there is no direct analog to the AAK theorem in the context of several complex variables}. Two of the major reasons are given in the table below (Note $d$ refers to the (complex) dimension of the input).

\begin{center}
\begin{tabular}{ | m{17em} | m{17em} | } 
  \hline
  $d = 1$ & $d > 1$ \\ 
  \hline
  $\Gamma_f$ bounded $\iff$ \\
  $f = g + h, \quad g \in H^2, h \in L^{\infty}$ & There exists $f \in L^2$, such that $\Gamma_f$ is bounded on $H^2$, but no decomposition $f = g + h$ as shown to the left is possible. \\ 
  \hline
  $\Gamma_f$ extends to a compact operator on $H^2 \iff$ $f = g + h, \quad g \in H^2$ and $h$ is continuous & There exist no compact, nonzero Hankel operators on $H^2$.  \\ 
  \hline
\end{tabular}
\end{center}

See \cite{cotlar1994nehari}, \cite{cotlar1996two}, \cite{ahern2009compactness}, and \cite{ferguson2000characterizations} for details. The lower right box is perhaps the most disconcerting. Specifically, this means that there exist no Hankel operators with a finite number of positive singular values when the dimension $d$ is greater than one. However, there is a framework in which a multidimensional generalization of the AAK theorem is available to us, called \textit{restricted BMO}, where here, \enquote{BMO} stands for \enquote{Bounded Mean Oscillation}.

\section{Restricted BMO}\label{sec:bmor}

Because of the bizarre properties of Hankel operators in the context of several complex variables, Cotlar and Sadosky propose using an alternate route to discovering a multidimensional AAK theorem \cite{cotlar1996two}. Their point of departure is via the space of functions of bounded mean oscillation, or BMO, which consists of functions $f$ such that

\begin{equation}
    \sup\limits_{I \subseteq \mathbb{T}}\Bigl\{ \frac{1}{|I|}\int\limits_{I}\Big|f(z) - \frac{1}{|I|}\int\limits_{I}|f(\zeta)|d\zeta\Big|dz\Bigr\} < \infty
\end{equation}

Because $\Gamma_f$ is bounded in Hankel norm if and only if $f \in \text{BMO}$, it makes sense to draw this connection.

In one dimension there are many equivalent definitions of BMO, which prove useful in developing the classical AAK theorem. The problem in the several complex variables case is that these definitions are no longer equivalent and it is not obvious which definition one should work with in order to make the proper, generalized theory.

In particular, Charles Fefferman proved BMO is the dual space to $H^1$, thus yielding another definition of BMO \cite{fefferman1971characterizations}. Moreover, BMO can also be characterized by

\begin{equation}
    \mathrm{BMO} = L^{\infty} + HL^{\infty},
\end{equation}
where $H$ is the Hilbert transform

\begin{equation}
    Hf(\theta) = \lim\limits_{\varepsilon \rightarrow 0}\frac{1}{\pi}\int\limits_{\varepsilon < |\theta - \phi| < \pi} \frac{f(\phi)}{\theta - \phi}d\phi.
\end{equation}
This is simply the principal value of the convolution of $f$ with $\frac{1}{\pi t}$. Yet another definition of BMO is given by

\begin{equation}\label{eq: bmor1d}
    f \in \text{BMO} \iff \exists \phi_0, \phi_1 \in L^{\infty} \text{ s.t. } (\mathbb{I} - P)f = (\mathbb{I} - P)\phi_1, \ Pf = P\phi_0,
\end{equation}
where $P$ is the analytic projection operator.

The definition given in \eqref{eq: bmor1d} turns out to be the one which provides a way forward to a multidimensional version of the AAK theorem. Note also that this definition most easily generalizes the notion of multiple shifts in the signal processing perspective. The corresponding space however, will not be called BMO; this title is reserved for the dual space of $H^1$.
% In several complex variables, $\mathrm{BMO}(\mathbb{T}^d)$ (also called \textbf{product BMO}) is the dual space to $H^1(\mathbb{T}^d)$. Product BMO is characterized by

% \begin{multline}\label{eq: productBMO}
%     f \in \mathrm{BMO}(\mathbb{T}^d) \iff f = f_1 + H_{x_1}f_2 + H_{x_2}f_3 + \dots + H_{x_d}f_{d + 1} \\
%     + H_{x_1}H_{x_2}f_{d + 2} + \dots + H_{x_1}H_{x_2} + \dots + H_{x_d}f_{2^d} \\
%     f_1, f_2, \dots f_{2^d} \in L^{\infty}(\mathbb{T}^d),
% \end{multline}

% where $H_{x_j}$ is the Hilbert transform with respect to $x_j$ for all $j$. 

% $\mathrm{BMO}(\mathbb{T}^d)$ is actually a normed space:

% \begin{equation}
%     ||f||_{\mathrm{BMO}} := \inf\Big\{\max\limits_j ||f_j||_{\infty} \ \Big| \ \text{all decompositions given by Equation \ref{eq: productBMO}}\Big\}
% \end{equation}

% For $d = 1$, BMO and product BMO are clearly equal; however there are several adversarial aspects for $d > 1$ which make product BMO not very useful. The multidimensional generalization we shall use, which corresponds to \eqref{eq: bmor1d}, is given below.

\begin{defn}
A function $f$ on $\mathbb{T}^d$ is in \textbf{restricted BMO} or \textbf{BMOr} if there exist $f_0, f_1, \dots, f_d$ such that the one-dimensional, antianalytic projections of $f$ are the antianalytic projections of $f_j \in L^{\infty}(\mathbb{T}^d)$, and the strictly analytic part of $f$ is the strict analytic part of $f_0 \in L^{\infty}(\mathbb{T}^d)$:

\begin{align}\label{eq: bmor}
    (\mathbb{I} - P_{x_j})f = (\mathbb{I} - P_{x_j})f_j \qquad j = 1, 2, \dots, d \\
    (P_{x_1}P_{x_2}\dots P_{x_d})f = f_0
\end{align}
\end{defn}

BMOr is a Banach space, with norm

\begin{equation}
    ||f||_{\mathrm{BMOr}} = \inf\Big\{\max\limits_{j}||f_j||_{\infty} \Big| \text{all such decompositions in Equation \ref{eq: bmor}} \Big\}
\end{equation}
The norm also has perhaps the more illuminating form

\begin{equation}
    ||f||_{\mathrm{BMOr}} = \max\Big\{\max\limits_{0 \leq j \leq d}\{\inf\{||f - h_{x_j}||_{\infty} \ \big| \ h_{x_j} \in H^2_{x_j} \}  \}, \inf\{||f - h_{\perp}||_{\infty} \ \big| \ h_{\perp} \in H^{2\perp}\} \Big\}.
\end{equation}

%The norm is basically asking how far you are from the closest function which is either analytic in at least one element, or antianalytic in all its elements (and is of course square integrable), depending on whichever is farthest.

The multidimensional Nehari theorem, given below, implies bounded Hankel operators have bounded symbols under the BMOr norm. This is the paramount reason why introducing BMOr is required for model reduction.

\begin{theorem}[Nehari]\label{thm: nehari}
Let $\Gamma: H^2(\mathbb{T}^d) \cap \mathcal{P} \rightarrow H^2(\mathbb{T}^{d})^{\perp}$ be a Hankel operator (here, $\mathcal{P}$ are the trigonometric polynomials). $\Gamma$ is bounded if and only if there exist $f \in \mathrm{BMOr}$ such that $\Gamma_f \phi= (\mathbb{I} - P)(f\phi)$, for all $\phi \in H^2(\mathbb{T}^d)$, Moreover, 

\begin{equation}
    ||f||_{\mathrm{BMOr}} \leq ||\Gamma|| \leq \sqrt{d}||f||_{\mathrm{BMOr}}
\end{equation}
\end{theorem}

In $d=1$, we have $||f||_{\mathrm{BMOr}} = ||f||_{\mathrm{BMO}} = ||\Gamma||$. However, for $d > 1$, there exist Hankel operators with unbounded symbols under the product BMO norm. Moreover, every compact Hankel operator for $d > 1$ is identically equal to zero. This therefore suggests BMOr is the correct space with which to deal in more than one dimension. In fact using BMOr, we do get a multivariable generalization of the AAK theorem in the context of several complex variables.

\section{The Cotlar-Sadosky Theorem and its Consequences}\label{sec: cs}

\begin{defn}
The $\sigma-$\textit{numbers} of a Hankel operator $\Gamma$ with corresponding symbol $\phi$ is

\begin{multline}
    \sigma_n(\Gamma_{\phi}) := \inf\{||\Gamma_{b\phi}|| \ : \ b = b_1 \otimes \dots \otimes b_d, \\
    \text{with } b_j \text{ a one-dimensional Blaschke product with at most } n_j \text{ factors}\}
\end{multline}
\end{defn}

The $\sigma$-numbers are the BMOr analogue of \textit{Schmidt numbers} (precisely the singular values of the corresponding Hankel matrix) from the AAK theory \cite{chui1997discrete}. We are required to introduce $\sigma-$numbers precisely because of the fact that the only compact Hankel operator (one in which singular values can go to zero) is the zero Hankel operator for dimension greater than one. Note that $\sigma-$numbers correspond to singular values in one dimension.

Define the distance $\delta(\cdot, \cdot)$ by

\begin{equation}
    \delta(f, \mathrm{BMOA} + \mathcal{R}_n) = \inf\{||bf - h||_{\mathrm{BMOr}} : \quad h \in \mathrm{BMOA}, b = b_1 \otimes \dots \otimes b_n\},
\end{equation}
where $b_j$ is a one-dimensional Blaschke product with at most $n_j$ factors and $\mathcal{R}_n$ denotes the set of multivariate rational functions with all its poles in the unit polydisk. 

\begin{theorem}[Cotlar-Sadosky AAK-Theorem (CSAAK Theorem)]

\begin{equation}\label{eq: csaak}
    \frac{1}{\sqrt{d}}\sigma_n(\Gamma_f) \leq \delta(f, \mathrm{BMOA} + \mathcal{R}_n) \leq \sigma_n(\Gamma_f)
\end{equation}
\end{theorem}

% \begin{corollary}
% If $\prod\limits_{j=1}^{d}f_j = f \in L^{\infty}$ is the product resulting from the Cotlar-Sadosky Interpolation theorem (i.e. $||f_j||_{\infty} \leq 1$), then

% \begin{equation}
%     \delta(f, \mathrm{BMOA} + \mathcal{R}_n) \leq \prod\limits_{j=1}^d||f_j||_{\infty} \leq 1
% \end{equation}
% \end{corollary}

% \begin{proof}
% From the CSAAK theorem, we have $\delta(f, \mathrm{BMOA} + \mathcal{R}_n) \leq \sigma_n(\Gamma_f)$. Now $\sigma_n(\Gamma_f) \leq ||\Gamma_f||$ by definition. Finally, because $f \in L^{\infty}$, we have 

% \begin{equation}
%     ||\Gamma_f|| := \sup\frac{||\Gamma_f g||_2}{||g||_2} = \sup \frac{||(\mathbb{I} - P)(gf)||_2}{||g||_2} \leq ||f||_{\infty} \leq \prod\limits_{j=1}^d||f_j||_{\infty} \leq 1
% \end{equation}

% \end{proof}

\section{The Marginal AAK Theorem}

There are two problems in particular with the CS-AAK theorem, both of which are merely practical. First, it seems that the $\sigma-$numbers of a Hankel operator cannot easily be computed. Therefore, the practitioner may not be able to explicitly find a metric which yields errors on a performed model reduction. Second, the proof of the CS-AAK theorem is not constructive. This makes sense, considering we no longer have equality, but inequality in \eqref{eq: csaak}, so there are many rational functions achieving these bounds. 

Therefore, we prove another, albeit similar, multidimensional AAK theorem in another, more computationally tractable framework.

Denote the $d-$dimensional unit polydisk by $\mathbb{D}^d$. The following is a special case of Theorem $3.2$ of \cite{cotlar1996two}.

\begin{theorem}[Pick Interpolation Theorem]\label{thm: multipick}
 Let $z^{(1)} = (z^{(1)}_1, z^{(1)}_2, \dots z^{(1)}_d)$, $z^{(2)} = (z^{(2)}_1, z^{(2)}_2, \dots z^{(2)}_d)$, $\dots, z^{(n)} = (z^{(n)}_1, z^{(n)}_2, \dots z^{(n)}_d) \in \mathbb{D}^d$ and $\lambda_1, \lambda_2, \dots, \lambda_n \in \mathbb{C}$.
    
%     , and consider a function

% \begin{equation}
%     G(z_1, \dots, z_d) = \prod\limits_{k=1}^{d} \big( \sum\limits_{j=1}^{n}\frac{b_k(z_k)}{b'_k(z^{(k)}_j)}\frac{\sqrt[d]{\lambda_j}}{z_k -z^{(k)}_j} \big)
% \end{equation}

% In other words, $G = \prod\limits_{k=1}^{d}G_k(z_k)$ is a product of scalar-input functions which interpolates the data, where 

% \begin{equation}
%     G_k(z_k) = \sum\limits_{j=1}^{n}\frac{b_k(z_k)}{b'_k(z^{(k)}_j)}\frac{\sqrt[d]{\lambda_j}}{z_k -z^{(k)}_j}
% \end{equation}

Then there exist functions $F_1, \dots, F_d$, bounded on $\mathbb{D}^d$, analytic in $z_1, \dots, z_d$, respectively, such that 

\begin{equation}
    F_k(z^{(j)}) = \lambda_j^{\frac{1}{d}}%G_k(z_1, z_2, \dots, z_{k - 1}, z^{(j)}_k, z_{k + 1}, \dots, z_n),
\end{equation}
for all $j=1, \dots, n$, with $||F_k||_{\infty} \leq 1$ whenever the collection of matrices $\{K_k\}$, called \textbf{Pick matrices}, are nonnegative semidefinite, where

\begin{equation}\label{eq: cs_psd}
    [K_k]_{pq} = \frac{1 - \lambda_p^{\frac{1}{d}} \overline{\lambda_q^{\frac{1}{d}}}}{1 - z^{(p)}_k\overline{z^{(q)}_k}}
\end{equation}

% Let $z^{(1)} = (z^{(1)}_1, z^{(1)}_2, \dots z^{(1)}_d), z^{(2)} = (z^{(2)}_1, z^{(2)}_2, \dots z^{(2)}_d), \dots, z^{(n)} = (z^{(n)}_1, z^{(n)}_2, \dots z^{(n)}_d) \in \mathbb{D}^d$ and $\lambda_1, \lambda_2, \dots, \lambda_n \in \mathbb{C}$, and consider a function

% \begin{equation}
%     G(z_1, \dots, z_d) = \prod\limits_{k=1}^{d} \big( \sum\limits_{j=1}^{n}\frac{b_k(z_k)}{b'_k(z^{(k)}_j)}\frac{\sqrt[d]{\lambda_j}}{z_k -z^{(k)}_j} \big)
% \end{equation}

% In other words, $G = \prod\limits_{k=1}^{d}G_k(z_k)$ is a product of scalar-input functions which interpolates the data, where 

% \begin{equation}
%     G_k(z_k) = \sum\limits_{j=1}^{n}\frac{b_k(z_k)}{b'_k(z^{(k)}_j)}\frac{\sqrt[d]{\lambda_j}}{z_k -z^{(k)}_j}
% \end{equation}

% Then there exist functions $F_1, \dots, F_d$, bounded on $\mathbb{D}^d$, analytic in $z_1, \dots, z_d$, respectively, such that 

% \begin{equation}
%     F_k(z_1, z_2, \dots, z_{k - 1}, z^{(j)}_k, z_{k + 1}, \dots, z_n) = G_k(z_1, z_2, \dots, z_{k - 1}, z^{(j)}_k, z_{k + 1}, \dots, z_n),
% \end{equation}

% for all $j=1, \dots, n$, with $||F_k||_{\infty} \leq 1$ whenever the collection of matrices $\{K_k\}$ are positive definite, where

% \begin{equation}\label{eq: cs_psd}
%     [K_k]_{pq} = \frac{1 - G_k(z^{(m)}_k) \overline{G_k(z^{(n)}_k)}}{1 - z^{(p)}_k\overline{z^{(q)}_k}}
% \end{equation}
\end{theorem}

\begin{proof}
For $j = 1, 2, \dots, d$, performing one-dimensional Pick interpolation (see, e.g. \cite{garnett2007bounded}) with data

\begin{equation*}
    \Big\{(z_j^{(1)}, \lambda_1^{\frac{1}{d}}), (z_j^{(2)}, \lambda_2^{\frac{1}{d}}), \dots, (z_j^{(n)}, \lambda_n^{\frac{1}{d}})\Big\}
\end{equation*}
yields functions $F_j$ with $L^{\infty}$ norms less than one. The product $\prod_{j=1}^{d}F_j$ therefore lies in the unit ball of $L^{\infty}$.
\end{proof}

%This theorem is particularly strong, because it implies the existence of functions which interpolate at an \textit{infinite} amount of interpolation points (hold one of the nodes along one of the dimensions constant). The Cotlar-Sadosky theorem is also particularly useful, because we can use Pick's theorem to find the $F_j$'s. 
This theorem is particularly useful because we can explicitly calculate the $F_j$'s from Pick's algorithm to obtain the interpolatory product $\prod\limits_{j=1}^{N}F_j$ with $L^{\infty}$ norm less than one.

In the context of Pick interpolation, almost paradoxically, one does not lose value in an interpolation by assuming a product structure. See \cite{cotlar1996two} for details. The intuition is that Pick interpolation only assumes a result whose $L^{\infty}$ is less than one, rather than an interpolant of minimal $L^{\infty}$ norm. This subtle difference allows us enough room to assume a product structure and prove the following, marginal AAK theorem.

% In applications, the practitioner may wish to appeal to a more general space than $H^{\infty}$ for approximation, as the condition of Equation \eqref{eq: cs_psd} is often too restrictive. The following theorem makes the Cotlar-Sadosky theorem applicable in practice.

\begin{theorem}[Marginal AAK]\label{thm: werneburg}
Let $\Phi(z_1, \dots, z_d) = \prod\limits_{j=1}^{d}\phi_j(z_j)$ and let $b = b_1 \ \otimes \ b_2 \ \otimes \ \dots \ \otimes \ b_d$, where $b_j$ denotes a one-dimensional Blaschke product having at most $n_j$ factors. Then

\begin{equation}
    \delta(\Phi, BMOA + \mathcal{R}_n) \leq \sqrt{d} \prod\limits_{j=1}^{d}s_{n_j},
\end{equation}
where $s_{n_j}$ denotes the $n_j^{th}$ Schmidt number corresponding to the one-dimensional Hankel operator $\Gamma_{\phi_j}$.
\end{theorem}

\begin{proof}
We have 

\begin{multline}
    \delta(\Phi, BMOA + \mathcal{R}_n) \leq \sigma_n(\Gamma_{\Phi}) = \inf \{||\Gamma_{b\Phi}||\} \leq \sqrt{d} \inf \{||b\Phi||_{\text{BMOr}}\} \\
    \leq \sqrt{d} \inf \{||b\Phi||_{L^{\infty}}\} \leq \sqrt{d} \inf \Big\{ \prod\limits_{j=1}^{d} ||b_j\phi_j||_{L^{\infty}}\Big\} \leq \sqrt{d} \prod\limits_{j=1}^{d}s_{n_j},
\end{multline}
 where the second inequality follows from Nehari's Theorem in $\mathbb{T}^d$ \cite{cotlar1996two} and the ultimate one follows from the one-dimensional AAK theorem (see Remark $2$ in \cite{cotlar1994nehari}).
\end{proof}

The inequality in Theorem \ref{thm: werneburg} is valuable, because it shows that, by trading off the optimality included in using Theorem \ref{thm: interpolation} and the $\omega-$kernel, for the product structure of $\Phi$ above, one has access to a dimension-by-dimension AAK model reduction. 

\begin{corollary}\label{cor: prolong}
Let $\Phi(z_1, \dots, z_d) = \prod\limits_{j=1}^{d}\phi_j(z_j)$ and let $b = b_1 \ \otimes \ b_2 \ \otimes \ \dots \ \otimes \ b_d$, where $b_j$ denotes a one-dimensional Blaschke product having at most $n_j$ factors. Let $\Psi$ a the minimum norm interpolant of $\Phi$ obtained by interpolating $\Phi$ at arbitrary points $\{x_j\}$ in an RKHS which contractively contains $H^2(\mathbb{D}^d)$ with norm $||\cdot||$. Then

\begin{equation}
    ||\Psi|| \leq \prod\limits_{j=1}^{d} s_{n_j},
\end{equation}
where $s_{n_j}$ denotes the $n_j^{th}$ Schmidt number corresponding to the one-dimensional Hankel operator $\Gamma_{\phi_j}$.
\end{corollary}

\begin{proof}
In a RKHS with norm $||\cdot||$ which contractively contains $H^2(\mathbb{D}^d)$, we have $||\cdot|| \leq ||\cdot||_{H^2(\mathbb{D}^d)} \leq ||\cdot||_{L^{\infty}(\mathbb{T}^d)}$.
\end{proof}

Note that the RKHS induced by the $\omega-$kernel contractively contains $H^2(\mathbb{D}^d)$. Indeed if $f(z) = \sum\limits_{I \in (\mathbb{Z}^+)^d}a_Iz^I$ is in the Hardy space on the polydisk $H^2(\mathbb{D}^d)$, where $I = (j_1, j_2, \dots, j_d) \in (\mathbb{Z}^+)^d$ is a multi-index, then \cite{paulsen2016introduction}

\begin{equation}
    ||f||_{H^2(\mathbb{D}^d)}^2 = \sum\limits_{I \in (\mathbb{Z}^+)^d}|a_I|^2
\end{equation}

Now, from the power series expansion

\begin{equation}
    \frac{1}{1 - w^*z} = \sum\limits_{j=0}^{\infty}(w^*z)^j = \sum\limits_{I \in (\mathbb{Z}^+)^d}\frac{j!}{j_1!j_2!\dots j_d!}(w_1^*z_1)^{j_1}(w_2^*z_2)^{j_2}\dots(w_d^*z_d)^{j_d},
\end{equation}
we have

\begin{equation}
    ||f||_{\omega^d}^2 = \sum\limits_{I \in (\mathbb{Z}^+)^d}\frac{|a_I|^2}{\Big(\frac{j!}{j_1!j_2!\dots j_d!} \Big)},
\end{equation}

where $||\cdot||_{\omega^d}$ denotes the norm induced by the $\omega-$kernel \cite{alpay2002schur}. Therefore, 

\begin{equation}
    ||\Psi||_{\omega^d} \leq ||\Phi||_{\omega^d} \leq ||\Phi||_{H^2(\mathbb{D}^d)} \leq ||\Phi||_{L^{\infty}} \leq \prod\limits_{j=1}^{d}||\phi||_{L^{\infty}} \leq \prod\limits_{j=1}^{d}s_{n_j}.
\end{equation}

Also note the possibility that performing Nevanlinna-Pick in dimension-by-dimension fashion may fail, because one or more of the Pick matrices in Theorem \ref{thm: multipick} may not be nonnegative semidefinite (note this is quite different from the nonnegative semidefiniteness one always gets by performing RKHS interpolation). In this case, one can still choose a marginal way of performing interpolation which results in a function in BMOr (e.g. any RKHS contractively contained in $H^2(\mathbb{D}^d)$). Indeed, it is quite often the case that the Pick matrices given in Theorem \ref{thm: multipick} are not all nonnegative semidefinite, so such methods may be more applicable in practice.

As an example of how marginal AAK works, consider the example of learning the multivariate polynomial $f(x, y) = xy^2$ from $100$ noisy data points (Figure \ref{fig: mAAK}).

We gave our neural networks $100$ samples drawn from the surface given by $f$, adding Gaussian noise drawn from $N(0, 0.1)$. Interpolation via the Segal-Bargmann kernel, gives results which generalize quite poorly. Recall that Segal-Bargmann interpolation does not generalize well in the presence of many interpolation nodes, due to Jensen's theorem. On the other hand, the power series expansion of the exponential function allows one to perform marginal AAK, which results in a neural network which generalizes far better.

\begin{figure}[H]
    \centering
    \includegraphics[scale=.4]{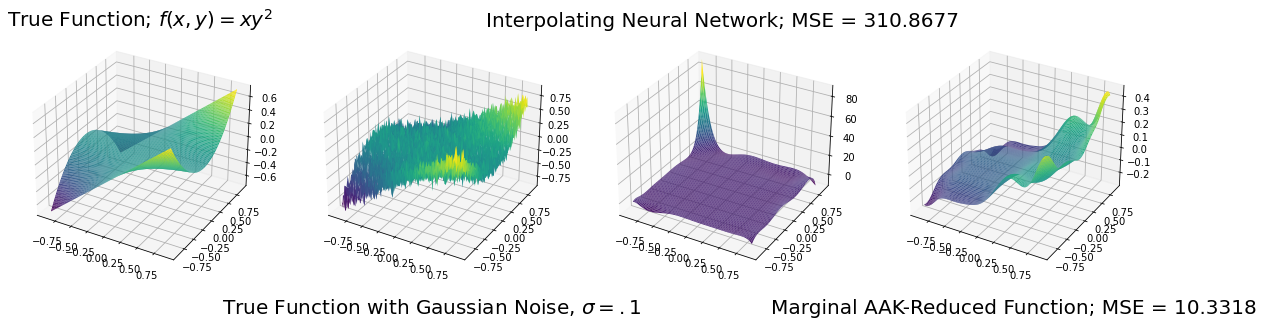}
    \caption{Learning a noisy surface using RKHS interpolation and a marginal AAK-reduced interpolant.}
    \label{fig: mAAK}
\end{figure}

The AAK model reduction outlined above works particularly well in the presence of noise. The reasons for this are twofold. First, interpolation works notoriously poorly on noisy data. Model reduction in this case removes the interpolation requirement of our approximating neural network, thereby allowing for better generalization. Second, if it turns out that $\delta(\Phi, BMOA + \mathcal{R}_n)$ is not particularly small, model reduction may still be valuable if one does not expect to achieve accuracy on the order of $O(\delta(\Phi, BMOA + \mathcal{R}_n))$ regardless. The reduction of noise in such a case is still of value.

\subsection{Prolongation Neural Networks}

We emphasize that the model-reduced function $f$ given by the marginal AAK theorem is generally not a neural network. Regardless, Corollary \ref{cor: prolong} gives an explicit method of constructing a neural network which approximates $f$ to arbitrary accuracy. This is because one can now \textit{sample} from $f$ to gain training data, a privilege rarely available in the context of training neural networks. From here the practitioner has the freedom to perform a variety of different methods to construct a neural network, e.g. Remez exchange or cross validation. We call a neural network which results from such a training method a \textbf{Prolongation Neural Network}, or \textbf{PNN}. We call $f$ a \textbf{prolongation}.

\begin{figure}[H]
    \centering
    \includegraphics[scale=.4]{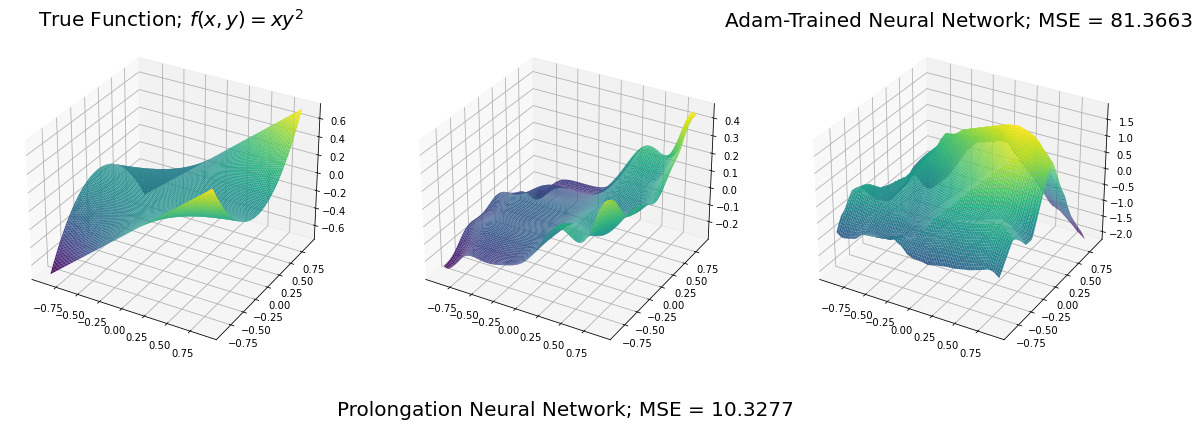}
    \caption{Learning a noisy surface using a PNN, in comparison with Adam.}
    \label{fig: mAAKnn}
\end{figure}

Figure \ref{fig: mAAKnn} is an illustrative example. The PNN depicted uses the prolongation given in Figure \ref{fig: mAAK} and $625$ training points given by a $(25 \times 25)$ grid of Chebyshev nodes. We use the activation induced by the Segal-Bargmann kernel. The rightmost surface is a $2-$layer neural network trained via Adam, with width $100$ (in correspondence with the fact that $100$ noisy training points were seen by the prolongation). The neural network was trained for $1,000$ epochs, using an initial learning rate of $5 \times 10^{-5}$, a weight decay of $10^{-3}$, and ReLU activation. We used mean square error as our loss function. The results suggest that the PNN is superior to Adam, as the mean square error is about $8$ times lower.

It should be remarked that the process of artificially generating data from which one can train a neural network is similar to the use of Generative Adversarial Networks (GAN) \cite{goodfellow2014generative}, except the \textit{generative} portion of a PNN is not a neural network and does not iteratively get trained.

\subsection{Fornasini-Marchesini Model Reduction}

\begin{theorem}[Fornasini-Marchesini \cite{gleason}]\label{thm: gleason}
If $F: \mathbb{C}^d \rightarrow \mathbb{C}^{p \times q}$ is  rational in each of its variables and analytic in a neighborhood of the origin, then $F$ can be written 

\begin{equation}\label{eq: gleason}
    F(z_1, z_2, \dots, z_d) = D + C\Bigg(\mathbb{I}_n - \sum\limits_{j=1}^{d}z_jA_j\Bigg)^{-1}\Bigg(\sum\limits_{k=1}^{d}z_kB_k\Bigg),
\end{equation}
where $D = F(0)$, $C \in \mathbb{C}^{p \times n}$, $A_1, A_2, \dots, A_d \in \mathbb{C}^{n \times n}$, and $B_1, B_2, \dots, B_d \in \mathbb{C}^{n \times q}$.
\end{theorem}

\begin{proof}
See \cite{gleason} for a short and accessible proof.
\end{proof}

Given a function $F$ of the form found in Theorem \ref{thm: gleason}, finding an $A, B, C, D$ such that \eqref{eq: gleason} holds is known as the \textit{Fornasini-Marchesini} problem. Such a tuple $(A, B, C, D)$ is a \textit{Fornasini-Marchesini realization}, or \textit{F-M realization}. While there are several methods of finding F-M realizations, there has been no known way of constructing a \textit{minimal realization} (i.e. one with no pole-zero cancellation) \cite{fm}, \cite{fmeig}. %We plan to do so here, but first we state some classical results \cite{gleason}.

By Theorem \ref{thm: gleason}, we know that functions with F-M realizations are subsets of $\mathcal{H}_{\omega}$, the RKHS induced by the $\omega-$kernel, on some ball with appropriate radius. To see this, simply note that if $f$ has an F-M realization, then there must exist some ball about zero on which $f$ is bounded (in fact, boundedness is more strict than is necessary, but will do here). Therefore, given some F-M realization $f$, one can perform interpolation in $\mathcal{H}_{\omega}$ (or some scaled version of it) on as many interpolation nodes as is necessary to gain some desired approximation. Finally, performing marginal AAK to the interpolatory function, one can use the Schmidt numbers to reduce the order of the model until it is of minimal order.

While such a function does not explicitly yield its F-M realization, there does exist one by Theorem \ref{thm: gleason}. Moreover, because the lack of pole-zero cancellation is desirable for numerical purposes, the procedure outlined above may be useful in the field of multidimensional systems.

\chapter{Applications}\label{ch: apps}

In \cite{werneburgmachine}, the authors use our RKKS interpolatory training methods to construct neural networks which predict the outcomes of different surgical methods. Specifically, the neural networks outperform professional urologists in predicting various improvements in urinary incontinence resulting from different surgeries. See Figure \ref{fig: urology} for some results from the paper. Here, we give a few more applications of our training methods which may be of independent interest.

\begin{figure}[H]
    \centering
    \includegraphics[scale=.3]{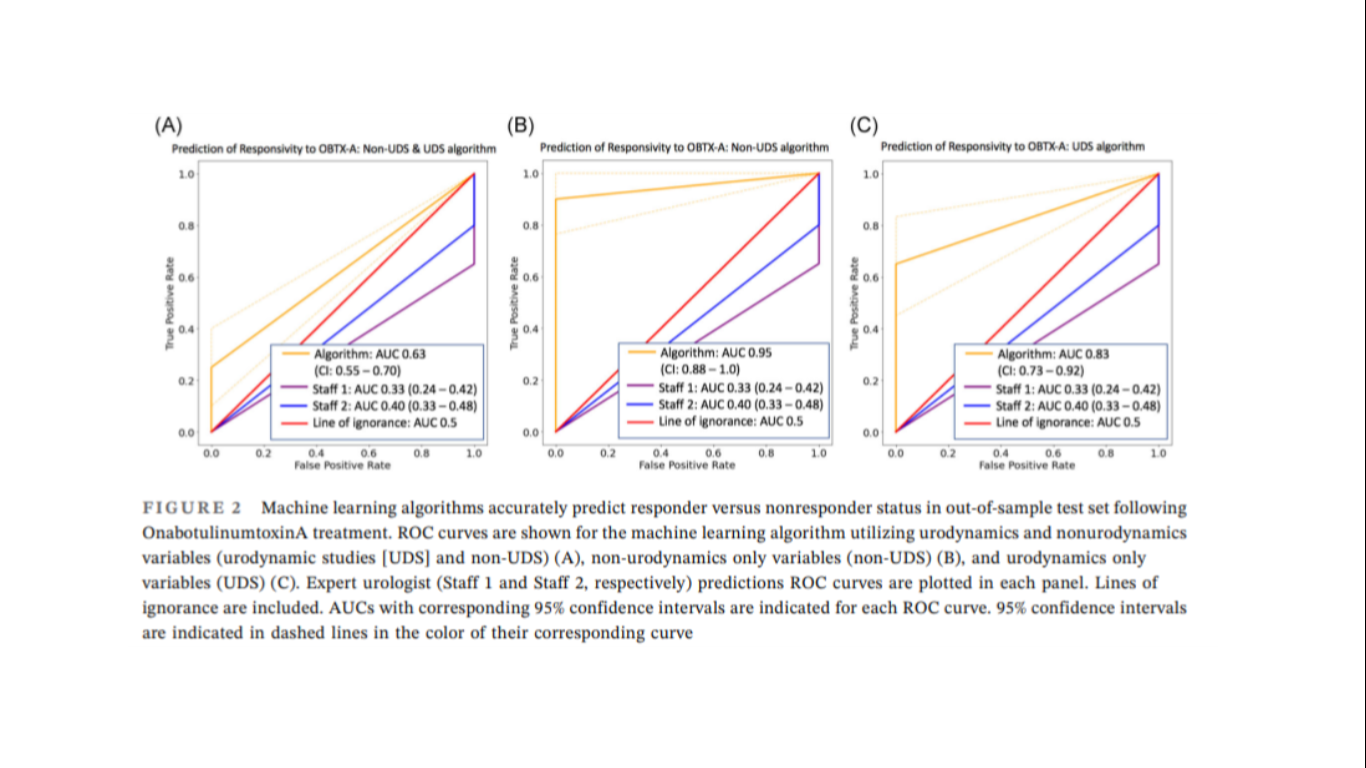}
    \caption{Taken from \cite{werneburgmachine}}.
    \label{fig: urology}
\end{figure}

\section{Calculating the Permanent of a Matrix}
The permanent of a matrix $A$ is defined by the Leibniz formula: 

\begin{equation}
    \textrm{perm}(A) = \sum\limits_{\sigma \in S_n}\prod\limits_{j=1}^n a_{j, \sigma(j)},
\end{equation}
and only differs from the determinant by $\textrm{sgn}(\sigma)$ in every term. However, it is believed to be more difficult to compute than the determinant, as it is $\mathbf{\#P}-$complete, whereas the determinant is in $\mathbf{P}$. Note $\mathbf{\#P} = \mathbf{P}$ is a stronger statement than $\mathbf{NP} = \mathbf{P}$.

We used our interpolation algorithm on $3 \times 3$ matrices. Results are shown in Figure \ref{fig: permanent}.

\begin{figure}[H]
    \centering
    \includegraphics[scale=.3]{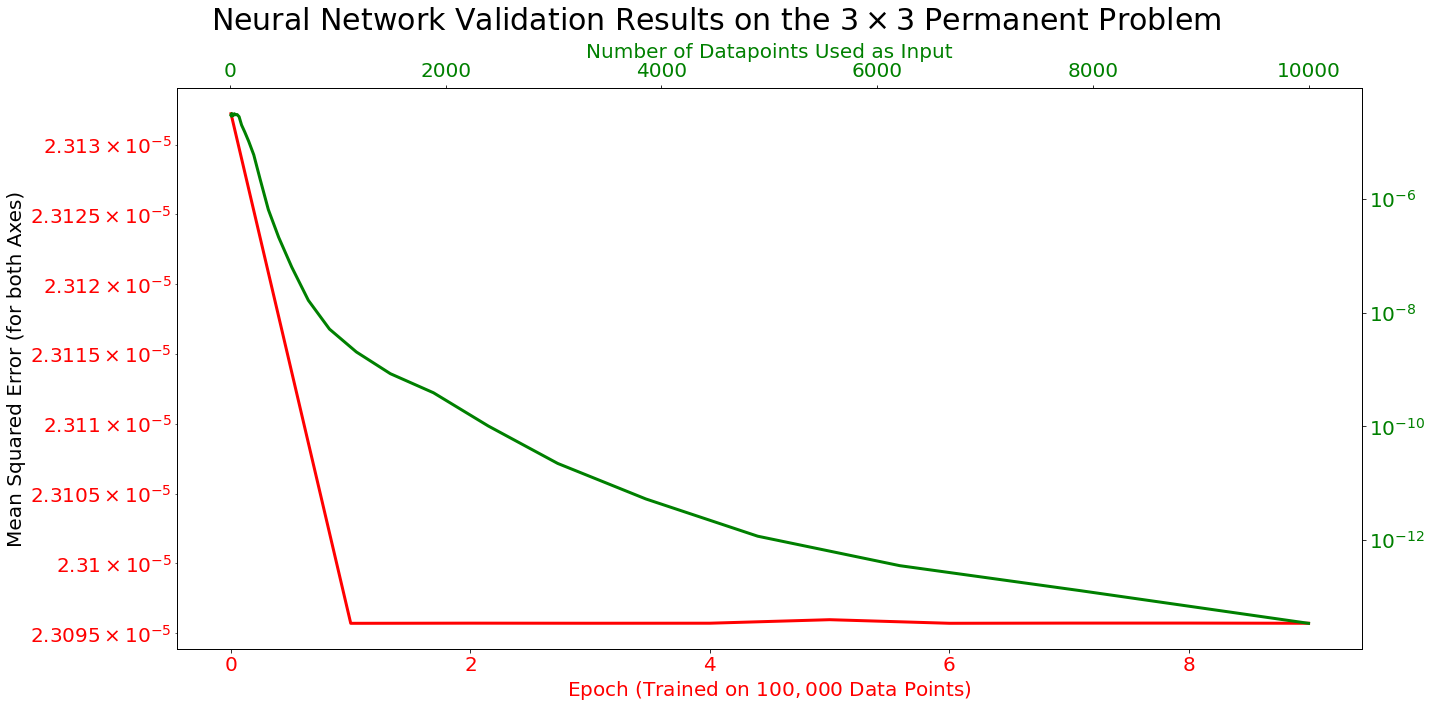}
    \caption{Error on calculating the permanent of a $3 \times 3$ matrix using the conventional method (Adam) and our interpolation method ($\omega-$kernel).}
    \label{fig: permanent}
\end{figure}

% \begin{figure}[H]
%     \centering
%     \includegraphics[scale=.3]{3x3Permanent.png}
%     \caption{Error on calculating the Permanent of a $3 \times 3$ matrix using conventional methods and our interpolation method.}
%     \label{fig:permanent}
% \end{figure}

% \begin{figure}[H]
%     \centering
%     \includegraphics[scale=.3]{AdamPermanent.png}
%     \caption{Error on calculating the permanent of a $3 \times 3$ matrix using conventional methods on many more examples.}
%     \label{fig:adampermanent}
% \end{figure}

We generate matrices from the Gaussian Unitary Ensemble (GUE), and scale them so that our activation function induced by the $\omega-$kernel does not blow up. Note that this scaling is not a restriction for this problem because it is trivial to compute the permanent of any scalar multiple of a matrix whose permanent is already known. We then train two neural networks on the same dataset: one via interpolation using the $\omega-$kernel and one via Adam and the piecewise linear activation ReLU (using activation $f(z) = \frac{1}{1 - z}$ and training via Adam resulted in far inferior accuracy). Recall Adam is a sophisticated form of gradient descent which uses a memory of the first and second moments of its past gradient steps. For Adam, we use a depth $2$, width $3000$ neural network, an initial learning rate $10^{-4}$ with adaptive step size, a weight decay parameter of value $10^{-3}$, and batches of size $64$. We use mean squared error as our loss function.

% There are two things to note about Figure \ref{fig:permanent}. First, the error from using our method starts off significantly lower than that using Adam. This is because learning even one training example already gives us at least a vague idea of the magnitude of the training examples. A gradient update will not capitalize on this information as quickly. 

% Second, the slope of the error from using our interpolation method is far steeper than that using Adam. While the red curve has a negative slope, the error is decreasing significantly slower.

Figure \ref{fig: permanent} suggests that our interpolation method is superior to Adam on the permanent problem. Adam, even with its adaptive step-size converges far too early and is about eight orders of magnitude less accurate than the interpolation network here.

% Finally, Figure \ref{fig:adampermanent} depicts the in-sample error through $100,000$ training examples. Adam will need many more training examples to catch up to an interpolation-trained neural network, which only looked at $10,000$.

It is interesting that, even though calculating the determinant of a matrix is in $\mathbf{P}$, the results given by training our neural network are very similar, as shown in Figure \ref{fig: determinant}. This is because, from the pure, neural network standpoint, one is attempting to learn a cubic polynomial in nine variables. Regardless of the coefficients of these polynomials, the tasks are of similar difficulty.

\begin{figure}[H]
    \centering
    \includegraphics[scale=.3]{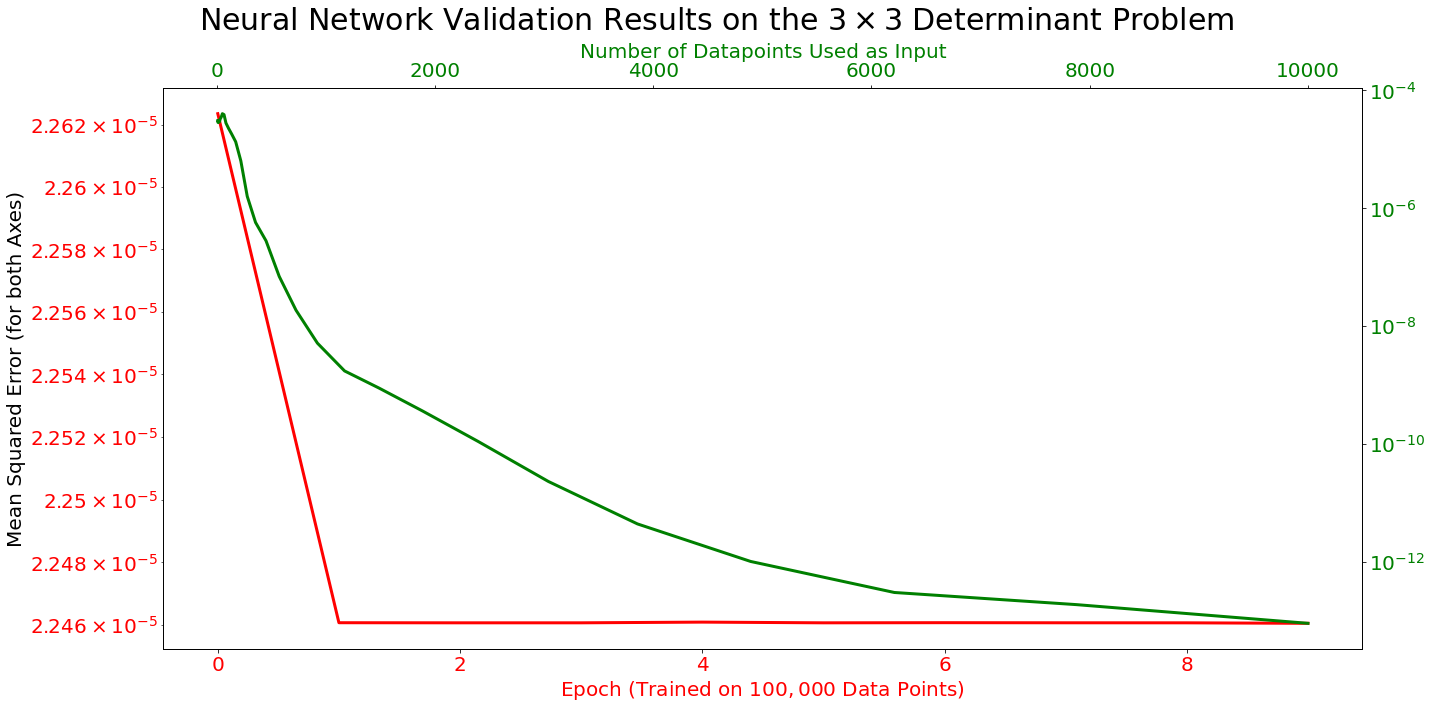}
    \caption{Error on calculating the determinant of a $3 \times 3$ matrix using the conventional method (Adam) and our interpolation method ($\omega-$kernel).}
    \label{fig: determinant}
\end{figure}

% \begin{figure}[H]
%     \centering
%     \includegraphics[scale=.3]{3x3Determinant.png}
%     \caption{Error on calculating the determinant of a $3 \times 3$ matrix using conventional methods and our interpolation method.}
%     \label{fig:deteriminant}
% \end{figure}

% \begin{figure}[H]
%     \centering
%     \includegraphics[scale=.3]{AdamDeterminant.png}
%     \caption{Error on calculating the determinant of a $3 \times 3$ matrix using conventional methods on many more examples.}
%     \label{fig:adamdeterminant}
% \end{figure}

The reason that Adam is suffering so much here is because there is a large discrepancy between forward and backward error on these problems. In particular, calculating the permanents of matrices generated from the GUE comes with no guarantee of being a well-conditioned problem. Given that the determinant of a matrix depends on its eigenvalues, and calculating the permanent is at least as hard as calculating the determinant, Adam should have a great deal of trouble on these problems. The results shown confirm this.

\section{Complete Elliptic Integrals of the Second Kind}
In \cite{hemati2019ellipse}, a neural network was trained to calculate complete elliptic integrals of the second kind up to a mean square error (divided by $2$) of $0.000769$. Here, we train a $2-$layer neural network with activation $f(z) = \frac{1}{1 - z}$ on $30$ Chebyshev nodes, corresponding to the eccentricities of complete elliptic integrals. The corresponding mean square error we receive is $8.83 \times 10^{-12}$, or about $5$ more digits of accuracy. The results of given in Figures \ref{fig: ellipticintegral} and \ref{fig: ellipticintegral_error}.

\begin{figure}[H]
\begin{subfigure}{.5\textwidth}
  \centering
  \includegraphics[width=.8\linewidth]{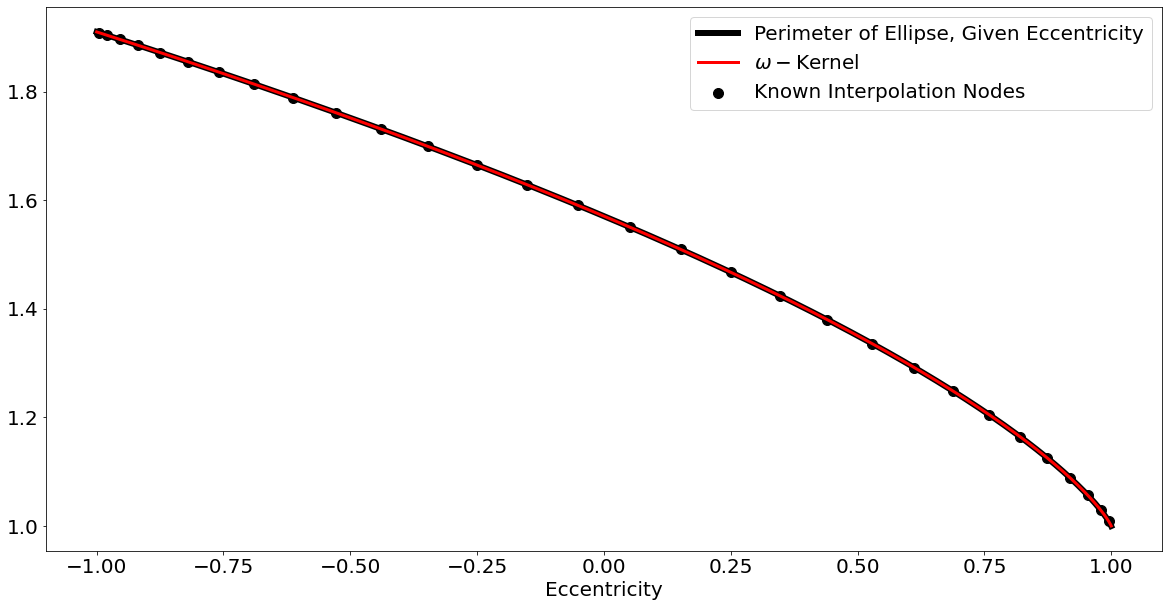}
  \caption{Neural Network Outputs and True Integral Values}
  \label{fig: ellipticintegral}
\end{subfigure}%
\begin{subfigure}{.5\textwidth}
  \centering
  \includegraphics[width=.8\linewidth]{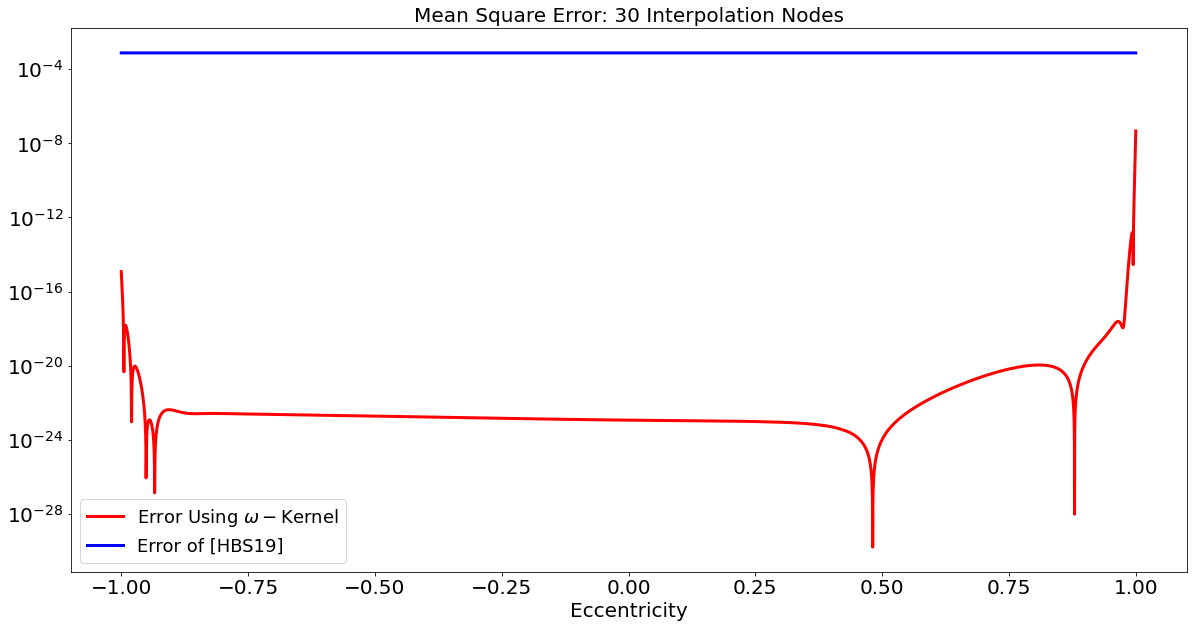}
  \caption{Mean Square Error}
  \label{fig: ellipticintegral_error}
\end{subfigure}
\caption{Our neural network obtains a mean square error of $8.83 \times 10^{-12}$, although its error is typically much lower (about $10^{-22}$).}
\label{fig: elliptic}
\end{figure}

\section{Learning Shrinkage}

We show that it is possible for a neural network how to learn shrinkage. In fact, for large samples, we find that it is possible to surpass the James-Stein estimator in terms of accuracy. 

Given one sample from a multivariate normal distribution, we train a neural network to calculate the mean.  We do so by first generating a random walk of length $N$ in $d$ dimensions, where the $j^{th}$ step is taken from a $d-$dimensional Gaussian distribution, with standard deviation $1.05^j$. These samples will be the \textit{means} of the normal distributions from which we shall sample one observation (this distribution will have standard deviation $1$). In the end, we will have $N$, $d-$dimensional samples, each coming from a distribution with an increasing mean on average. The first $N_t$ will be our training data and the last $N_v$ will be our validation data ($N = N_t + N_v$). Note the validation data comes from normal distributions with potentially much higher means than the training data. This removes the possibility of training a neural network to learn a Bayesian prior, rather than shrinkage. Moreover, because Brownian motion in dimension greater than $2$ is transient, any Bayesian prior learned from the training data will hurt out-of-sample performance.

Results for $N_t = 5, N_v = 43$ are shown in Figure \ref{fig:shrinkage_nn_walk}. Here, for each $d$ along the $x-$axis, we ran $10$ different experiments and averaged them to gain a picture of the mean performance of our training. Note that the variance of our estimator is smaller than that of both maximum likelihood and the James-Stein estimator.

\begin{figure}[H]
    \centering
    \includegraphics[scale=.3]{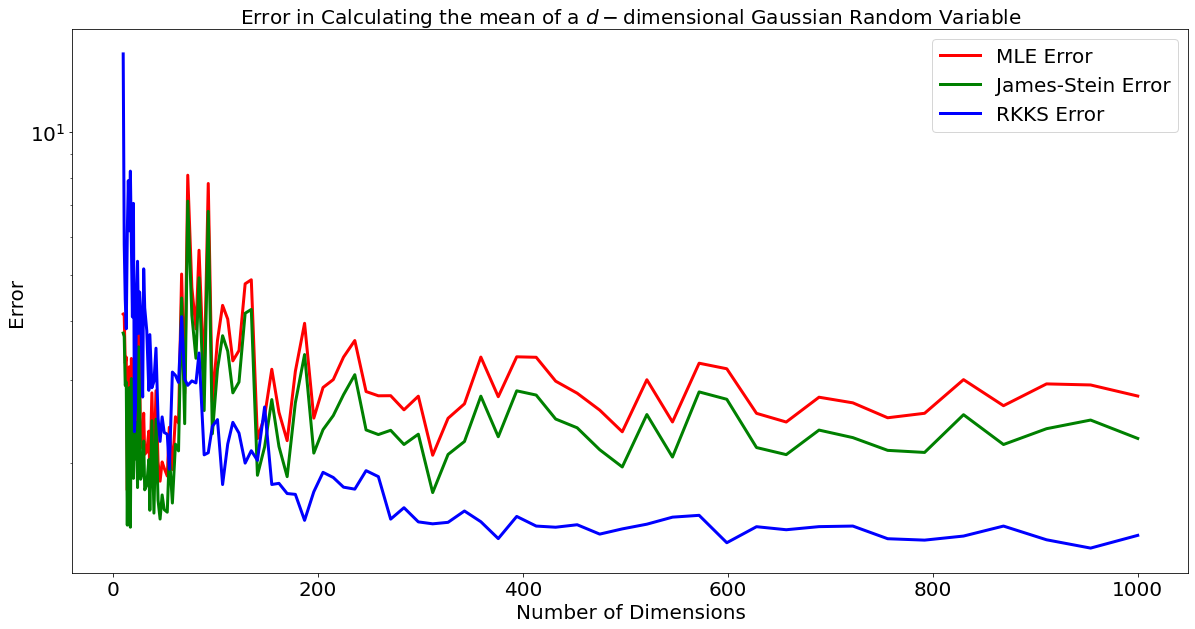}
    \caption{Relative, absolute error for mean estimation. Here, $N_t = 5$ and $N_v = 43$. Note $1.05^{48} > 10$, so the variance of the last step taken is greater than $100$. The plot shown here is the average of $10$ different experiments.}
    \label{fig:shrinkage_nn_walk}
\end{figure}

\endgroup

\begingroup\SingleSpacing
%\nocite{*}
\printbibliography
\endgroup

\end{document}